\documentclass[preprint]{imsart}

\RequirePackage{amsthm,amsmath,amsfonts,amssymb}
\RequirePackage[numbers,sort]{natbib}
\RequirePackage[colorlinks,citecolor=blue,urlcolor=blue]{hyperref}
\RequirePackage{graphicx}
\RequirePackage{cleveref} 
\RequirePackage{xcolor} 

\usepackage[textwidth=34pc,
textheight=574pt,
marginparwidth=0pt,
hoffset=2pc]{geometry}

\startlocaldefs
\theoremstyle{plain}
\newtheorem{theorem}{Theorem}[section]
\newtheorem{lemma}{Lemma}[section]
\newtheorem{corollary}{Corollary}[section]
\newtheorem{proposition}{Proposition}[section]
\theoremstyle{remark}

\newtheorem{remark}[theorem]{Remark}

\crefname{enumi}{condition}{conditions}
\Crefname{enumi}{Condition}{Conditions}

\newcommand{\bbR}{\mathbb{R}}
\newcommand{\bbC}{\mathbb{C}}
\newcommand{\calD}{\mathcal{D}}
\newcommand{\ntrain}{n_{\text{train}}}
\newcommand{\ntest}{n_{\text{test}}}
\newcommand{\supop}{\mathcal{S}}
\newcommand{\Hol}{\mathrm{Hol}}
\newcommand{\calM}{\mathcal{M}}
\newcommand{\calA}{\mathcal{A}}
\newcommand{\bbH}{\mathbb{H}}
\newcommand{\bbD}{\mathbb{D}}
\newcommand{\calY}{\mathcal{Y}}
\newcommand{\bbE}{\mathbb{E}}
\newcommand{\tr}{\mathrm{tr}}
\newcommand{\bbP}{\mathbb{P}}
\newcommand{\bbN}{\mathbb{N}}
\newcommand{\st}{\,:\,}
\newcommand{\diag}{\mathrm{diag}}
\renewcommand{\d}{\mathrm{d}}
\newcommand{\cov}{\mathcal{C}}
\newcommand{\ha}{\hat{a}}
\newcommand{\hA}{\hat{A}}
\newcommand{\F}{\mathcal{F}}
\newcommand{\Ftau}{\mathcal{F}^{(\tau)}}
\newcommand{\T}{\mathcal{T}}
\newcommand{\net}{\mathcal{N}}
\newcommand{\Lsub}{L^{(\mathrm{sub})}}
\newcommand{\Fsub}{\mathcal{F}^{(\mathrm{sub})}}
\newcommand{\Msub}{M^{(\mathrm{sub})}}
\newcommand{\calS}{\mathcal{S}}
\newcommand{\hy}{\hat{y}}
\newcommand{\Etest}{E_{\text{test}}}
\newcommand{\stab}{\mathcal{L}}

\newcommand{\supp}{\mathrm{supp}}
\newcommand{\D}{\mathrm{D}}
\newcommand{\calN}{\mathcal{N}}
\newcommand{\hX}{\hat{X}}
\newcommand{\wridge}{w_{\text{ridge}}}
\newcommand{\G}{\mathcal{G}}
\newcommand{\ubar}{\underline{u}}
\newcommand{\calL}{\mathcal{L}}
\newcommand{\calX}{\mathcal{X}}
\newcommand{\hx}{\hat{x}}

\crefname{assumption}{assumption}{assumptions}
\Crefname{assumption}{Assumption}{Assumptions}

\endlocaldefs

\begin{document}

\begin{frontmatter}
\title{Dyson Equation for Correlated Linearizations and Test Error of Random Features Regression}
\runauthor{Latourelle-Vigeant and Paquette}
\runtitle{Dyson Equation for Correlated Linearizations and Test Errors}

\begin{aug}
\author[A,t1]{\fnms{Hugo}~\snm{Latourelle-Vigeant}\ead[label=e1]{hugo.latourelle-vigeant@yale.edu}},
\and
\author[B]{\fnms{Elliot}~\snm{Paquette}\ead[label=e2]{elliot.paquette@mcgill.ca}}
\address[A]{Department of Statistics and Data Science, Yale University\printead[presep={,\ }]{e1}}

\address[B]{Department of Mathematics and Statistics, McGill University\printead[presep={,\ }]{e2}}
\end{aug}

\thankstext{t1}{This work was conducted while HLV was affiliated with McGill University.}

\begin{abstract}
    This paper develops some theory of the Dyson equation for correlated linearizations and uses it to solve a problem on asymptotic deterministic equivalent for the test error in random features regression. The theory developed for the correlated Dyson equation includes existence-uniqueness, spectral support bounds, and stability properties. This theory is new for constructing deterministic equivalents for pseudo-resolvents of a class of linearizations with correlated entries. In the application, this theory is used to give a deterministic equivalent of the test error in random features ridge regression, in a proportional scaling regime, wherein we have conditioned on both training and test datasets.
\end{abstract}



\end{frontmatter}


\section{Introduction}

Contemporary artificial neural networks have found widespread applications in diverse domains. A notable trend in modern neural network design is the increasing size and complexity of these models. In practical applications, there is a prevalent use of highly overparametrized models. These overparametrized models exhibit exceptional capacity, showcasing the ability to perfectly fit the training data, even in scenarios where the labels are pure noise and nonetheless generalize well~\cite{zhang_generalization}. Despite the remarkable practical success of neural networks, a considerable gap exists between theoretical understanding and real-world performance in machine learning. Neural networks pose unique challenges for analysis due to two key factors. Firstly, the high dimensionality of these models often leads to behaviors that defy conventional statistical knowledge. Secondly, the presence of non-linear activation functions, which are known to enhance the expressive capacity of neural networks, further complicates analysis.

In recent years, random matrix theory has been used to provide valuable insights into the behavior of neural networks and other machine learning models. A notable line of research at the confluence of random matrix theory and neural networks combines linearizations with operator-valued free probability to analyze simple neural networks~\cite{mel2022anisotropic,Adlam2019ARM,tripledescent,adlam_doubledescent,tripuraneni2021overparameterization}. This approach has been particularly fruitful in obtaining asymptotic characterizations of the training and test errors under various conditions.

Beyond the linear model, the random features model introduced in~\cite{Rahimi_RF} stands out as one of the simplest models with significant expressive power. In contrast to the linear model, the random features model incorporates a non-linear activation function and can be overparametrized. Despite its simplicity, the random features model provides a mathematically tractable framework that proves instrumental in studying phenomena observed in real-life machine learning models, such as multiple descent~\cite{montanari_RF,adlam_doubledescent,double_descent} and implicit generalization~\cite{chouard2022quantitative,jacot_RF}. Extensive studies have delved into the training and test error of random features ridge regression in high dimensions~\cite{pmlr-v119-gerace20a, montanari_RF, Adlam2019ARM, adlam_doubledescent, hastie_22, mel2022anisotropic, MEI20223,schröder2024asymptotics,tripuraneni2021overparameterization, hu2024asymptotics, wang2023overparameterized}. Notably, the non-linear random features model is connected to a simpler linear Gaussian model through a universality phenomenon~\cite{montanari_RF, pmlr-v119-gerace20a, goldt2022gaussian, hong_lu_GET, schroder2023deterministic, hu2024asymptotics}.

The concept of linearization---which also goes under the name ``realization'' in the control community~\cite{HELTON20181} and ``linear pencil method'' in operator theory~\cite{mel2022anisotropic,Adlam2019ARM,tripledescent,adlam_doubledescent,tripuraneni2021overparameterization}---entails representing rational functions of random matrices as blocks of inverses of larger random matrices which depend linearly on its random matrix inputs. These linearizations possess simpler correlation structures, rendering them more amenable to certain types of analysis. Beyond their application in analyzing the training and test error of neural networks, linearizations have been used to study polynomials of random matrices on the global scale~\cite{Anderson_2013, BelinschiMaiSpeicher+2017+21+53, haagerup_new_2005, HELTON20181, HELTON2006105} and the local scale~\cite{nemish_local_2020, fronk2023norm, anderson_anticommutator}. This exploration naturally leads to the study of pseudo-resolvents, or generalized resolvents. Pseudo-resolvents are inherently more challenging to study than resolvents due to the absence of a spectral parameter spanning the entire diagonal. Nonetheless, one effective approach to study pseudo-resolvent involves analyzing a fixed-point equation known as the Dyson equation for linearizations (DEL)~\cite{Anderson_2013,anderson_anticommutator,nemish_local_2020,fronk2023norm}.

\subsection{Main Contributions}
In this work, we introduce an extension of the Dyson equation for linearizations (DEL), tailored specifically for linearizations with correlated blocks.  The DEL gives a systematic approach to constructing deterministic equivalents for functionals of random matrices, where the underlying random matrix should satisfy a \emph{Gaussian equivalence principle}, which is to say the spectral properties of the pseudoresolvent are determined solely by the correlation structure of the linearization.

One of the main contributions of this work is to show the existence and uniqueness of the deterministic equivalent constructed from the DEL in \Cref{theorem:main_properties}. This extends existing theory on Dyson equations which focuses on the cases of resolvents of correlated matrices and on pseudoresolvents of block matrices whose blocks are i.i.d. (or similarly mean-field); see \Cref{sec:related_work} for further discussion on how this relates to existing DEL theory.

Furthermore, we give an approach to showing the stability of the DEL, which we implement for the empirical test error of random features ridge regression. We note that there is no known simple test for stability of linearizations, even in the case of i.i.d. blocks (see the discussion in \cite{nemish_local_2020}), and hence this must be checked for each problem to which this method is applied. Within this framework, we derive an anisotropic global law for a class of pseudo-resolvents with general correlation structures.



We then apply this general theory to derive an asymptotically exact representation for the empirical test error of random features ridge regression, conditioning on both the training and test datasets. Specifically, we confirm~\cite[Conjecture 1]{louart_random-matrix-approach} under the additional assumption that the norms of the training and test random features matrices are bounded in expectation. As a consequence, we establish a general Gaussian equivalence principle for the empirical test error of random feature ridge regression. This Gaussian equivalence holds regardless of the choice of activation functions and labels, and we provide empirical evidence demonstrating that it accurately characterizes the training and test errors in real-world datasets.


\subsection{Notation}
We adhere to the following conventions throughout this paper. Lowercase letters (e.g., \(v\)) represent real or complex scalars or vectors, while uppercase letters (e.g., \(M\)) denote real or complex matrices. Calligraphic letters are utilized for denoting sets (e.g., \(\calM\)) or operators on the space of complex matrices (e.g., \(\supop\)). The symbol \(\bbH\) refers to the upper half of the complex plane. We use \(\bbR_{\geq 0}\) to represent the set of non-negative real numbers and \(\bbR_{>0}\) to represent the set of positive real numbers. \(\bbD\) denotes the open unit disk in the complex plane, and \(\bbC\) represents the set of complex numbers.

When considering a matrix \(M\in \bbC^{\ell\times \ell}\) or an operator on the space of \(\ell\times \ell\) matrices, we often employ the blockwise representation:
\[ M =
\begin{bmatrix}
    M_{1,1} & M_{1,2} \\
    M_{2,1} & M_{2,2}
\end{bmatrix}.
\]
Here, \(M_{1,1}\in \bbC^{n\times n}\), \(M_{1,2}\in \bbC^{n\times d}\), \(M_{2,1}\in \bbC^{d\times n}\), and \(M_{2,2}\in \bbC^{d\times d}\).

Given a complex matrix \(M\), we denote its real transpose by \(M^{\top}\) and the conjugate transpose by \(M^{\ast}\). When dealing with matrix sub-blocks, \(M_{i,j}^{\ast}=(M_{i,j})^{\ast}\) denotes the conjugate transpose of the \((i,j)\) sub-block of \(M\). If \(\calD\) is a domain in a complex normed vector space \(\calX\) and \(\calY_0\) is a subset of a complex normed vector space \(\calY\), we denote by \(\Hol(\calD,\calY_0)\) the set of holomorphic functions \(f:\calD\to\calY\) with \(f(\calD)\subseteq \calY_0\).

We use \(\|\cdot\|\) to represent the standard Euclidean norm when applied to a vector and the spectral norm when applied to either matrices or complex matrix-valued functions. Additionally, we employ \(\|\cdot\|_{F}\) to denote the Frobenius norm and \(\|\cdot\|_{\ast}\) for the nuclear norm. The trace of a matrix is denoted by \(\tr\).

If \(n,d\to\infty\), we write \(d\asymp n\) to denote \(\lim_{n,d\to\infty}\frac{n}{d}=c\in \bbR_{>0}\) and \(d\lesssim n\) if \(d\leq cn\) for some constant \(c\in \bbR_{>0}\).

\subsection{Organization}

\Cref{sec:main_results} introduces the main results and settings of this work. In particular, \Cref{sec:main_DEL} establishes the existence and uniqueness of a solution to the Dyson equation for correlated linearizations, alongside some key properties of the solution, such as a Stieltjes transform representation, under broad assumptions. \Cref{sec:assumptions,sec:main_anisotropic} then leverages this solution to derive conditions under which an anisotropic global law holds for the pseudo-resolvent of a linearization, offering a sketch of the proof to clarify and motivate these conditions. In \Cref{sec:main_application}, we apply these findings to a random features ridge regression model and present a deterministic equivalent for its empirical test error, alongside a discussion of the theorem and its implications. We position our work within the existing literature in \Cref{sec:related_work}. The proofs of the main results are provided in \Cref{sec:properties,sec:proof_global_anisotropic,sec:proof_rf}, respectively. \Cref{sec:conclusion} concludes the paper. Finally, we present preliminary results in \ref{sec:prelim} and the proof of some intermediate lemmas in \ref{sec:tecnical_lemmas}.

\section{Main Results}\label{sec:main_results}

This section presents the main results of this work, organized into four parts: First, we establish the main properties of the Dyson equation for correlated linearizations. Second, we outline the necessary assumptions for deriving the anisotropic global law. Third, we state the anisotropic global law itself. Finally, we apply these results to a random features ridge regression model to analyze the test error.

\subsection{Dyson Equation for Correlated Linearizations}\label{sec:main_DEL}

Consider a class of real self-adjoint linearizations of the form
\begin{equation}\label{eq:linearization}
    L =
    \begin{bmatrix}
        A & B^{\top} \\
        B & Q
    \end{bmatrix}
    \in \bbR^{\ell\times \ell}.
\end{equation}
Here, \(A\in \bbR^{n\times n}\) is self-adjoint, \(Q\in \bbR^{d\times d}\) is self-adjoint with both \(Q\) and \(\bbE Q\) non-singular, and \(B\in \bbR^{d\times n}\) is arbitrary. Our primary focus is analyzing the high-dimensional behavior of the pseudo-resolvent \((L-z\Lambda)^{-1}\), where \(\Lambda:=\diag\{ I_n,0_{d\times d}\}\) and \(z\in\bbH:=\{z\in \bbC \st \Im[z]>0\}\) represents the upper half of the complex plane. Our framework relies on the \emph{Dyson equation for linearizations (DEL)}
\begin{equation}\label{eq:DEL}
    (\bbE L - \supop(M) - z\Lambda) M = I_{\ell},
\end{equation}
where the spectral parameter \(z\in\bbH\) and the \emph{superoperator}
\begin{equation}\label{eq:supop}
    \supop:M\in \bbC^{\ell\times\ell}\mapsto
    \begin{bmatrix}
        \supop_{1,1}(M) & \bbE[(A-\bbE A)M_{1,1}(B-\bbE B)^{\top}] \\
        \bbE[(B-\bbE B)M_{1,1}(A-\bbE A)] & \bbE[(B-\bbE B)M_{1,1}(B-\bbE B)^{\top}]
    \end{bmatrix} \in \bbC^{\ell\times \ell}
\end{equation}
is a positivity-preserving linear map encoding the correlation structure of the linearization. Here, by positivity-preserving, we mean that \(\supop\) maps positive semidefinite matrices to positive semidefinite matrices. The sub-superoperator \(\supop_{1,1}\) also preserves positivity. Importantly, the expectation in the superoperator is taken only over the linearization, and not over the superoperator's argument.

Concretely, the more natural superoperator to consider would be the covariance-induced one
\[
   M \mapsto \bbE[(L-\bbE L)M(L-\bbE L)^\top].
\]
However, in this form the off-diagonal and lower-right blocks depend on the full matrix \(M\). Since our analysis requires that these blocks depend only on \(M_{1,1}\) in order to close the self-consistent equations, we adopt the reduced version in~\eqref{eq:supop} which is obtained by discarding the terms involving \(M_{1,2}, M_{2,1},\) and \(M_{2,2}\) in the off-diagonal and lower-right blocks of \(\bbE[(L-\bbE L)M(L-\bbE L)^\top]\). This reduction is justified under mild assumptions (e.g., uncorrelated blocks or independent rows of \(B\) with well-conditioned covariance matrix), in which the discarded terms are negligible. The error associated with this approximation is explicitly controlled by the norm \(\|\tilde{\supop}\|\) appearing below. Finally, the precise form of \(\supop_{1,1}\) is not essential for the theory, as long as it is positivity-preserving and provides an asymptotically accurate approximation to the upper-left block of the covariance-induced superoperator. As a useful mental model, one can think of \(\supop_{1,1}(M) \approx \bbE[(L-\bbE L)M(L-\bbE L)^\top]_{1,1}\).

In order to ensure the existence of a unique solution to the Dyson equation for linearizations, we need to restrict~\eqref{eq:DEL} to a suitable set. Consequently, based on properties of the pseudo-resolvent \((L-z\Lambda)^{-1}\),\footnote{Under the present setting, \(Q\) is non-singular and the Schur complement of the lower-right \(d\times d\) block of \(L-z\Lambda\) is given by \(A-B^{\top}Q^{-1}B-z I_{n}\). Since its imaginary part is negative definite, the Schur complement is non-singular. Hence, we may apply a block matrix inversion lemma (see \Cref{lemma:block_inversion}) to conclude that the matrix \(L-z\Lambda\) is non-singular, and the pseudo-resolvent \((L-z\Lambda)^{-1}\) is well-defined for every \(z\in\bbH\).} we introduce the \emph{admissible set}.
\begin{equation}\label{eq:admissible_func}
    \calM:=\Hol(\bbH,\calA), \quad \calA:=
        \{ W\in \bbC^{\ell\times \ell}:\Im[ W]\succeq 0\text{ and }\Im[ W_{1,1}]\succ 0\}.
\end{equation}

Our primary strategy for analyzing~\eqref{eq:DEL} involves initially establishing analogous results for a regularized version of the equation. This regularization typically simplifies the problem, enabling us to leverage existing knowledge. Subsequently, we demonstrate the feasibility of setting the regularization parameter to zero, effectively reverting to the original equation. Importantly, we ensure that the statements derived for the regularized variant remain valid in this limit, thereby providing insights into the properties of~\eqref{eq:DEL}. For this reason, we introduce the \emph{regularized Dyson equation for linearizations (RDEL)}
\begin{equation}
    \label{eq:RDEL}
    ( \bbE L- \supop( M^{(\tau)})-z  \Lambda-i\tau I_{\ell}) M^{(\tau)}= I_{\ell}
\end{equation}
for every \(\tau> 0\). The corresponding admissible set is given by
\begin{equation}\label{eq:admissible_func_aug}
    \calM_{+}:=\Hol(\bbH,\calA_{+}),\quad \calA_{+}:=\{ W\in \bbC^{\ell\times \ell}:\Im[ W]\succ 0\}.
\end{equation}

Our initial key result naturally revolves around establishing the existence and uniqueness of a solution for~\eqref{eq:DEL}. In traditional Dyson equation theory, wherein the spectral parameter spans the entire diagonal, the existence of a unique solution usually emerges comprehensively from~\cite[Theorem 2.1]{helton2007operatorvalued}. Indeed, we leverage this theorem precisely to establish the existence and uniqueness of a solution to~\eqref{eq:RDEL}. However, due to the absence of a spectral parameter spanning the entire diagonal in our case, demonstrating the existence of a solution to~\eqref{eq:DEL} is not trivial and requires careful analysis. Nonetheless, by leveraging the suitable properties of the admissible set and the surrogate regularized Dyson equation, we obtain the following existence and uniqueness result.

\begin{theorem}[Main Properties]\label{theorem:main_properties}
    There exists a unique analytic matrix-valued function \(M\in \calM\) such that \(M(z)\) solves the DEL~\eqref{eq:DEL} for every \(z\in \bbH\). Additionally,
    \begin{equation*}
        M(z) =  \begin{bmatrix}
            0_{n\times n} & 0_{n\times d} \\
            0_{d\times n} & (\bbE Q)^{-1}
        \end{bmatrix}
        +  \int_{ \bbR}\frac{ \Omega(\d \lambda)}{\lambda-z}
    \end{equation*}
    for all \(z\in \bbH\), where \(\Omega\) is a compactly supported matrix-valued measure on bounded Borel subsets of \(\bbR\) satisfying 
    \[
        \int_{\bbR}\Omega(\d\lambda) =
        \begin{bmatrix}
            I_{n}            & -\bbE B^{\top}(\bbE Q)^{-1}         \\
            -(\bbE Q)^{-1} \bbE B & (\bbE Q)^{-1}\bbE[B B^{\top}](\bbE Q)^{-1}.
        \end{bmatrix}.
    \]
\end{theorem}

\Cref{theorem:main_properties} is essentially a combination of \cref{lemma:existenceuniquness,lemma:stieltjes_representation}, and we refer the reader to these lemmas for the proof. For the rest of this paper, we will utilize the notation \(M\) to represent the unique solution as ensured by \Cref{theorem:main_properties}, \(M^{(\tau)}\) to represent the unique solution to the regularized DEL~\eqref{eq:RDEL}, and we will omit the explicit mention of \(z\) when the context confines it to a fixed \(z\in \bbH\).

\subsection{Assumptions for a General Anisotropic Global Law}
\label{sec:assumptions}

Having established the existence of a unique solution \( M \) to~\eqref{eq:DEL}, our next goal is to demonstrate that this solution serves as a favorable asymptotic approximation for the pseudo-resolvent \((L - z\Lambda)^{-1}\). To achieve this, we impose several conditions on the linearization \( L \), which are essential for proving the anisotropic global law. We present these conditions in a general form and, following their statement, offer a detailed discussion on how they can be practically verified.

\begin{enumerate}
    \renewcommand{\labelenumi}{\theenumi}
    \renewcommand{\theenumi}{(C\arabic{enumi})}
    \item\emph{Boundedness.}\label{condition:flat}
    Suppose there exists \(s \in \mathbb{R}_{>0}\) such that \(\|\supop(W)\| \leq s\|W\|\) for every \(W \in \mathbb{C}^{\ell\times \ell}\) and \(\limsup_{\ell\to\infty}s<\infty\). Furthermore, assume that \(\limsup_{\ell\to\infty}\|\mathbb{E} L\| <\infty\) and \(\limsup_{\ell\to\infty}\mathbb{E}\|(L-z\Lambda)^{-1}\|^{2} <\infty\).

    \item \emph{Stability.}\label{condition:convergence_RDEL}
    For every \(z \in \mathbb{H}\), there exists a function \(f\) and a subsequence \(\{\tau_{k}\} \subseteq \mathbb{R}_{>0}\) such that \(\tau_{k} \to 0\), \(f(\tau_{k}) \to 0\), and \(\|M^{(\tau_{k})}(z)-M(z)\| \leq f(\tau_{k}) + o_{\ell}(1)\) for all \(k \in \mathbb{N}\) and every sufficiently large \(\ell \in \mathbb{N}\).

    \item \emph{Gaussian Design.}\label{condition:gaussian_design}
    There exists \(\gamma \in \mathbb{N}\), \(g \sim \mathcal{N}(0,I_{\gamma})\), and a measurable symmetric\footnote{By symmetric, we mean that \(\cov(x) = \cov(x)^{\top}\) for every \(x \in \mathbb{R}^{\gamma}\).} map \(\cov: \mathbb{R}^{\gamma} \mapsto \mathbb{R}^{\ell \times \ell}\) with \(\bbE[\cov(g)]=0\) such that \(L \equiv L(g) = \cov(g) + \mathbb{E} L\).
\end{enumerate}

Before proceeding, we briefly discuss the implications of \cref{condition:convergence_RDEL,condition:flat,condition:gaussian_design}. The first condition,~\ref{condition:flat}, ensures that key quantities remain bounded as the dimension of the problem increases. In the application of our framework, we will derive an explicit formula for the superoperator \(\supop\) in terms of the linearization \(L\). This condition is a weaker form of the \emph{flatness} condition commonly encountered in the Dyson equation literature~\cite{erdos2019matrix,alt_thesis,alt_Kronecker}. Consequently, we will adopt the term flatness to describe this property. Specifically, since the blocks of the superoperator \(\supop\) are often quadratic in the blocks of the linearization \(L\), the boundedness of \(\|\supop\|\) can be established through the boundedness of the second moments of the linearization \(L\). This holds in common settings of interest where \(L\) is correctly scaled. The boundedness of \(\|\bbE L\|\) is a natural condition, ensuring that the linearization does not have a large deterministic component. Additionally, the boundedness of \(\mathbb{E} \| (L - z \Lambda)^{-1} \|^2\) is a technical requirement that ensures the regularized pseudo-resolvent approximates the actual pseudo-resolvent effectively. Using \Cref{lemma:block_inversion}, we can express 
\[
    (L - z \Lambda)^{-1} = \begin{bmatrix}
        R & -R B^\top Q^{-1} \\
        -Q^{-1} B R & Q^{-1} + Q^{-1} B R B^\top Q^{-1}
    \end{bmatrix},
\]
where \(R = (A - B^\top Q^{-1} B - zI_n)^{-1}\). Given that \(\|R\| \leq (\Im[z])^{-1}\), the condition \(\limsup_{\ell \to \infty}\allowbreak \bbE \| (L - z \Lambda)^{-1} \|^2 < \infty\) holds whenever, for example, \(\|Q^{-1}\|\) and \(\|B\|\) are bounded in \(L^{8}\). Overall, the validity of \cref{condition:flat} can be established through boundedness of moments of the linearization \(L\).

The second condition,~\ref{condition:convergence_RDEL}, is a stability condition that ensures that the solution to the regularized Dyson equation for linearizations is a good approximation for the solution to the Dyson equation for linearizations. Stability is central to the analysis of Dyson equations, often studied via the \emph{stability operator}. Following the notation of~\cite{alt_Kronecker}, the stability operator is defined as \(\stab: X \in \bbC^{\ell \times \ell} \mapsto X - M \supop(X) M\), where \(M \equiv M(z)\) is the unique solution to a Dyson equation. The stability operator plays a crucial role in establishing the stability of the Dyson equation, as shown in~\cite[Lemma 4.10]{ajanki_stability_2019} and others~\cite{erdos2019matrix,nemish_local_2020}. The stability operator also naturally arises in the proof of \Cref{theorem:main_properties}, where its invertibility enables the recursive determination of coefficients in the power series expansion of the solution. The connection between the stability operator and \cref{condition:convergence_RDEL} becomes evident when considering the derivative of \(M^{(\tau)}(z)\) with respect to \(i\tau\), yielding \(\stab(\partial_{i\tau} M(z)) = (M(z))^2\). Since \(M(z)\) is bounded in operator norm (as a consequence of the Stieltjes transform representation in \Cref{theorem:main_properties}), having an invertible stability operator with a bounded inverse implies \cref{condition:convergence_RDEL}. This condition is satisfied within the Dyson equation framework~\cite{nemish_local_2020,Anderson_2013,fronk2023norm}, as indicated in references~\cite[Equation 4.11]{nemish_local_2020},~\cite[Estimates 6.3.3]{Anderson_2013}, and~\cite[Equation A.25]{fronk2023norm}. Generally, \cref{condition:convergence_RDEL} holds due to the dimension-independent representation of the DEL solution, constructed using free probability techniques. For example,~\cite[Lemma 5.4]{haagerup_new_2005} demonstrates that such a representation exists when \(L\) has the form \(L = A_0 \otimes I_n + \sum_{j=1}^k A_j \otimes X_j\), with \(\{A_j\}_{j=0}^{k}\) being self-adjoint matrices and \(\{X_j\}_{j=1}^{k}\) independent random matrices drawn from i.i.d.~centered Gaussian distributions.

Finally, the third condition,~\ref{condition:gaussian_design}, is a design condition ensuring that the linearization \(L\) depends on a Gaussian source. If \(\cov\) is a linear map, then the linearization \(L\) is a matrix with jointly Gaussian entries. However,~\ref{condition:gaussian_design} allows for a more general, non-linear, map \(\cov\). This allows us to study non-linear functions of random matrices, a common scenario in machine learning research. For example, the map \(\cov\) could take the form \(Z \in \bbR^{n \times d} \mapsto \sigma(XZ) \in \bbR^{n \times d}\), where \(X \in \bbR^{n \times m}\) is a fixed matrix and \(\sigma: \bbR \mapsto \bbR\) is an elementwise activation function.

We note that while conditions \ref{condition:flat} and \ref{condition:convergence_RDEL} are in some sense essential to the approach we take, \ref{condition:gaussian_design} is not.  It could be replaced by an argument that shows the error term in the DEL is small (for example by means of leave-one-out type analysis\footnote{``Leave-one-out'' methods refer to standard techniques in random matrix theory where one decomposes the resolvent of a matrix into that of a closely related matrix obtained by removing a single row, column, or entry, together with an additive perturbation term in which the removed part is decoupled from the resolvent (see, e.g., the proof of the Marchenko-Pastur law in~\cite[Section~3.3.2]{bai2010spectral}, which is referred to as ``leave-one-out'' in~\cite{Couillet_Liao_2022}). This approach is related to Lindeberg-type replacement arguments, but differs in that one removes, rather than replaces, rows, columns, or entries.}). We also note that similar assumptions have appeared in random matrix theory analyses of neural networks \cite[Assumption 1]{louart_random-matrix-approach}.  Furthermore \ref{condition:gaussian_design} is frequently satisfied for problem setups arising in the random matrix theory of neural networks.

\subsection{Anisotropic Global Law}\label{sec:main_anisotropic}

We introduce two objects required for the statement of the anisotropic global law. Let \(\tilde{\supop}\) be the operator defined as\footnote{We may also remove any term in the upper-left block of \(\bbE\left[(L-\bbE L)M(L-\bbE L)\right]\) from \(\supop\) and add them to \(\tilde{\supop}\) without changing any of our arguments.}
\begin{equation}\label{eq:Stilde}
    \tilde{\supop}:M\in \bbC^{\ell\times \ell}\mapsto \bbE\left[(L-\bbE L)M(L-\bbE L)\right]
    - \supop(M)\in \bbC^{\ell\times \ell}.
\end{equation}
We think of \(\tilde{\supop}\) as a correction term, which measures the deviation of the superoperator \(\supop\) from the ``full'' superoperator \(M\mapsto \bbE\left[(L-\bbE L)M(L-\bbE L)\right]\), that should be vanishing for \(\ell\to\infty\). We also introduce
\begin{multline}
    \Delta(L,\tau;z)  =  \bbE[(L-\bbE L)(L-z\Lambda -i\tau I_{\ell})^{-1}] 
    \\  + \bbE[(\tilde{L}-\bbE L)(L-z\Lambda -i\tau I_{\ell})^{-1}(\tilde{L}-\bbE L)(L-z\Lambda -i\tau I_{\ell})^{-1}]  \label{eq:dist_gaussian}
\end{multline}
where \(\tilde{L}\) is an i.i.d. copy of \(L\). \(\|\Delta(L,\tau)\|\) characterizes the distance between \(L\) and a matrix with Gaussian entries, in the sense that \(\Delta(L,\tau)\approx 0\) whenever \(L\) almost statisfies a matrix Stein's lemma. In fact, \(\Delta(L,\tau)=0\) holds whenever \(L\) has Gaussian entries. We defer the proof of the following lemma to Appendix B.
\begin{lemma}\label{lemma:dist_gaussian}
    If \(\tau\in \bbR_{> 0}\), \(z\in \bbH\) and \(L-\bbE L\in \bbR^{\ell\times \ell}\) is a matrix with jointly Gaussian entries, then \(\Delta(L,\tau;z)=0\).
\end{lemma}

Having established the conditions on \(L\) and introduced the necessary objects, we can now state the global anisotropic law for the pseudo-resolvent of the linearization \(L\). The following theorem, which serves as a central result, follows from \Cref{corollary:expect_deterministic_equivalent} and \Cref{lemma:concentration}, which are stated and proven in \Cref{sec:proof_global_anisotropic}, as well as \Cref{lemma:trace_norm}.

\begin{theorem}[Global Anisotropic Law for Pseudo-resolvents]\label{theorem:convergence}
    Let \(z\in \bbH\), \(M\in \calM\) be the unique solution to~\eqref{eq:DEL}, \(\tilde{\supop}\) be defined in~\eqref{eq:Stilde}, \(\Delta(L,\tau)\) be defined in~\eqref{eq:dist_gaussian} and assume that \cref{condition:convergence_RDEL,condition:flat,condition:gaussian_design} holds. Suppose that the mapping \(g\in (\bbR^{\gamma},\|\cdot\|_{2})\mapsto \supop((L(g)-z\Lambda - i \tau I_{\ell})^{-1})\in (\bbC^{\ell\times \ell},\|\cdot\|_{2})\) is \(\lambda\)-Lipschitz and \(\lim_{\ell\to\infty}\sqrt{\ell}\lambda = \lim_{\ell\to\infty}\|\tilde{\supop}\|=\lim_{\ell\to\infty}\|\Delta(L,\tau)\|=0\) for every \(\tau\in \bbR_{>0}\) small enough. Additionally, suppose that the mapping \(g\in (\bbR^{\gamma},\|\cdot\|_{2})\mapsto  (L(g)-z\Lambda)^{-1}\in (\bbC^{\ell\times \ell},\|\cdot\|_{F})\) is \(c\ell^{-r}\)-Lipschitz for some \(c,r\in \bbR_{>0}\). Then, \(\tr(U(L-z\Lambda)^{-1}-M(z))\to 0\) almost surely as \(\ell\to\infty\) for every sequence of matrices \(U\in \bbC^{\ell\times \ell}\) with \(\|U\|_{\ast}\leq 1\).
\end{theorem}

\Cref{theorem:convergence} asserts that scalar observations of the pseudo-resolvent \((L - z\Lambda)^{-1}\) converge asymptotically to scalar observations of the solution \(M(z)\) to~\eqref{eq:DEL}. It is important to note that our result is anisotropic, as the scalar observations can be taken along specific directions by appropriately choosing the sequence of matrices \(U\).

To motivate the conditions and assumptions used in \Cref{theorem:convergence}, we provide an outline of the argument through a series of pairwise comparisons:
\begin{subequations}\label{eqs:comparisons}
    \begin{align}
        (L-z\Lambda)^{-1} - M(z)  & = (L-z\Lambda)^{-1} - \bbE(L-z\Lambda)^{-1}\label{eq:comparison1}
        \\ & + \bbE(L-z\Lambda)^{-1} - \bbE(L-z\Lambda - i\tau I_{\ell})^{-1}\label{eq:comparison2} 
        \\ & + \bbE(L-z\Lambda - i\tau I_{\ell})^{-1} - M^{(\tau)}(z)\label{eq:comparison3} 
        \\ & + M^{(\tau)}(z) - M(z). \label{eq:comparison4}
    \end{align}
\end{subequations}

The first comparison in~\eqref{eq:comparison1} addresses the concentration aspect of our argument. While this difference may not always be controlled in norm, we can often establish concentration, either in probability or almost surely, of generalized trace entries of the pseudo-resolvent around its mean. Particularly when \(L\) is a function of a Gaussian source, as specified in~\ref{condition:gaussian_design}, this concentration can be demonstrated using Gaussian concentration inequalities for Lipschitz functions. This is why we require the mapping \(g \in (\bbR^{\gamma}, \|\cdot\|_{2}) \mapsto (L(g) - z\Lambda)^{-1} \in (\bbC^{\ell \times \ell}, \|\cdot\|_{F})\) to be \((c\ell^{-r})\)-Lipschitz. By isolating this concentration step, we simplify the analysis by primarily working with \emph{deterministic} objects, which offers significant simplifications in various steps and allows us to work with norm bounds. We will demonstrate this concentration in \Cref{lemma:concentration}.

The second comparison in~\eqref{eq:comparison2} assesses the proximity of the pseudo-resolvent to its regularized counterpart, measured in norm. We show in \Cref{lemma:Ftau_vs_F} that this difference can be easily controlled by the parameter \(\tau\) and \(\bbE \|(L-z\Lambda)^{-1}\|^{2}\). Consequently, if \(\bbE \|(L-z\Lambda)^{-1}\|^{2}\) is bounded as required by~\ref{condition:flat}, we can employ the regularized pseudo-resolvent with small \(\tau\in \bbR_{>0}\) as an approximation for the pseudo-resolvent.

The third comparison in~\eqref{eq:comparison3} is related to the stability properties of the RDEL. We will use the \emph{Carathéodory-Riffen-Finsler pseudometric} (CRF-pseudometric) to show that~\eqref{eq:RDEL} is stable under small additive perturbations in \Cref{sec:stability}. We will also demonstrate that the expected regularized pseudo-resolvent almost satisfies the RDEL, up to a small additive perturbation that vanishes as the dimension of the problem increases in \Cref{sec:perturbation}. To do this, we decompose the perturbation term into three components suitable for analysis. The first component arises from using the expected pseudo-resolvent and essentially requires the super-operator \(\supop\) to be \emph{averaging}. This means that \(\supop((L - z\Lambda - i\tau I_{\ell})^{-1})\) should exhibit a ``law of large numbers'' behavior and converge to a deterministic limit. We derive a condition for \(\supop((L - z\Lambda - i\tau I_{\ell})^{-1})\) to concentrate around its mean based on Gaussian concentration, which requires the Lipschitz property of the mapping \(g \in (\bbR^{\gamma}, \|\cdot\|_{2}) \mapsto \supop((L(g) - z\Lambda - i \tau I_{\ell})^{-1}) \in (\bbC^{\ell \times \ell}, \|\cdot\|_{2})\). The second perturbation term involves the operator \(\tilde{\supop}\) defined earlier, and the third involves \(\|\Delta(L, \tau)\|\). For details on the decomposition of the perturbation term, we refer the reader to \Cref{sec:perturbation}.

The fourth and final comparison,~\eqref{eq:comparison4}, simply states that the solution to~\eqref{eq:RDEL} should be a good approximation for~\eqref{eq:DEL} for small \(\tau\). For a fixed dimension \(\ell \in \bbN\), it follows from the construction of \(M\) using the limit point of the normal family \(\{M_{1,1}^{(\tau)}\st \tau > 0\}\) as \(\tau\to 0\) that \(\|M^{(\tau)}(z)-M(z)\|\to 0\) as \(\tau\to 0\). However, because we are taking \(\ell\to \infty\) and \(\tau \to 0\) concurrently, we rely on~\ref{condition:convergence_RDEL} to control this difference.

The assumptions and conditions in \Cref{theorem:convergence} are presented in a general form. However, we emphasize that we explicitly demonstrate how these conditions can be verified in the setting of a random features ridge regression model.

\subsection{Empirical Test Error of Random Features Ridge Regression}\label{sec:main_application}

Consider a supervised training problem with a labelled dataset \(\calD = \{(x_{j},y_{j})\}_{j=1}^{\ntrain}\) with \(x_{j}\in \bbR^{n_{0}}\) and \(y_{j}\in \bbR\) for every \(j\in \{1,2,\ldots, \ntrain\}\). For conciseness, let \(X\in \bbR^{n_{\ntrain}\times n_{0}}\) be the matrix with \(j\)th rows corresponding to \(x_{j}^{\top}\) and \(y\) be the vectors of labels. We wish to learn a relation between the inputs \(x_{j}\) and the outputs \(y_{j}\) by fitting the parametric function \(w\mapsto n^{-\frac{1}{2}}\sigma(x^{\top}W)w\) for some random matrix \(W\in \bbR^{n_{0}\times d}\), a \(\lambda_{\sigma}\)-Lipschitz activation function \(\sigma:\bbR\mapsto \bbR\) applied entrywise and some weights \(w\in \bbR^{d}\). This corresponds to the random features model of~\cite{Rahimi_RF}. This model can be viewed as a two-layer neural network, where the first layer is frozen at random initialization, and only the second layer is trained. Following the setup of~\cite{louart_random-matrix-approach}, we will assume that \(W=\varphi(Z)\) for some \(Z\in \bbR^{n_{0}\times d}\) with independent standard normal entries and \(\varphi:\bbR\mapsto \bbR\) a \(\lambda_{\varphi}\)-Lipschitz function applied entrywise. The Lipschitz constants \(\lambda_{\sigma}\) and \(\lambda_{\varphi}\) should be independent of the dimension of the problem in the sense that, as \(n\to\infty\) with \(\ntrain,n_{0},d\asymp n\), \(\limsup_{n\to\infty}(\lambda_{\varphi}\vee \lambda_{\sigma})<\infty\).

In order to find suitable weights \(w\), we minimize the \(\ell_{2}\)-regularized norm squared loss
\begin{equation}\label{eq:ridge_regression}
    \min_{w\in \bbR^{d}} \left\|y-A  w\right\|^{2}+ \delta \| w\|^{2}
\end{equation}
where \(A=n^{-1/2}\sigma(XW)\in \bbR^{\ntrain\times d}\) denotes the \emph{random features matrix} and \(\delta\in \bbR_{>0}\) is the \emph{ridge parameter}. In other words, we are fitting a random features model using ridge regression. The minimization problem in~\eqref{eq:ridge_regression} is strongly convex and admits the closed form solution \(\wridge = A^{\top}(AA^{\top} + \delta I_{\ntrain})^{-1}y\) which is called the \emph{ridge estimator}.

Suppose that we computed the ridge estimator \(\wridge\) which solves~\eqref{eq:ridge_regression}. To evaluate its performance, we can compute the empirical test error, or out-of-sample error, on a separate labeled dataset \(\hat{\calD} = \{(\hx_{j},\hy_{j})\}_{j=1}^{\ntest}\) using the squared norm of the residuals
\begin{equation}\label{eq:test_error}
    \Etest:=\|\hy - \hA  \wridge\|^{2} = \|\hy - \hA  A^{\top}(AA^{\top} + \delta I_{\ntrain})^{-1}y\|^{2}
\end{equation}
with \(\hA=n^{-\frac{1}{2}}\sigma(\hX W)\in \bbR^{\ntest\times d}\). This measures the performance of the model \(x\mapsto n^{-\frac{1}{2}}\sigma(x^{\top}W)\wridge\) on \(\hat{\calD}\). If \(\hat{\calD}=\calD\), then~\eqref{eq:test_error} corresponds to the training error.

Since \(\Etest\) is a scalar observation of many random variables, we expect that it will concentrate around a deterministic quantity depending on the first and second moments of \(A\) and \(\hA\). Consequently, denote
\[
\bbE[(a^{\top}_{1},\hat{a}^{\top}_{1})^{\top}(a^{\top}_{1},\hat{a}^{\top}_{1})] =
\begin{bmatrix}
    K_{AA^{\top}} & K_{A\hA^{\top}} \\
    K_{\hA A^{\top}} & K_{\hA\hA^{\top}}
\end{bmatrix}
\]
where \(\{(a^{\top}_{j},\hat{a}^{\top}_{j})^{\top}\}_{j=1}^{d}\) represent the i.i.d. columns of \(A\) and \(\hA\) respectively. Indeed, \(K_{AA^{\top}}\), \(K_{A\hA^{\top}}\), \(K_{\hA A^{\top}}\), and \(K_{\hA\hA^{\top}}\) encode the covariance between the entries of \(A\) and \(\hA\). Using the tools developed above, we verify~\cite[Conjecture 1]{louart_random-matrix-approach} under an additional boundedness assumption.

\begin{theorem}\label{theorem:rf_error}
    Assume that \(\bbE A=\bbE \hA = 0\). Furthermore, suppose that \(\ntrain,d,\ntest,n_{0}\asymp n\) such that \(\lambda_{\sigma}\), \(\lambda_{\varphi}\) \(\|X\|\), \(\|\hX\|\), \(\|y\|\), \(\|\hy\|\), \(\bbE[\|A\|]\) and \(\bbE[\|\hA\|]\) remain bounded as \(n\to\infty\).\footnote{Here, \(\ntrain,d,\ntest,n_{0}\asymp n\) formally means that there exists constants \(c_{1},c_{2},c_{3}\in \bbR_{>0}\) such that \(\lim_{n\to\infty}\frac{n}{\ntrain}=c_{1}\), \(\lim_{n\to\infty}\frac{n}{\ntest}=c_{2}\) and \(\lim_{n\to\infty}\frac{n}{n_0}=c_{3}\).} Let \(\alpha\) be the unique non-positive real number satisfying
    \[
        \alpha = -(1+\tr(K_{AA^{\top}} (\delta I_{\ntrain}-d\alpha K_{AA^{\top}})^{-1}))^{-1} \in \bbR_{\leq 0}
    \]
    and denote \(M=(\delta I_{\ntrain}-d\alpha K_{AA^{\top}})^{-1}\) as well as
    \[
        \beta = \frac{\alpha^{2}\tr(K_{\hA\hA^{\top}}+d\alpha K_{\hA A^{\top}}M(I_{\ntrain}+\delta M)K_{A\hA^{\top}})}{1-\|\sqrt{d}\alpha K_{AA^{\top}}^{\frac{1}{2}}MK_{AA^{\top}}^{\frac{1}{2}}\|_{F}^{2}}\in \bbR_{\geq 0}.
    \]
    Then, \(d\beta\|K_{AA^{\top}}^{\frac{1}{2}}My\|^{2}+\|d\alpha K_{\hA A^{\top}}M y + \hy\|^{2}- \Etest  \to 0\) almost surely as \(n\to\infty\).
\end{theorem}

\begin{figure}[h]
    \centering
    \includegraphics[width=\linewidth]{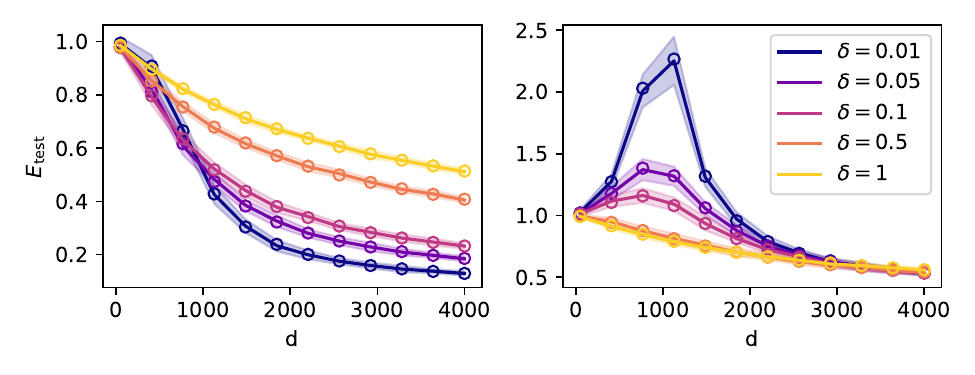}
    \caption{\(\Etest\) vs the deterministic approximation given in \Cref{theorem:rf_error} for various odd activation functions with different sizes of hidden layers \(d\) and ridge parameter \(\delta\). The data matrices, as well as the response variables, are sampled from a synthetic regression dataset, \(\ntrain=\ntest=n_{0}=1000\). Left: Error function activation (\(\sigma(x)=\mathrm{erf}(x)\)); Right: Sign activation (\(\sigma(x)=\mathrm{sign}(x)\)).}
    \label{fig:various_activation}
\end{figure}

We prove \Cref{theorem:rf_error} in \Cref{sec:proof_rf} by verifying the assumptions and applying \Cref{theorem:convergence}.

To do so, we begin by expanding the squared norm in the empirical test error \(\Etest\). This expansion reveals that it can be written as a sum of bilinear forms involving \(\hA A^\top(AA^\top + \delta I_{\ntrain})^{-1}\) and its square. We recognize this term as a block of a pseudo-resolvent corresponding to the linearization in \Cref{eq:first_linearization}, with the spectral parameter set to 0. This reduces the problem to an analysis of the pseudo-resolvent, along with a continuity argument.

The conditions and assumptions of \Cref{theorem:convergence} are verified as follows. First, \Cref{condition:gaussian_design} holds immediately, as the matrices \(X\) and \(\hX\) are Lipschitz functions from a Gaussian source. The Lipschitz continuity of the activation functions \(\varphi\) and \(\sigma\), together with the boundedness of \(\|X\|\) and \(\|\hX\|\), allow us to see the random features matrices as compositions of Lipschitz functions applied to Gaussian sources, and consequently makes it possible to apply standard concentration results. Additionally, the assumption that the first moment of \(\|A\|\) and \(\|\hA\|\) remains bounded as the dimension grows provides control over higher moments. This control makes it straightforward to verify that \Cref{condition:flat} holds. The Lipschitz condition in \Cref{theorem:convergence} is similarly satisfied due to the Lipschitzness of the activation functions and the boundedness of the data matrices. The stability condition, \Cref{condition:convergence_RDEL}, is shown to hold for all \(z\) in a neighborhood of the origin, based on a bound on the solution to the RDEL in operator norm, independent of the imaginary part of the spectral parameter. This analysis is case-specific, and depends on the structure of the particular linearization \(L\).

We demonstrate that \(\tilde{\supop}\) is negligible by leveraging an explicit formula. Finally, using a leave-one-out argument, we establish that \(\Delta(L,\tau)\) is also negligible. With these verifications, we apply \Cref{theorem:convergence} to show that the solution to the DEL is a deterministic equivalent for the first term, obtained by sending the spectral parameter to 0 via properties of the Stieltjes transform. Then, using properties of the RDEL and the result for the first term, we derive a deterministic equivalent for the squared term.

Let us briefly comment on the interpretation of \Cref{theorem:rf_error}. 
On the one hand, it shows that the empirical test error \(\Etest\) concentrates around a deterministic quantity that depends only on the first-order and second-order statistics of the random features matrices \(A\) and \(\hA\). This establishes a Gaussian equivalence principle for the empirical test error of random features ridge regression, as discussed in \Cref{sec:gaussian_equivalence}. The Gaussian equivalence principle reduces the analysis of random features ridge regression to that of a surrogate Gaussian model, thereby simplifying subsequent analyses of this model. Furthermore, the Gaussian universality of the test error implies that, once the relevant covariance operators are fixed, the distributional details of the random features beyond their second moments become immaterial. This can be viewed as a fundamental limit on statistical power: only the Gaussian approximation of the post-activation features matters for the quality of the fit.

On the other hand, because the deterministic limit is available in closed form, it enables a  systematic study of how model parameters such as the activation function \(\sigma\), the ridge parameter \(\delta\), or the mismatch between training and test distributions (encoded in the covariance matrices) affect the asymptotic performance. For instance, \Cref{sec:implicit_regularization} leverages the explicit formula in \Cref{theorem:rf_error} to link the asymptotic performance of random features ridge regression to that of kernel ridge regression, where the effective regularization is implicitly increased by the random features model.

In the rest of this section, we will discuss some aspects pertaining to the assumptions and implications of \Cref{theorem:rf_error}.

\subsubsection{Boundedness Assumptions}

The conditions
\[
    \limsup_{n\to\infty}\bbE[\|A\|]<\infty \quad \text{and} \quad \limsup_{n\to\infty}\bbE[\|\hA\|]<\infty
\]
are satisfied when the data matrices exhibit approximate orthogonality, as discussed in~\cite{fan_ck,wang2024nonlinear}. The boundedness conditions are also met with concentrated random vectors, as outlined in~\cite[Assumption 2]{Liao_2021}. Notably, these assumptions include the common case of i.i.d. standard normal entries in independent data matrices, a widely studied scenario in the literature.

It is worth noting that under the assumption that the norms of the training and test random features matrices are bounded in expectation,~\cite{louart_random-matrix-approach} conjectured \Cref{theorem:rf_error} and provided a sketch of the proof. Moreover,~\cite{Liao_2021} proved a similar result in the context of random Fourier features. While the conclusion of \Cref{theorem:rf_error} may seem unsurprising, to the best of our knowledge, this is the first time that the theorem has been rigorously established at this level of generality. Beyond being a significant theoretical result, the derivation of \Cref{theorem:rf_error} also demonstrates the broader utility of our framework based on the Dyson equation for correlated linearizations to solve problems at the interface of random matrix theory and machine learning.

\subsubsection{Bounded Denominator}
While not obvious at first, we show in \Cref{lemma:assumption_holds} that \(1-\|\sqrt{d}\alpha K_{AA^{\top}}^{1/2}MK_{AA^{\top}}^{1/2}\|_{F}^{2}\) is positive and bounded away from \(0\) as \(n\to\infty\) in the setting of \Cref{theorem:rf_error}. This implies that \(\beta\), and therefore \(\Etest\), is well-behaved in the proportional limit.

\subsubsection{Data Assumptions and Real-World Relevance}

Our assumptions concerning the norms of \(A\) and \(\hA\) implicitly impose constraints on the data matrices \(X\) and \(\hX\). However, conditioning on these matrices allows us to establish asymptotic equivalence without requiring restrictive distributional assumptions. Consequently, our results are applicable to a broad spectrum of data matrices, offering a more accurate model for real-world datasets. For instance, as illustrated in \Cref{fig:rf_rr_real_datasets}, we observe a striking alignment between the empirical test error \(\Etest\) and the deterministic approximation provided by \Cref{theorem:rf_error} across various dimensions and ridge parameters when the data is sourced from real-world flattened image classification datasets such as MNIST~\cite{deng2012mnist}, Fashion-MNIST~\cite{xiao2017fashionmnist}, CIFAR-10~\cite{Krizhevsky09learningmultiple}, and CIFAR-100~\cite{Krizhevsky09learningmultiple}. Notably, the agreement between empirical simulations and the theoretical prediction of \Cref{theorem:rf_error} holds even for datasets with anisotropic features.

\begin{figure}[t]
    \centering
    \includegraphics[width=\linewidth]{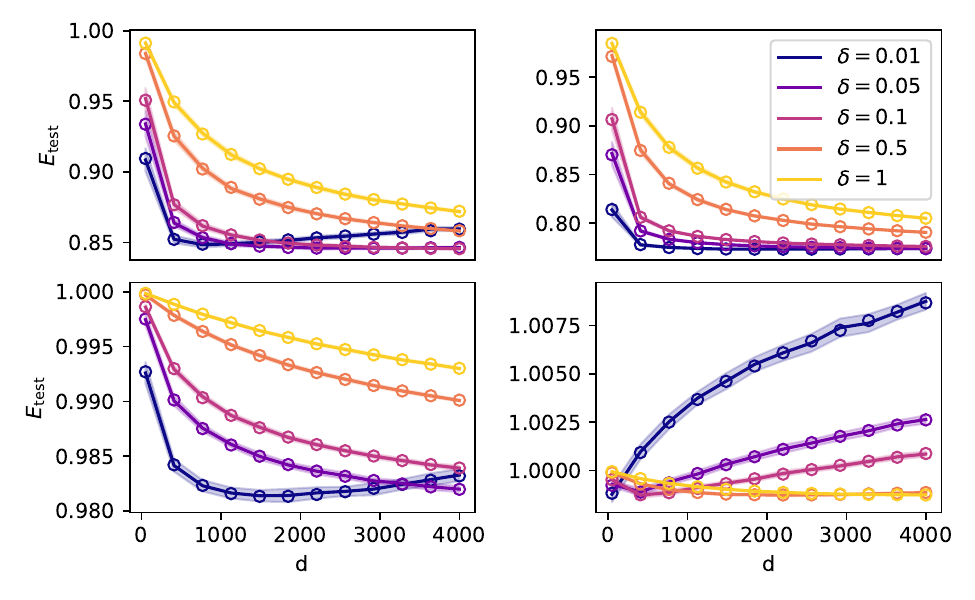}
    \caption{\(\Etest\) vs the deterministic approximation given in \Cref{theorem:rf_error} for various flattened image classification datasets with different sizes of hidden layers \(d\) and ridge parameter \(\delta\). Sine activation (\(\sigma=\sin\)), \(\ntrain=1500\), \(\ntest=1000\). Upper left: MNIST~\cite{deng2012mnist}; Upper right: Fashion-MNIST~\cite{xiao2017fashionmnist}; Lower left: CIFAR-10~\cite{Krizhevsky09learningmultiple}; Lower right: CIFAR-100~\cite{Krizhevsky09learningmultiple}.}
    \label{fig:rf_rr_real_datasets}
\end{figure}

\subsubsection{Extension to Deep Random Features}\label{sec:deep_rf}
\Cref{theorem:rf_error} extends naturally to deep random features models, which consider compositions of random feature layers
\[
    x \in \bbR^{n_{0}} \mapsto n_{k}^{-\frac{1}{2}}\sigma\left(n^{-\frac{1}{2}}_{k-1}\sigma\left(\cdots n^{-\frac{1}{2}}_2\sigma\left(n^{-\frac{1}{2}}_{1}\sigma(x^{\top}W_{1})W_{2}\right)\right)W_{k}\right)w \in \bbR^{n_{k}}.
\]
Here, \(k\in \bbN\) corresponds to the number of layers and \(\{W_{j}\in \bbR^{n_{j-1}\times n_{k}}\}_{j=1}^{k}\) is a collection of random matrices with independent standard normal entries.

More precisely, suppose we adopt the setting of~\cite{fan_ck}, which considers deep random features models in the proportional regime where \(n_1\asymp n_2\asymp \cdots \asymp n_{k}\asymp n\) with standard normal weights, an activation function \(\sigma\) that is Lipschitz with some additional regularity assumptions, and input data matrices \(X\) satisfying the notion of approximate pairwise orthogonality from~\cite[Definition~3.1]{fan_ck}. This notion notably ensures that \(\|X\|\) and \(\bbE[\|n^{-1/2}\sigma(XW)\|]\) remain bounded as \(n,d\to\infty\) for \(W\) with i.i.d.\ \(\calN(0,1)\) entries~\cite[Definition~3.1 and Lemma~D.4]{fan_ck}. Moreover,~\cite[Lemma~D.1]{fan_ck} shows that approximate orthogonality propagates through the layers with overwhelming probability as \(n,d\to\infty\).

Hence, let both \(X\) and \(\hX\) satisfy these conditions. Define
\[
    \tilde{X}=n_{k-1}^{-\frac{1}{2}}\sigma(\cdots n_{1}^{-\frac{1}{2}}\sigma(n_{0}^{-\frac{1}{2}}XW_{1})\cdots W_{k-1})\in \bbR^{\ntrain\times n_{k-1}}
\] and similarly for \(\hat{\tilde{X}}\) with \(\hX\) in place of \(X\). Let \(W := W_{k}\) and define the random features matrices at the last layer
\[
    A = n^{-1/2}\sigma(\tilde{X}W)\in \bbR^{\ntrain\times n_k},
    \qquad
    \hA = n^{-1/2}\sigma(\hat{\tilde{X}}W)\in \bbR^{\ntest\times n_k}.
\]

Conditioning on an event of overwhelming probability involving the weights \(\{W_j\}_{j=1}^{k-1}\) and using the approximate orthogonality of \(X\) and \(\hX\), it follows from~\cite[Lemma~D.1]{fan_ck} that the conditions of \Cref{theorem:rf_error} are satisfied. Applying \Cref{theorem:rf_error} then yields an object that does not depend on the last-layer weights \(W\) and is asymptotically equivalent to the empirical test error of the deep random features model with overwhelming probability as \(n,d\to\infty\). Admittedly, this object still depends on the earlier weights \(\{W_j\}_{j=1}^{k-1}\) 
through \(\tilde{X}\) and \(\hat{\tilde{X}}\), and further analysis is required to obtain a 
fully explicit deterministic equivalent, but this is beyond the scope of this paper.

\subsubsection{Gaussian Equivalence}\label{sec:gaussian_equivalence}

\Cref{theorem:rf_error} establishes a Gaussian equivalence principle, indicating that every random features model trained with ridge regression, as described in the statement of \Cref{theorem:rf_error}, performs equivalently to a surrogate Gaussian model with a matching covariance structure. However, it is important to note, as mentioned in~\cite{louart_random-matrix-approach}, that the distribution of the input data can impact the performance of the random features model. This influence stems from the fact that, although there is Gaussian equivalence at the level of random feature matrices, the distribution of the input may influence the covariance matrices \(K_{AA^{\top}}\), \(K_{A\hA^{\top}}\), \(K_{\hA A^{\top}}\), and \(K_{\hA\hA^{\top}}\), which are directly linked to the performance of the random features model.

\subsubsection{Global Anisotropic Law for Rectangular Random Matrices}

As part of the proof of \Cref{theorem:rf_error}, we establish in \Cref{corollary:first_deterministic_equiv} a result that may be of independent interest. Specifically, we show that \(\tr(U(L^{-1} - M(0))) \to 0\) almost surely as \(n \to \infty\), for any sequence \(U \in \bbC^{\ell \times \ell}\) with \(\|U\|_{\ast} \leq 1\), where \(\ell = \ntrain + d + 2\ntest\). The matrix \(L^{-1}\) is defined as
\[
L^{-1} = \begin{bmatrix}
    R & RA & RA\hA^{\top} & 0 \\
    A^{\top}R & \bar{R} & \bar{R}\hA^{\top} & 0 \\
    \hA A^{\top}R & \hA\bar{R} & \hA\bar{R}\hA^{\top} & -I_{\ntest} \\
    0 & 0 & -I_{\ntest} & 0
\end{bmatrix},
\]
where \(R = (AA^{\top} + \delta I_{\ntrain})^{-1}\), \(\bar{R} = -\delta(\delta I_{d} + A^{\top}A)^{-1}\), and \(M\) is given in~\eqref{eq:M_RF}.

For any \((i, j) \in \{1, 2, 3, 4\}^{2}\), consider \(U\) as a block matrix with \(U_{i,j}\) in the \((j, i)\)th block and zeros elsewhere. This construction leads to \(\tr(U (L^{-1} - M(0))) = U_{i,j} (L^{-1})_{i,j} - U_{i,j} M(0)_{j,i} \to 0\) almost surely as \(n \to \infty\). Moreover, the matrix \(U^{\top} U\) is block diagonal, with \(U_{i,j}^{\top} U_{i,j}\) in the \((i, i)\)-th block and zeros elsewhere, ensuring that \(\|U\|_{\ast} \leq 1\) whenever \(\|U_{i,j}\|_{\ast} \leq 1\). This implies that \(\tr(U ((L^{-1})_{i,j} - M_{i,j}(0))) \to 0\) almost surely as \(n \to \infty\), for any sequence \(U\) of the same size as \((L^{-1})_{i,j}\) and \(M_{i,j}(0)\), with \(\|U\|_{\ast} \leq 1\). This result constitutes a global anisotropic law for every block of \(L^{-1}\), providing deterministic equivalents for products of \(A\), \(\hA\), and related resolvents.

Importantly, some blocks of \(L^{-1}\) are rectangular, highlighting how our framework extends to study the asymptotic behavior of rectangular random matrices. By first deriving a global anisotropic law for a self-adjoint linearization with rectangular blocks, we can then extract deterministic equivalents for the blocks of interest.

\subsubsection{Implicit Regularization and Relation to Kernel Regression}\label{sec:implicit_regularization}

\Cref{theorem:rf_error} demonstrates the concentration of the empirical test error of random features ridge regression around the deterministic quantity \(d\beta\|K_{AA^{\top}}^{1/2}My\|^{2}+\|d\alpha K_{\hA A^{\top}}M y + \hy\|^{2}\) as the dimensions increases in the proportional regime. The second term, \(\|d\alpha K_{\hA A^{\top}}M y + \hy\|^{2}=\|\hy-d K_{\hA A^{\top}}(dK_{AA^{\top}}+(-\delta/\alpha)I_{\ntrain})^{-1} y \|^{2}\), is equivalent to the squared norm of the empirical test error for kernel ridge regression with ridge parameter \(-\delta/\alpha\in \bbR_{>0}\) and the conjugate kernel \(K(x_{1},x_{2})=\frac{d}{n}\bbE_{z\sim \calN(0,I_{n_{0}})}[\sigma(x_{1}^{\top}\varphi(z))\sigma(x^{\top}_{2}\varphi(z))]\). We will show in \Cref{sec:proof_rf} that \(-1 \leq \alpha < 0\), which implies that \(\delta<-\delta/\alpha\). This reveals that the randomness in the random features matrix acts as a form of regularization, similar to increasing the ridge parameter in kernel ridge regression. This is reminiscent of~\cite{jacot_RF} and the generalization in~\cite{chouard2022quantitative}, and is related to the implicit regularization of the random features model. In fact, the proof of \Cref{theorem:rf_error} recovers both of those results.

However, the additional term \(d\beta\|K_{AA^{\top}}^{1/2}My\|^{2}\) represents the variance in the empirical test error due to the random weights. Consequently, despite the computational benefits of the random features ridge regression compared to kernel ridge regression, \Cref{theorem:rf_error} indicates that in the proportional regime, the random features model may still underperform kernel ridge regression if \(d\beta\|K_{AA^{\top}}^{1/2}My\|^{2}\) is significant.

\subsubsection{Numerical Considerations}

Even though \Cref{theorem:rf_error} is an asymptotic result, \cref{fig:various_activation,fig:rf_rr_real_datasets} demonstrates a close match between the empirical test error \(\Etest\) and the deterministic approximation provided by \Cref{theorem:rf_error} for various activation functions \(\sigma\) and datasets. Notably, the approximation remains accurate for realistic dimensions and even for the non-Lipschitz continuous sign function. 

\Cref{theorem:rf_error} suggests that computing the asymptotic deterministic equivalent for \(\Etest\) can be reduced to solving a scalar fixed-point equation. As shown in \Cref{lemma:numerical_linearization_first}, the iterates \(\{\alpha_{k}\}_{k\in \bbN_{0}}\) obtained by iterating \(\alpha_{k+1}= -(1+\tr(K_{AA^{\top}}(\delta I_{\ntrain}-\alpha_{k} dK_{AA^{\top}})^{-1}))^{-1}\) for every \(k\in \bbN\) with arbitrary \(\alpha_{0}\in \bbR_{\leq 0}\) converge to \(\alpha\) as \(k\to\infty\). Using the spectral decomposition \(K_{AA^{\top}}=U\diag\{\lambda_{j}(K_{AA^{\top}})\}_{j=1}^{\ntrain}U^{\top}\) for some orthonormal matrix \(U\), we can rewrite the iteration as
\[
    \alpha_{k+1}= -\left(1+\sum_{j=1}^{\ntrain}\frac{\lambda_{j}(K_{AA^{\top}})}{\delta-d\alpha_{k}\lambda_{j}(K_{AA^{\top}})}\right)^{-1}.
\]
Hence, instead of performing a matrix inversion at each iteration, we compute one spectral decomposition and use the above formula to update \(\alpha_{k}\) for every \(k\in \bbN\). This allows us to more efficiently compute the deterministic equivalent of \(\Etest\) for various activation functions and dimensions. Moreover, when \(\varphi\) is the identity, the kernel matrices \(K_{AA^{\top}}\), \(K_{A\hA^{\top}}\), and \(K_{\hA\hA^{\top}}\) can be computed without Monte Carlo simulations using the closed-form expressions in~\cite[Table 1]{louart_random-matrix-approach}.

\section{Related Work}\label{sec:related_work}

In this section, we connect our results to the existing body of research.

\subsection{Conjugate Kernel}

Consider a fully-connected feedforward neural network with \(L\) layers, input dimension \(n_{0}=d\), hidden layer dimensions \(\{n_{\ell}\}_{\ell=1}^{L-1}\), scalar output and activation functions \(\sigma\). The network is defined as the parametric function
\[
    f:x \in \bbR^{n_{0}} \mapsto w^\top \frac{1}{\sqrt{n_{L}}}\sigma\left(\frac{1}{\sqrt{n_{L-1}}}\sigma\left(\cdots \frac{1}{\sqrt{n_{1}}}\sigma\left(\frac{1}{\sqrt{n_{0}}}x^{\top}W_{1}\right)W_{2}\right)\cdots W_{L-1}\right) \in \bbR,
\]
where \(W_{\ell}\in \bbR^{n_{\ell-1}\times n_{\ell}}\) for every \(\ell\in [L-1]\) and \(w\in \bbR^{n_{L}}\) are the weights of the network and the activation function \(\sigma\) is applied entrywise. Given \(n\) datapoints \(\{x_{i}\}_{i=1}^{n}\), we arrange them in a data matrix \(X\in \bbR^{n\times n_{0}}\) and let \(f(X)=[f(x_{1})^\top,\ldots,f(x_{n})^\top]^{\top}\in \bbR^{n}\) be the vector of network outputs. Recursively define the matrices of post-activation by
\[
    X_0 = X, \quad X_{\ell} = \frac{1}{\sqrt{n_{\ell}}}\sigma\left(\frac{1}{\sqrt{n_{\ell-1}}}X_{\ell-1}W_{\ell}\right) \in \bbR^{n\times n_{\ell}}, \quad \ell \in \{1,\ldots,L\}.
\]

Then, the \emph{conjugate kernel matrix} is defined as the Gram matrix of features produced by the final layer of a network, i.e., \(X_{L}X_{L}^{\top}\in \bbR^{n\times n}\)~\cite{fan_ck}. The conjugate kernel is central to the analysis of random features models. This connection stems from the fact that the output of a neural network is linear in those derived features. Thus, the conjugate kernel characterizes the training and test error of this linear model. In the case of shallow random features models (\(L=1\)), the conjugate kernel is equivalent to the Gram matrix of the random features themselves. Indeed, in the setting of \Cref{theorem:rf_error}, the conjugate kernel is given by \(AA^{\top}\).

Numerous studies have employed random matrix theory to investigate the conjugate kernel in the proportional regime. For shallow random features models with isotropic data and weight matrices, works such as~\cite{pennington_ck,benigni_moments-method} use the moment method to establish deterministic equivalents for the conjugate kernel, while~\cite{piccolo2021analysis} extends these results to include an additive bias. More generally,~\cite{chouard2022quantitative,louart_random-matrix-approach} analyze the conjugate kernel of random features models with deterministic data, relying on concentration of measure and leave-one-out techniques. Moving beyond the shallow setting,~\cite{fan_ck} introduces a notion of approximate pairwise orthogonality for the data that propagates through network layers, enabling the analysis of deep random features models (see \Cref{sec:deep_rf}). Their framework also covers isotropic data as a special case, with the difference that the data are conditioned upon. Beyond bulk spectral properties, motivated by results such as~\cite{ba2022highdimensional,cui2024asymptotics}, where a single gradient step at random initialization is shown to be equivalent to a spiked random features model, recent works have investigated outlier eigenvalues of the conjugate kernel~\cite{benigni2022largest,wang2024nonlinear}. This structure in random features model turns out to have important implications, notably in understanding the mechanisms of feature learning.

The proof of \Cref{theorem:rf_error} recovers a deterministic equivalent for the conjugate kernel \(AA^\top\), consistent with the findings of~\cite{chouard2022quantitative,louart_random-matrix-approach}, employing a significantly different approach. In particular, replacing the ridge parameter by a spectral parameter in \Cref{theorem:rf_error} and using the Stieltjes inversion lemma, we can recover the density of the empirical spectral distribution of the conjugate kernel.

\subsection{Dyson Equation and Linearizations}

The Dyson equation is a self-consistent equation that has proven to be a valuable tool in random matrix theory. For an excellent introduction to this subject, we recommend~\cite{erdos2019matrix}. It has been particularly crucial in establishing local laws for Wigner-type matrices, which are self-adjoint random matrices with independent centered entries and a prescribed variance profile~\cite{Ajanki_2017_vector,Ajanki_2019_vector1,Ajanki_2019_vector2}. In this setting, the Dyson equation collapses to a scalar one when the variance profile is stochastic (corresponding to the generalized Wigner case) or becomes vector-valued in the more general setting, in which case it is referred to as the quadratic vector equation in~\cite{Ajanki_2017_vector}. These works assume a flatness condition on the variance profile, which roughly enforces a mean-field scaling of the variances, together with boundedness of the vector solution. The latter property is easily verified in the spectral bulk but requires more delicate analysis near the spectral edges. Under these assumptions, the vector Dyson equation has been studied in great detail, and for instance it has been established that the associated self-consistent density of states is uniformly \(1/3\)-Hölder continuous with respect to the Lebesgue measure and supported on finitely many intervals, called bands~\cite{alt_band}. Overall, this line of work provides a comprehensive analysis of the vector Dyson equation and its implications for random matrices, notably optimal local laws in both the bulk and at the spectral edges. In general, such local laws are proved by combining a stability analysis of the vector Dyson map with some fluctuation-averaging techniques to control random fluctuations of the resolvent around its deterministic equivalent. While the flatness assumption covers a broad class of variance profiles, it excludes important structured examples such as band matrices or Hermite dilations, which require different methods.

Going beyond Wigner-type matrices, the Dyson equation has also been employed to study random matrices with correlated entries~\cite{Alt_2020,ERD_S_2019,ajanki_stability_2019}. In this line of work, one considers real symmetric or complex Hermitian random matrices of the form \(H = A + W\), where \(A\) is a deterministic \emph{bare matrix} and \(W\) is a random perturbation with centered entries that may exhibit correlations subject to suitable moment bounds. In contrast to the independent-entry case, these correlations are assumed to be local, in the sense that they decay with a metric distance between indices. A flatness condition is also imposed on the covariance structure, ensuring a mean-field scaling of the correlations analogous to the variance flatness in the vector Dyson setting. Under these assumptions, the Dyson equation no longer reduces to a vector equation but remains genuinely matrix-valued, giving rise to the matrix Dyson equation (MDE). The stability for the MDE was developed in~\cite{ajanki_stability_2019}, providing a foundation for further analysis. Building on this foundation,~\cite{ERD_S_2019} proved optimal local laws in the bulk for correlated ensembles by employing a multivariate cumulant expansion to show that the resolvent satisfies the MDE up to a negligible error term. In parallel,~\cite{Alt_2020} focused on regular spectral edges, which notably shows the absence of eigenvalues outside the self-consistent support and delocalization of eigenvectors. Although this framework allows for correlated entries, it remains restricted to a decaying correlation structures satisfying the mean-field flatness condition. By contrast, the models considered in this work involve structured correlation patterns that do not necessarily exhibit rapid decay, and our analysis concerns pseudo-resolvents, for which the existing MDE stability techniques cannot be applied directly.

The Dyson equation has also been applied to the study of Kronecker random matrices~\cite{alt_Kronecker}, which are block-structured random matrices obtained as linear combinations of Kronecker products between deterministic structure matrices and random matrices with independent, centered entries (subject to symmetry constraints). A key feature of this model is that it removes the flatness assumption required in earlier works, imposing only an upper bound on the entries of the variance profile. In this context, the corresponding self-consistent relation is referred to as the \emph{Hermitized Dyson equation}, which can be viewed as a Kronecker-type extension of the matrix Dyson equation. Using this formulation,~\cite{alt_Kronecker} proved that the eigenvalues of such ensembles lie, with high probability, within an \(\epsilon\)-neighborhood of the self-consistent spectrum defined by the Hermitized Dyson equation. Their stability analysis builds upon the framework developed in~\cite{ajanki_stability_2019}, suitably adapted to the Kronecker setting. An important distinction from our work is that Kronecker random matrices are linear in the random variables, whereas our models involve polynomial dependence, requiring a linearization step and the analysis of pseudo-resolvents.

Overall, the Dyson equation is a useful tool for deriving deterministic equivalents, but it is limited to certain classes of ``Wigner-like'' matrices. In order to extend its applicability to a broader class of matrices, we employ a linearization trick. The concept of linearization became particularly important following the influential work of~\cite{haagerup_new_2005}. This work essentially demonstrated that to analyze a polynomial expression in matrices, it is sufficient to consider a linear polynomial with matrix coefficients. More precisely, given a polynomial in non-commuting variables, they construct a linearization \(L\) such that the original polynomial can be recovered as a Schur complement of the linearization. Because the matrices considered in~\cite{haagerup_new_2005} are independent Gaussian Unitary Ensemble matrices, the authors can exploit operator-valued free probability to analyze the asymptotic behavior of the spectrum. Their main result shows that the empirical spectral distribution of polynomials in independent GUE matrices converges, in the large-dimensional limit, to that of the corresponding polynomial in freely independent semicircular variables.

However, a limitation of~\cite{haagerup_new_2005} is that their linearization construction does not, in general, preserve the self-adjointness of the original polynomial expression.~\cite{Anderson_2013} addressed this problem, proving that a self-adjoint polynomial expression allows the linearization's coefficients to be chosen in a way that retains self-adjointness~\cite{Anderson_2013}. Using this refined construction,~\cite{Anderson_2013} established the convergence of the largest singular value of self-adjoint polynomials in independent square Wigner matrices to the operator norm of the corresponding polynomial in free semicircular variables. Beyond the linearization itself, the analysis employs a version of the Dyson equation for linearizations, referred to there as the \emph{Schwinger-Dyson equation}. For polynomials in independent Wigner matrices, this equation collapses to a finite-dimensional system whose deterministic solution can be constructed explicitly using tools from free probability. Rather than relying on stability analysis or fluctuation-averaging methods as in the previously mentioned works on the Dyson equation,~\cite{Anderson_2013} combined matrix identities with \(L^p\)-type estimates to obtain global convergence results. Furthermore,~\cite{BelinschiMaiSpeicher+2017+21+53} combined Anderson's self-adjoint linearization with operator-valued free probability to prove that the empirical eigenvalue distribution of self-adjoint polynomials in several independent Wigner matrices converges to that of the corresponding polynomial in freely independent semicircular variables. The linearization framework was further generalized in~\cite{HELTON20181,HELTON2006105} to cover rational expressions in non-commutative variables, thereby extending the reach of free-probability techniques beyond polynomial settings.

The fact that linearizations are not unique has motivated the development of several refined constructions. In particular,~\cite{nemish_local_2020} introduced the concept of a minimal linearization, providing a systematic way to construct a self-adjoint linearization of smallest possible size for a given polynomial. This paper considers self-adjoint polynomials in several independent square random matrices with centered entries of identical variance. Within the Dyson equation for linearizations framework, the authors establish a local law for such models. A key technical insight is the identification of a nilpotent structure within the minimal linearization, which allows them to control norm of the pseudo-resolvent. Instead of the \(L^p\)-type estimates used in~\cite{Anderson_2013}, their proof relies on establishing stability of the DEL itself, combined with large-deviation and fluctuation-averaging estimates adapted from earlier Dyson-equation analyses. As in~\cite{Anderson_2013}, the DEL reduces in their setting to a finite-dimensional system of equations, since the matrices have independent entries with equal variance. The existence of the deterministic solution can therefore be stated explicitely using tools from free-probability. Analogously to our own approach, their argument involves working with a regularized DEL, where a small imaginary part is added to stabilize the inverse. They derive uniform bounds in this regularized setting before removing the regularization.

The local law proved in~\cite{nemish_local_2020} extends the earlier work~\cite{anderson_anticommutator}, which focuses on a specific polynomial, namely the anticommutator \(XY + YX\) of two independent Wigner matrices. In that setting, the author formulates and analyzes a DEL tailored to the specific self-adjoint linearization of the anticommutator, proving existence and stability of the fixed-point equation and using these properties to derive a local law for the resolvent. Finally,~\cite{fronk2023norm} investigates Hermitian non-commutative quadratic polynomials in multiple independent Wigner matrices, with a particular focus on the spectral edges. They establish a local law near the edges and, building on it, show that the operator norm of such polynomials converges to a deterministic limit. This work can be viewed as a specialization of~\cite{nemish_local_2020} to the quadratic case, where the additional structure allows for sharper control near the edges, a regime that is typically more delicate than the bulk.

A key strength of these linearization techniques is that their supporting arguments are often constructive, giving explicit instructions for creating appropriate linearizations. Combined with operator-valued free probability, the linearization trick---referred to as the pencil method in this context---has found successful applications in the study of simple neural networks~\cite{mel2022anisotropic, Adlam2019ARM, tripledescent, adlam_doubledescent, tripuraneni2021overparameterization}.

We conclude this section by briefly highlighting how our setting differs from previous works. First, we consider pseudo-resolvents rather than resolvents, which lack several of the structural and norm-control properties enjoyed by classical resolvents. As a result, the existence and stability properties of the associated Dyson equation are not directly transferable, and new arguments are required to establish them. Second, we allow for structured and possibly rectangular random matrices with correlated entries, whereas prior works on pseudo-resolvents focus on square matrices with independent entries of identical variance. Consequently, the corresponding Dyson equation no longer collapses to a finite-dimensional system, but must instead be treated as an matrix-valued equation.

Third, since the deterministic equivalent is not assumed to satisfy an algebraic equation in free variables, free-probability techniques cannot be used to construct and analyze the deterministic limit. Establishing existence and stability therefore requires a different analytic approach tailored to the pseudo-resolvent framework. Finally, we note that in this work we restrict attention to global laws, whereas~\cite{nemish_local_2020,anderson_anticommutator} derive local laws with optimal precision.

\subsection{Gaussian Equivalence}

As a consequence of \Cref{theorem:rf_error}, we discuss in~\Cref{sec:gaussian_equivalence} that random features ridge regression follows a Gaussian equivalence principle. This means that in terms of training and empirical test error, it behaves equivalently to a surrogate linear Gaussian model with matching covariance. This phenomenon was initially established for random features ridge regression under the linear regime in the sense of test error~\cite{montanari_RF}. Subsequent work proved its extension to broader loss functions and regularization, initially via non-rigorous replica methods~\cite{pmlr-v119-gerace20a} and later through rigorous analysis~\cite{goldt2022gaussian,hong_lu_GET}. Gaussian equivalence has also been demonstrated for deep random features~\cite{schroder2023deterministic} and for random features ridge regression beyond the linear scaling regime~\cite{hu2024asymptotics}.

Our work extends this literature by establishing Gaussian equivalence in terms of the empirical test error of random features ridge regression. This contributes to the generalization of Gaussian equivalence under broader distributional assumptions, aligning with the direction explored by~\cite{schröder2024asymptotics}.

\subsection{Training and Test Error of Random Features}

The random features model provides valuable insights into the behavior of more complex machine learning models and serves as a useful benchmark. Previous studies have extensively examined the test error of random features ridge regression in the proportional isotropic regime both using non-rigorous replica methods~\cite{pmlr-v119-gerace20a} and rigorous analyses~\cite{montanari_RF, Adlam2019ARM, adlam_doubledescent}. This line of research characterizes the training and test error of random features ridge regression in order to offer insights into the impacts of various model choices, such as overparameterization.

Research has expanded to address anisotropic data~\cite{hastie_22, mel2022anisotropic, MEI20223,schröder2024asymptotics} and covariate shift scenarios~\cite{tripuraneni2021overparameterization}. Beyond the linear scaling regime, studies revealed significant transitions in the degree of label learning as a function of the polynomial scaling exponent~\cite{hu2024asymptotics,MEI20223}. Additionally, some studies investigated training, test, and cross-validation errors in the highly overparameterized regime with nearly orthogonal data assumptions~\cite{wang2023overparameterized}. Another line of research investigates the test error of random features beyond ridge regression. Studies have explored generic convex losses~\cite{goldt2022gaussian} and alternative penalty terms~\cite{Loureiro_2022, bosch2023precise}. A recent trend focuses on \emph{deep random features}, a multi-layer generalization of random features. Empirical findings indicate that trained neural network outputs can be modeled by a deep random features model, with each layer's covariance corresponding to that of the neural network~\cite{guth2023rainbow}. Research has delved into deep structured linear networks~\cite{zavatone_deep_linear} and deep non-linear networks~\cite{bosch2023precise,schroder2023deterministic,schröder2024asymptotics}.

Our study diverges from these works by focusing on the empirical test error of random features ridge regression, without assuming specific data models or distributions beyond some boundedness conditions. Our main result regarding the test error of random features ridge regression is most similar to the work of~\cite{louart_random-matrix-approach}, who established an asymptotically exact expression for the training error of random features ridge regression. As previously mentioned, the authors conjectured that~\Cref{theorem:rf_error} holds without the additional conditions imposing bounded expectations for the norm of the random features matrices.~\cite{Liao_2021} resolves this conjecture in the special case of random Fourier features. Both~\cite{Liao_2021,louart_random-matrix-approach} employ leave-one-out techniques and concentration of measure arguments in their approaches, as do we in part of the universality argument.

\section{Properties of the Dyson Equation for Linearizations}\label{sec:properties}

In this section, we establish the main properties of the DEL and RDEL. It will be convenient to view~\eqref{eq:DEL} as a fixed point equation, so we introduce the \emph{DEL map}
\begin{equation}\label{eq:F}
    \F: f\in \calM\mapsto \left( \bbE L - \supop( f(\cdot))-(\cdot)  \Lambda\right)^{-1} \in \calM,
\end{equation}
assuming its well-definedness, which we establish through \cref{lemma:invertible,lemma:prior_bounds,lemma:holomorphicity}. With this definition, we can reexpress~\eqref{eq:DEL} as \(M=\F(M)\). Whenever convenient, we will fix a spectral parameter \(z\in \bbH\) and operate with \(\F\) over \(\calA\). Similarly, the formulation of~\eqref{eq:RDEL} becomes \(M^{(\tau)}=\F^{(\tau)}(M^{(\tau)})\), where
\begin{equation}\label{eq:F_tau}
    \F^{(\tau)}: f\in \calM_{+}\mapsto \left( \bbE L - \supop( f(\cdot))-(\cdot)  \Lambda-i\tau I_{\ell}\right)^{-1} \in \calM_{+}
\end{equation}
is the \emph{RDEL map}.

We begin by demonstrating general properties of the DEL map \(\F\) and the RDEL map \(\F^{(\tau)}\), such as their well-definedness. Subsequently, we consider the behavior of the fixed point equations for large spectral parameters, demonstrating that they are well-behaved in this limit. Then, we establish a Stieltjes transform representation, as well as a power series representation for the solution of the Dyson equation. These properties will be crucial to establish the existence of a unique solution to~\eqref{eq:DEL}.

\subsection{General Properties}\label{sec:general_properties}

The primary challenge in analyzing~\eqref{eq:DEL} lies in the fact that the spectral parameter does not span the entire diagonal. Consequently, obtaining certain desirable properties that are typically straightforward to establish for the full resolvent case requires additional effort. As an initial example, it is not immediately evident whether the matrix \(\bbE L - \supop(M) - z\Lambda\) is non-singular. The following lemma confirms that this is indeed the case.

\begin{lemma}\label{lemma:invertible}
    Let \(\tau\in \bbR_{\geq 0}\) and \( M\in  \calA\). Then, \( \bbE L- \supop( M)-z  \Lambda-i\tau I_{\ell}\) is invertible.
\end{lemma}

\begin{proof}
    Let \(M\in \calA\) be arbitrary. By definition of \(\calA\), \(\Im[M]\succeq 0\). Since positivity-preserving, we may use \Cref{lemma:lin_pos_symmetric} to get \(\Im[\supop(M)] = \supop(\Im[M])\succeq 0\). If \(\tau>0\), it follows directly that \(\Im[\bbE L- \supop( M)-z  \Lambda-i\tau I_{\ell}]\preceq -\tau\), which implies that \(\bbE L- \supop( M)-z  \Lambda-i\tau I_{\ell}\) is non-singular by \Cref{lemma:real/imag_singularity_bound}.

    Assume that \(\tau=0\) and let \(v^{\ast}=( v_{1}^{\ast}, v_{2}^{\ast})\) with \(v_{1}\in \bbC^{n}\) and \( v_{2}\in \bbC^{d}\) be a unitary vector in the kernel of \( \bbE L- \supop( M)-z \Lambda\). We will show that \(v=0\) and conclude that the kernel of \(\bbE L- \supop( M)-z\Lambda\) is trivial. Decomposing \(\bbE L- \supop(M)-z \Lambda\) into its real and imaginary parts, we have
    \[
        0 = v^{\ast}(\bbE L- \supop( M)-z \Lambda)v = v^{\ast}\Re[\bbE L- \supop( M)-z \Lambda]v
        + i v^{\ast}\Im[\bbE L- \supop( M)-z \Lambda]v.
    \]
    Since both \(\Re[\bbE L- \supop(M)-z \Lambda]\) and \(\Im[\bbE L- \supop(M)-z \Lambda]\) are Hermitian, the quadratic forms are real and \(v^{\ast}\Re[\bbE L- \supop( M)-z \Lambda]v=v^{\ast}\Im[\bbE L- \supop(M)-z \Lambda]v = 0\). Since \(\Im[\supop(M)]\succeq 0\), the imaginary part of the upper-left \(n\times n\) block of \( \bbE L- \supop(M)-z \Lambda\) is
    negative definite and the imaginary part of the whole matrix is negative semidefinite. Consequently, it must be the case that \(v_{1}=0\).
    
    Returning to the equation \(( \bbE L- \supop(M)-z \Lambda) v=0\), we have in particular that \(( \bbE Q- \supop_{2,2}(M)) v_{2}=0\). Left-multiplying by \(v_{2}^{\ast}\) and decomposing the matrix \(\bbE Q- \supop_{2,2}(M)\) into its real and imaginary parts,
    \[
        0 = v_{2}^{\ast}\Re[\bbE Q-\supop_{2,2}(M)]v_{2} + i v_{2}^{\ast}\Im[\bbE Q-\supop_{2,2}(M)]v_{2}.
    \]
    Again, since the real and imaginary parts of a matrix are Hermitian, the quadratic forms are real and \(v_{2}^{\ast}\Im[\bbE Q-\supop_{2,2}(M)]v_{2}=v_{2}^{\ast}\Re[\bbE Q-\supop_{2,2}(M)]v_{2}=0\). In particular, since \(M\in \calA\), \(\Im[M_{1,1}]\succ 0\) and \(0=v_{2}^{\ast}\Im[\bbE Q-\supop_{2,2}(M)]v_{2}\leq - v_{2}^{\ast}\supop_{2,2}(\Im[M])v_{2} \leq 0\), where we have once again used \Cref{lemma:lin_pos_symmetric}. By definition of \(\supop_{2,2}\), \(v_{2}^{\ast}\supop_{2,2}(\Im[M])v_{2} = \bbE v^{\ast}_{2}(B-\bbE B)\Im[M_{1,1}](B-\bbE B)^{\top}v_{2}\). Consequently, as \(\Im[M_{1,1}]\) is positive definite, it must be the case that \((B-\bbE B)^{\top}v_{2}=0\) almost surely. Going back to the equation \((\bbE Q- \supop_{2,2}(M)) v_{2}=0\), we obtain \(\bbE Q v_{2}=0\). However, \(\bbE Q\) is non-singular, which implies that \(v_{2}=0\). This is a contradiction, the kernel of \( \bbE L- \supop( M)-z \Lambda\) is trivial, and the matrix is non-singular.
\end{proof}

For every \(M\in \calA\), \(\tau\in \bbR_{\geq 0}\) and \(z\in \bbH\), the upper-left \(n\times n\) block of \(\bbE L- \supop( M)-z  \Lambda-i\tau I_{\ell}\) has negative definite imaginary part. Hence, by \Cref{lemma:real/imag_singularity_bound}, the upper-left \(n\times n\) block is non-singular. By \Cref{lemma:block_inversion}, this implies that the full matrix \(\bbE L- \supop( M)-z  \Lambda-i\tau I_{\ell}\) is non-singular if and only if the associated Schur complement is non-singular. Since we established non-singularity of the full matrix, this implies that the lower-right block of \((\bbE L- \supop( M)-z  \Lambda-i\tau I_{\ell})^{-1}\) is non-singular. Similarly, the previous lemma establishes that the lower-right \(d\times d\) block of \(\bbE L - \supop( M)-z  \Lambda-i\tau I_{\ell}\) is non-singular. Additionally, the associated Schur complement (see \Cref{lemma:block_inversion}) has a negative definite imaginary part, which, by \Cref{lemma:real/imag_singularity_bound}, implies that the Schur complement is non-singular. Consequently, we obtain the following corollary.

\begin{corollary}\label{corollary:diag_invert}
    Let \(\tau\in \bbR_{\geq 0}\) and \( M\in  \calA\). Then, the diagonal blocks of \(( \bbE L - \supop( M)-z  \Lambda-i\tau I_{\ell})^{-1}\) are invertible.
\end{corollary}

\Cref{lemma:invertible} is a first step towards considering the DEL~\eqref{eq:DEL} as a fixed point equation \(\F(M)=M\), along with its regularized counterpart. A second step in this direction is showing that \(\F\) and \(\Ftau\) both map their respective domains to themselves. We adapt the argument from~\cite{helton2007operatorvalued}.

\begin{lemma}\label{lemma:prior_bounds}
    Let \(\tau\in \bbR_{\geq 0}\), \(z\in \bbH\) and \( M\in \calA\). Then,
    \[
        \Im[\F^{(\tau)}(M)]\succeq \tau\F^{(\tau)}(M)(\F^{(\tau)}(M))^{\ast}, \quad
        \Im[\F_{1,1}^{(\tau)}(M)]\succeq \Im[z]\F_{1,1}^{(\tau)}(M)(\F_{1,1}^{(\tau)}(M))^{\ast}
    \]
    and \(\|\F_{1,1}^{(\tau)}(M)\|\leq (\Im[z])^{-1}\). Furthermore, if \(\tau>0\), then \(\|\F^{(\tau)}(M)\|\leq \tau^{-1}\).
\end{lemma}

\begin{proof}
    By \cref{lemma:real/imag_inverse,lemma:lin_pos_symmetric},
    \begin{align*}
        \Im[\F^{(\tau)}(M)] & = - \F^{(\tau)}(M)\Im[\bbE L - \supop(M) - z\Lambda - i\tau I_{\ell}](\F^{(\tau)}(M))^{\ast} 
        \\ & \succeq \F^{(\tau)}(M)(\Im[z]\Lambda+\tau I_{\ell})(\F^{(\tau)}(M))^{\ast} 
        \succeq \tau\F^{(\tau)}(M)(\F^{(\tau)}(M))^{\ast}\succeq 0.
    \end{align*}
    Leveraging the observation that the spectral norm of positive semidefinite matrices adheres to the Loewner partial ordering of positive semidefinite matrices, we obtain \(\| \Im[\F^{(\tau)}(M)]\| \geq \tau \|\F^{(\tau)}(M)(\F^{(\tau)}(M))^{\ast}\| = \tau \|\F^{(\tau)}(M)\|^{2}\). By \Cref{lemma:real/imag_norm_bound}, \(\|\F^{(\tau)}(M)\|\leq \tau^{-1}\) whenever \(\tau > 0\). For the other inequalities, extract the upper-left \(n\times n\) block of the equation \(\Im[\F^{(\tau)}(M)]\succeq \Im[z]\F^{(\tau)}(M)\Lambda(\F^{(\tau)}(M))^{\ast}\) to obtain \(\Im[\F_{1,1}^{(\tau)}(M)]\succeq \Im[z]\F_{1,1}^{(\tau)}(M)(\F_{1,1}^{(\tau)}(M))^{\ast}\) and \(\|\F_{1,1}^{(\tau)}(M)\|\leq (\Im[z])^{-1}\).
\end{proof}

It is important to note that \Cref{lemma:prior_bounds} reveals a weaker control of the DEL in comparison to the RDEL. Specifically, we only have an a priori norm bound for the upper-left \(n\times n\) block of \(\F\).

Combining \cref{lemma:invertible,lemma:prior_bounds}, it only remains to show that the DEL map \(\F\) and the RDEL map \(\F^{(\tau)}\) preserve holomorphicity to establish that they map their respective domains to themselves.

\begin{lemma}\label{lemma:holomorphicity}
    Let \(\tau\in \bbR_{\geq 0}\). Then, \(\F\) and \(\F^{(\tau)}\) are well-defined. In particular, they map their respective domains into themselves.
\end{lemma}
\begin{proof}
    As mentioned above, it suffices to prove that \(\F\) and \(\F^{(\tau)}\) preserve holomorphicity. We will only prove this for \(\F\), as the proof for \(\F^{(\tau)}\) is analogous. Let \(M\in \calM\) be arbitrary and \(z,h\in \bbH\). Since \(M\) is holomorphic on \(\bbH\), let \(\D M(z):h\in\bbH \mapsto \D M(z)h\) be the Fréchet derivative of \(M\) at \(z\). Furthermore, let \(\D [\F (M(z))(z)]:h\in \bbH \mapsto \F(M(z))(z)\supop(\D M(z)h)\F(M(z))(z)+h[\F(M(z))(z)]^{2}\). Clearly, \(\D [\F (M(z))(z)]\) is a bounded linear operator. Furthermore, \(\F(M(z+h))(z+h) - \F(M(z))(z)-\D [\F (M(z))(z) = X_{1}+X_{2}\) with 
    \[
        X_{1}=-\F(M(z+h))(z+h)\supop(M(z+h)-M(z)-\supop(\D M(z)h))\F(M(z))(z+h)
    \] 
    and \(X_{2}= h([\F(M(z))(z)]^{2} -\F(M(z))(z+h)\F(M(z))(z))\).
    On one hand, since \(M\) is holomorphic,
    \[
        \lim_{h\to 0} \frac{\|X_{1}\|}{|h|} 
         \leq  s\|\F(M(z))(z)\|^{2} \lim_{h\to 0}\frac{\|M(z+h)-M(z)-\supop(\D M(z)h)\|}{|h|}
         = 0.
    \]
    On the other hand, by continuity of the map \(z\mapsto \F(M(z))(z)\) on \(\bbH\), we have \(\lim_{h\to 0} \|X_{2}\|/h=0\). Consequently, \(\F(M(z+h))(z+h) - \F(M(z))(z)-\D [\F (M(z))(z)]h = o(h)\), which implies that \(z\mapsto \F(M(z))(z)\) is holomorphic on \(\bbH\). This concludes the proof.
\end{proof}

The insights garnered from \cref{lemma:holomorphicity,lemma:prior_bounds,lemma:invertible} leads us to conceptualize the DEL and RDEL as fixed point equations. This characterization is instrumental, as we rely on it to assert the existence of a unique solution for~\eqref{eq:RDEL} using the contractive property of the RDEL map \(\F^{(\tau)}\) with respect to the CRF-pseudometric. Furthermore, this perspective will be crucial to establish the stability of the DEL.

In what follows, we will take advantage of the block structure of the DEL and RDEL. Using \Cref{lemma:block_inversion}, we decompose~\eqref{eq:DEL} as
\begin{subequations}\label{eqs:M_blocks}
    \begin{align}
        M_{1,1} = & \left(\T_{1,1}(M) - \T_{1,2}(M)(\T_{2,2}(M))^{-1} \T_{2,1}(M) \right)^{-1}, \label{eq:M11_schur}
        \\  M_{2,2} = & \left( \T_{2,2}(M)-\T_{2,1}(M)(\T_{1,1}(M))^{-1} \T_{1,2}(M) \right)^{-1}, \label{eq:M22_schur}
        \\  M_{1,2} = & \,  - \F_{1,1}(M) \T_{1,2}(M)(\T_{2,2}(M))^{-1}\text{ and} \label{eq:M12_schur}
        \\ M_{2,1} = &\, -(\T_{2,2}(M))^{-1}\T_{2,1}(M) \F_{1,1}(M) \label{eq:M21_schur}
    \end{align}
where we write \(\T(M) = \bbE L - \supop(M)-z\Lambda\) for notational convenience. It may sometimes be practical to work with the equivalent form
\begin{equation}\label{eq:M22_schur_v2}
    M_{2,2} =  (\T_{2,2}(M))^{-1} +(\T_{2,2}(M))^{-1}\T_{2,1}(M)\F_{1,1}(M) \T_{1,2}(M)(\T_{2,2}(M))^{-1}.
\end{equation}
\end{subequations}
We may decompose~\eqref{eq:RDEL} similarly. In this case, we will write \(\T^{(\tau)}(M) =  \bbE L- \supop( M^{(\tau)})-z\Lambda-i\tau I_{d}\).

\subsection{Large Spectral Parameter}

For any solution \(M\) to~\eqref{eq:DEL}, \Cref{lemma:prior_bounds} implies that \(\| M_{1,1}(z)\|\leq (\Im[z])^{-1}\). This bound is particularly useful when \(\Im[z]\) is large, as it allows us to ensure that the norm of the upper-left \(n\times n\) block of \(M\) is arbitrarily small. In fact, as this suggests, it will be beneficial to analyze the limit of the DEL map as \(\Im[z]\) grows large. For every \(\tau\in \bbR_{\geq 0}\), we define
\begin{equation}\label{eq:M_inf}
M^{(\tau)}_{\star}  = ( \bbE Q -i\tau I_{d}
)^{-1} \quad\text{ and }\quad  M^{(\tau)}_{\infty} =
\begin{bmatrix}
    0_{n\times n} & 0_{n\times d}      \\
    0_{d\times n} & M^{(\tau)}_{\star}
\end{bmatrix}
\end{equation}
and denote \( M_{\star}= M_{\star}^{(0)}\), \( M_{\infty}= M_{\infty}^{(0)}\). Indeed, by \Cref{lemma:real/imag_inverse}, \(\Im[ M^{(\tau)}_{\ast}]\succeq 0\). We demonstrate in \Cref{lemma:convergence_img} that \(M^{(\tau)}_{\star}\) corresponds precisely to the limit of \(M^{(\tau)}(z)\) as \(\Im[z]\) diverges to infinity.

\begin{lemma}\label{lemma:convergence_img}
    Fix \(\tau\in \bbR_{\geq 0}\) and assume that \( M^{(\tau)}\in \calM\) such that, for all \(z\in  \bbH\), \( M^{(\tau)}(z)\) solves the RDEL~\eqref{eq:RDEL}. Then, \( \| M^{(\tau)}(z)- M^{(\tau)}_{\infty}\|\to 0\) as \(\Im[z]\to\infty\).
\end{lemma}
\begin{proof}
    We proceed block-wise. By \Cref{lemma:prior_bounds}, \(\| M^{(\tau)}_{1,1}(z)\|\leq (\Im[z])^{-1}\). Consequently, \(\| M^{(\tau)}_{1,1}(z)\|\to 0\) as \(\Im[z]\to\infty\). Furthermore, it follows from \Cref{lemma:real/imag_inverse} and the assumed properties of the superoperator that \(\Im[(\T^{(\tau)}_{1,1}(M^{(\tau)}))^{-1}]\succeq \Im[z](\T^{(\tau)}_{1,1}(M^{(\tau)}))^{-1}(\T^{(\tau)}_{1,1}(M^{(\tau)}))^{-\ast}\). Here, we use the notation \(\T^{(\tau)}(M)=\bbE L-\supop(M)-z\Lambda-i\tau I_{\ell}\) introduced in \Cref{eqs:M_blocks}. Taking the norm of both sides and rearranging, we obtain \(\|(\T^{(\tau)}_{1,1}(M^{(\tau)}))^{-1}\| \leq (\Im[z])^{-1}\). Hence, \(\|(\T^{(\tau)}_{1,1}(M^{(\tau)}))^{-1}\|\to 0\) as \(\Im[z]\to\infty\). Furthermore, using the flatness of the superoperator, we have \(\|\supop_{1,2}(M^{(\tau)})\|\vee \|\supop_{2,1}(M^{(\tau)})\|\vee \|\supop_{2,2}(M^{(\tau)})\| \to 0\) as \(\Im[z]\to\infty\). This implies that \(\T_{1,2}(M^{(\tau)})\to \bbE B^{\top}\), \(\T_{2,1}(M^{(\tau)})\to \bbE B\) and \(\T_{2,2}(M^{(\tau)})\to \bbE Q-i\tau I_{d}\) as \(\Im[z]\to\infty\). Since \(\bbE Q-i\tau I_{d}\) is non-singular and the taking a matrix inverse is a continuous operation, \((\T_{2,2}(M^{(\tau)}))^{-1}\to M^{(\tau)}_{\ast}\) as \(\Im[z]\to\infty\). Finally, using \Cref{eqs:M_blocks}, we conclude that \(M^{(\tau)}(z)\to M^{(\tau)}_{\infty}\) as \(\Im[z]\to\infty\).
\end{proof}

\begin{remark}
    Similarly to \Cref{lemma:convergence_img}, it follows from \cref{lemma:block_inversion,lemma:real/imag_singularity_bound} that \(\|(L-z\Lambda - i\tau I_{\ell})^{-1}-\diag\{0_{n\times n},(Q-i\tau I_{d})^{-1}\}\|\) as \(\Im[z]\to\infty\).
\end{remark}

Although the pseudo-resolvent and the solution to the (R)DEL display favorable properties when the spectral parameter moves far from the real axis, it is crucial to understand their behavior near the real axis, as this region contains the spectral information. Hence, we need to bring the spectral parameter closer to the real axis. The next lemma constructs a loose bound on the norm of any solution to~\eqref{eq:RDEL}. This bound holds for every spectral parameter large enough in norm, regardless of the magnitude of its imaginary part. As a result, we can explore the behavior of the solution of the (R)DEL for large spectral values that are close to the real line.

\begin{lemma}\label{lemma:large_z_norm_bound}
    Fix \(\tau\in \bbR_{\geq 0}\) and assume that \( M^{(\tau)}\in \calM\) such that, for all \(z\in  \bbH\), \(M^{(\tau)}(z)\) solves the RDEL~\eqref{eq:RDEL}. Then, there exists some constant \(c\in \bbR_{>0}\) such that \(\|M^{(\tau)}(z)- M^{(\tau)}_{\infty}\| \leq c (|z|-\kappa)^{-1}\) for all \(z\in \{z\in  \bbH \st |z|>\kappa+c\kappa^{-1}\}\) with \(\kappa:= 2  \|( \bbE Q-i\tau I_{d})^{-1}\|(\| \bbE B\|+(2\|( \bbE Q-i\tau I_{d})^{-1}\|)^{-1})^{2}+\| \bbE A\|+(2\|( \bbE Q-i\tau I_{d})^{-1}\|)^{-1}+s\| ( \bbE Q-i\tau I_{d})^{-1}\|\).
\end{lemma}

\begin{proof}
    Fix \(z \in \bbH\) with \(|z|>\kappa\) and let \(M^{(\tau)}\equiv M^{(\tau)}(z)\). For notational convenience, we denote \(a=\|\bbE A\|\), \(b=\|\bbE B\|\) and \(m_{\star}=\|M_{\star}^{(\tau)}\|\). We will show that there exists \(c\in \bbR_{>0}\) such that \(\|M^{(\tau)}(z)-M_{\infty}^{(\tau)}\| \notin (c(|z|-\kappa)^{-1},\kappa]\) for every \(z\in \bbH\) with \(|z|> \kappa+c\kappa^{-1}\). By \Cref{lemma:convergence_img}, \(\|M^{(\tau)}(z)-M_{\infty}^{(\tau)}\|\) is in a neighborhood of the origin for every \(z\in \bbH\) with \(\Im[z]\) large enough. Since \(z\mapsto \|M^{(\tau)}(z)-M_{\infty}^{(\tau)}\|\) is a continuous function on \(\{z\in \bbH\st |z|> \kappa+c\kappa^{-1}\}\), this will imply that \(\|M^{(\tau)}(z)-M_{\infty}^{(\tau)}\|\in [0,c(|z|-\kappa)^{-1}]\) for every \(z\in \{z\in \bbH\st |z|> \kappa+c\kappa^{-1}\}\). 
    
    Suppose that \(\|M^{(\tau)} - M^{(\tau)}_{\infty}\|\leq (2s m_{\star})^{-1}\) such that \(\|M^{(\tau)}\|\leq \|M^{(\tau)} - M^{(\tau)}_{\infty}\|+\|M^{(\tau)}_{\infty}\|\leq (2sm_{\star})^{-1}+m_{\star}\). We consider the blocks separately using~\eqref{eqs:M_blocks}. By definition of \(\supop_{2,2}\), \(\|\supop_{2,2}(M^{(\tau)})\| \leq s\| M^{(\tau)}_{1,1}\| \leq (2m_{\star})^{-1}\). It follows from \Cref{lemma:neumann} that
    \[
        \|(\T_{2,2}(M^{(\tau)}))^{-1}\| = \|( \bbE Q-i\tau I_{d})^{-1}\left( \supop_{2,2}(M^{(\tau)})( \bbE Q-i\tau I_{d})^{-1}- I_{d}\right)^{-1}\| \leq 2m_{\star}.
    \]
    By subadditivity of the spectral norm,
    \begin{multline*}
        \|\bbE A-\supop_{1,1}(M^{(\tau)}) - \T_{1,2}(M^{(\tau)})(\T_{2,2}(M^{(\tau)}))^{-1} \T_{2,1}(M^{(\tau)}) \| \leq a+(2m_{\star})^{-1}\\+2m_{\star}(b+(2m_{\star})^{-1})^{2}.
    \end{multline*}
    For \(z\in \bbH\) with \(|z|>\kappa \geq a+(2m_{\star})^{-1}+2m_{\star}(b+(2m_{\star})^{-1})^{2}\), it follows from \Cref{lemma:neumann,eq:M11_schur} that \(\|M^{(\tau)}_{1,1}\| \leq (|z|-\kappa)^{-1}\). We now turn our attention to \(M^{(\tau)}_{2,2}\). Using \Cref{lemma:res_trick,eq:M22_schur},
    \[
        M^{(\tau)}_{2,2}-M_{\star}^{(\tau)} = M^{(\tau)}_{2,2}\left(\supop_{2,2}(M^{(\tau)})+\T_{2,1}(M^{(\tau)})(\T_{1,1}(M^{(\tau)}))^{-1} \T_{1,2}(M^{(\tau)}) \right)^{-1}M^{(\tau)}_{\ast}.
    \]
    By \Cref{lemma:neumann}, \(\|(\T_{1,1}(M^{(\tau)}))^{-1}\| \leq (|z|-a-(2m_{\star})^{-1}-sm_{\star})^{-1}\). Hence,
    \begin{multline*}
        \|M^{(\tau)}_{2,2}-M_{\star}^{(\tau)}\| \leq s m_{\star}((2sm_{\star})^{-1}+m_{\star})\|M^{(\tau)}_{1,1}\|
        \\ +m_{\star}((2sm_{\star})^{-1}+m_{\star})(b+(2m_{\star})^{-1})^{2}(|z|-\kappa)^{-1}.
    \end{multline*}
    Plugging the bound for \(\|M^{(\tau)}_{1,1}\|\) derived above and simplifying,
    \[
        \|M^{(\tau)}_{2,2}-M_{\star}^{(\tau)}\| \leq m_{\star}(s+(b+(2m_{\star})^{-1})^{2})((2sm_{\star})^{-1}+m_{\star})(|z|-\kappa)^{-1}.
    \]
    It only remains to treat \(\|M^{(\tau)}_{1,2}\|\) and \(\|M^{(\tau)}_{2,1}\|\), which we directly bound by
    \[
        \max\{\|M^{(\tau)}_{1,2}\|,\|M^{(\tau)}_{2,1}\|\}\leq 2m_{\star}(b+(2m_{\star})^{-1})\|M^{(\tau)}_{1,1}\| \leq 2m_{\star}(b+(2m_{\star})^{-1})(|z|-\kappa)^{-1}
    \]
    using~\eqref{eq:M12_schur} and~\eqref{eq:M21_schur}. To summarize, we showed that for every \(z\in \bbH\) with \(|z|>a+(2m_{\star})^{-1}+2m_{\star}(b+(2m_{\star})^{-1})^{2}+sm_{\star}\), \(\|M^{(\tau)}(z)- M^{(\tau)}_{\infty}\|\leq (2sm_{\star})^{-1}\) implies that
    \[
        \|M^{(\tau)}(z)- M^{(\tau)}_{\infty}\|  \leq \|M^{(\tau)}_{1,1}(z)\|+\|M^{(\tau)}_{1,2}\|+\|M^{(\tau)}_{2,1}\|+\|M^{(\tau)}_{2,2}- M^{(\tau)}_{\star}\|
        \leq c (|z|-\kappa)^{-1}
    \]
    with \(c:=1+m_{\star}(s+(b+(2m_{\star})^{-1})^{2})((2sm_{\star})^{-1}+m_{\star})+4m_{\star}(b+(2m_{\star})^{-1})\). Choosing \(|z|> \kappa+c\kappa^{-1}\) completes the proof.
\end{proof}

The previous lemma is a key step in controlling the norm of the solution to the RDEL for large spectral parameters. We will now proceed with our analysis of the RDEL by establishing an upper bound on the imaginary part of any solution when the spectral parameter large.

\begin{lemma}\label{lemma:large_z_img_bound}
    Fix \(\tau\in \bbR_{\geq 0}\) and assume that \( M^{(\tau)}\in \calM\) such that, for all \(z\in  \bbH\), \( M^{(\tau)}(z)\) solves the RDEL~\eqref{eq:RDEL}. Let \(\kappa\) and \(c\) be defined as in \Cref{lemma:large_z_norm_bound}. Then, there exists \(\kappa_{+}\geq \kappa+c\kappa^{-1}\) and constant \(c_{+} \in \bbR_{>0}\) such that \(\|\Im[M^{(\tau)}_{1,1}(z)]\| \leq c_{+}(|z|-\kappa)^{-2}(\tau + \Im[z])\) for every \(z\in \{z\in \bbH \st |z|\geq \kappa_{+}\}\). In particular, if \(\tau=0\), then \(\|\Im[M(z)]\|\) converges uniformly to \(0\) as \(\Im[z]\downarrow 0\) on \(\{z\in \bbH \st \Re[z]\geq \kappa_{+}\}\).
\end{lemma}
\begin{proof}
    Let \(m\equiv m(z)=c(|z|-\kappa)^{-1}\) be the bound in \Cref{lemma:large_z_norm_bound} and choose \(\kappa_{+}\geq \kappa + c\kappa^{-1}\) such that \(3sm^{2}\leq 2^{-1}\) and \(4sm^{2}(1+2(m+m_{\star})^{2})\leq 2^{-1}\) for all \(z\in \bbH\) with \(|z|\geq \kappa_{+}\). Fix \(z\in \bbH\) with \(|z|\geq \kappa_{+}\) and denote \( M^{(\tau)}= M^{(\tau)}(z)\). 
    
    By \cref{lemma:real/imag_inverse,lemma:lin_pos_symmetric}, we may write \(\Im[M^{(\tau)}] = M^{(\tau)} (\Im[z]\Lambda + i\tau I_{\ell} + \supop(\Im[M^{(\tau)}]))(M^{(\tau)})^{\ast}\). By \Cref{lemma:large_z_norm_bound}, \(\|M^{(\tau)}\|\leq \|M^{(\tau)}-M^{(\tau)}_{\infty}\|+\|M^{(\tau)}_{\infty}\|\leq m+m_{\star}\) where we denote \(m_{\star}=\|M^{(\tau)}_{\ast}\|=\|(\bbE Q-i\tau I_{d})^{-1}\|\). Furthermore, by flatness, \(\|\supop_{1,2}(\Im[M^{(\tau)}])\|\vee \|\supop_{2,1}(\Im[M^{(\tau)}])\|\vee \|\supop_{2,2}(\Im[M^{(\tau)}])\|\leq s\|\Im[M^{(\tau)}_{1,1}]\|\) and \(\|\supop_{1,1}(\Im[M^{(\tau)}])\| \leq s\|\Im[M^{(\tau)}]\|\). Let \(N\in \bbR^{2\times 2}\) such that \(N_{j,k}= \|\Im[M^{(\tau)}_{j,k}]\|\). Then,
    \[
        N
        \leq
        \begin{bmatrix}
            m & m \\
            m & m+ m_{\star}
        \end{bmatrix}
        \begin{bmatrix}
            \Im[z] + \tau +s\|\Im[M^{(\tau)}]\| & s\|\Im[M^{(\tau)}_{1,1}]\|      \\
            s\|\Im[M^{(\tau)}_{1,1}]\|  & \tau +s\|\Im[M^{(\tau)}_{1,1}]\|
        \end{bmatrix}
        \begin{bmatrix}
            m & m \\
            m & m+ m_{\star}
        \end{bmatrix},
    \]
    where the inequality is entry-wise. Expanding the product, we get that
    \[
        \|\Im[M^{(\tau)}_{j,k}]\| \leq m^{2}\Im[z] + 2(m+m_{\star})^{2}\tau + 3s(m+m_{\star})^{2}\|\Im[M^{(\tau)}_{1,1}]\| + sm^{2}\|\Im[M^{(\tau)}]\|
    \]
    for every \((j,k)\in \{(1,2),(2,1),(2,2)\}\). In particular, since \(\|\Im[M^{(\tau)}]\| \leq \sum_{j,k=1}^{2}\|\Im[M^{(\tau)}_{j,k}]\|\),
    \[
        x \leq  m^{2}\Im[z] + 2(m+m_{\star})^{2}\tau + 4s(m+m_{\star})^{2}\|\Im[M^{(\tau)}_{1,1}]\| + 3sm^{2}x
    \] 
    where \(x=\|\Im[M^{(\tau)}_{1,2}]\|\vee \|\Im[M^{(\tau)}_{2,1}]\| \vee \|\Im[M^{(\tau)}_{2,2}]\|\). Given our choice of \(\kappa_{+}\), we can rearrange to obtain \( x \leq  2m^{2}\Im[z] + 4(m+m_{\star})^{2}\tau + 8s(m+m_{\star})^{2}\|\Im[M^{(\tau)}_{1,1}]\|\). Using the bound for \(N\) above,
    \begin{align*}
        \|\Im[M^{(\tau)}_{1,1}]\| & \leq m^{2}\Im[z]+2m^{2}\tau+4sm^{2}\|\Im[M^{(\tau)}_{1,1}]\|+m^{2}x
        \\ & \leq m^{2}(1+2m^{2})\Im[z]+2m^{2}(1+2(m+m_{\star})^{2})\tau
        \\ & + 4sm^{2}(1+2(m+m_{\star})^{2})\|\Im[M^{(\tau)}_{1,1}]\|.
    \end{align*}
    Rearranging, we obtain that \(\|\Im[M^{(\tau)}_{1,1}]\| \leq 2m^{2}(1+2m^{2})\Im[z]+4m^{2}(1+2(m+m_{\star})^{2})\tau\). This proves the first part of the lemma. The second part follows from the first part and the derived bound on \(x\).
\end{proof}

\subsection{Stieltjes Transform Representation}\label{sec:stieltjes}

Given our chosen admissible set, any solution to~\eqref{eq:DEL} or~\eqref{eq:RDEL} is a matrix-valued Herglotz function. Specifically, every solution possesses a Stieltjes transform representation, as guaranteed by the Nevanlinna, or Riesz-Herglotz, representation theorem~\cite{matrix_herglotz}. This representation will prove particularly advantageous, as the positive semidefinite measure in the Stieltjes transform representation of the solution to the DEL is compactly supported.

\begin{lemma}[Stieltjes transform representation]\label{lemma:stieltjes_representation}
    Assume that \( M\in \calM\) such that, for all \(z\in  \bbH\), \( M(z)\) solves the DEL~\eqref{eq:DEL}. Then,
    \begin{equation*}
        M(z) =  M_{\infty} +  \int_{ \bbR}\frac{ \Omega(\d \lambda)}{\lambda-z}
    \end{equation*}
    for all \(z\in \bbH\), where \(\Omega\) is a matrix-valued measure on bounded Borel subsets of \(\bbR\). Additionally, \(\Omega\) is compactly supported and satisfies
    \[
        \int_{\bbR}\Omega(\d\lambda) =
        \begin{bmatrix}
            I_{n}            & -\bbE B^{\top}(\bbE Q)^{-1}         \\
            -(\bbE Q)^{-1} \bbE B & (\bbE Q)^{-1}\bbE[B B^{\top}](\bbE Q)^{-1}.
        \end{bmatrix}.
    \]
\end{lemma}

\begin{proof}
    By \Cref{lemma:convergence_img}, \(\|M-M_{\infty}\|\to 0\) as \(\Im[z]\to\infty\). Using~\eqref{eq:M11_schur}, we have
    \[
        zM_{1,1} = -I_{n}+ (\bbE A - \supop_{1,1}(M)-\T_{1,2}(M)(\T_{2,2}(M))^{-1} \T_{2,1}(M))M_{1,1}.
    \]
    Hence, \(zM_{1,1}(z)\to -I_{n}\) as \(\Im[z]\to \infty\). In addition, by~\eqref{eq:M12_schur},~\eqref{eq:M21_schur} and \Cref{lemma:convergence_img}, \(zM_{1,2}\to \bbE B^{\top}(\bbE Q)^{-1}\) and \(zM_{2,1}\to (\bbE Q)^{-1}\bbE B\) as \(\Im[z]\to \infty\). Finally, by~\eqref{eq:M22_schur_v2} and \Cref{lemma:res_trick},
    \[
        z(M_{2,2}-M_{\star})=(\T_{2,2}(M))^{-1}\supop_{2,2}(zM)M_{\star} +(\T_{2,2}(M))^{-1}\T_{2,1}(M)zM_{1,1} \T_{1,2}(M)(\T_{2,2}(M))^{-1}.
    \]
    By definition of the superoperator,
    \[
        \lim_{\Im[z]\to\infty}\supop_{2,2}(zM)=\bbE[(B-\bbE B)\lim_{\Im[z]\to\infty}zM_{1,1}(B-\bbE B)^{\top}]=-\bbE[(B-\bbE B)(B-\bbE B)^{\top}].
    \]
    Since \(\supop_{2,2}(M)\) approach \(0\) as \(\Im[z]\) approaches infinity, \((\T_{2,2}(M))^{-1}\to (\bbE Q)^{-1}\) as \(\Im[z]\to \infty\). Also, \(\T_{1,2}(M)\to \bbE B^{\top}\) and \(\T_{2,1}(M)\to \bbE B\) as \(\Im[z]\to \infty\). We get \(z(M_{2,2}-M_{\star})\to (\bbE Q)^{-1}\bbE[BB^{\top}](\bbE Q)^{-1}\) as \(\Im[z]\to\infty\). Since \(M_{\infty}\) is real, \(M-M_{\infty}\) is a matrix-valued Herglotz function. Hence, by~\cite[Theorem 2.3(iii) and Theorem 5.4(iv)]{matrix_herglotz},
    \[
        M(z) =  M_{\infty} +  \int_{ \bbR}\frac{ \Omega(\d \lambda)}{\lambda-z}
    \]
    for all \(z\in \bbH\), where \(\Omega\) is a matrix-valued measure on bounded Borel subsets of \(\bbR\) satisfying
    \[
        \int_{\bbR}\Omega(\d\lambda)= -\lim_{\eta\to\infty}i\eta M(i\eta)=
        \begin{bmatrix}
            I_{n} & -\bbE B^{\top}(\bbE Q)^{-1}         \\
            -(\bbE Q)^{-1} \bbE B & (\bbE Q)^{-1}\bbE[B B^{\top}](\bbE Q)^{-1}.
        \end{bmatrix}
    \]

    It only remains to show that \(\Omega\) is compactly supported. By \Cref{lemma:large_z_img_bound}, the imaginary part \(\Im[M(z)]\) converges uniformly to \(0\) as \(\Im[z]\to 0\) on \(\{z\in \bbH \st \Re[z]\geq \kappa_{+}\}\). Hence, by~\cite[Theorem 5.4(v) and Theorem 5.4(vi)]{matrix_herglotz}, \(\Omega\) is compactly supported on \([-\kappa_{+},\kappa_{+}]\), where \(\kappa_{+}\) is the constant defined in \Cref{lemma:large_z_img_bound}.
\end{proof}

We can interpret the integral representation in \Cref{lemma:stieltjes_representation} as a matrix-valued Stieltjes transform, and we will refer to it using this terminology. Furthermore, given the normalization of \(\Omega\) in \Cref{lemma:stieltjes_representation}, we say that \(\Omega_{1,1}\) is a matrix-valued probability measure in the sense that \(v^{\ast}\Omega_{1,1}v\) is a real Borel measure satisfying \(\int_{\bbR}v^{\ast}\Omega_{1,1}(\d\lambda)v=1\) for every \(v\in \bbC^{n}\) with \(\|v\|=1\).

\Cref{lemma:stieltjes_representation} provides an explicit bound on the solution to~\eqref{eq:DEL}, which we state in the following corollary.

\begin{corollary}\label{corollary:norm_bound}
    Assume that \( M\in \calM\) such that, for all \(z\in  \bbH\), \( M(z)\) solves the DEL~\eqref{eq:DEL}. Then, for every \(z\in \bbH\),
    \begin{equation}
        \|M(z)\|\leq \|M_{\infty}\|+ \mathrm{dist}(z,\supp(\Omega))^{-1}\left\|\int_{\bbR}\Omega(\d\lambda)\right\|.
    \end{equation}
\end{corollary}

It is tempting to try to directly generalize \Cref{lemma:stieltjes_representation} to the solution of the regularized matrix equation~\eqref{eq:RDEL}. However, we encounter a challenge in applying the same procedure to obtain a bound on the solution to the regularized version. The issue arises from the fact that, when the regularization parameter \(\tau\) is strictly positive, \(M^{(\tau)}_{\infty}\) has a positive semidefinite imaginary part, which implies that the function \(z \mapsto M^{(\tau)}(z) - M^{(\tau)}_{\infty}\) may not be Herglotz. One potential alternative approach is to utilize a multivariate Herglotz representation, as discussed in~\cite{Luger_2017}. This representation provides an integral representation for the function \((z, i\tau) \mapsto M^{(\tau)}(z)\) involving a multivariate measure. However, it should be noted that in such representations, the measure cannot be finite unless it is trivial. Nonetheless, an analogue of \Cref{lemma:stieltjes_representation} holds for the upper-left \(n \times n\) block of the solution to the RDEL. The result is obtained via a similar argument, so we omit the proof.

\begin{lemma}\label{lemma:stieltjes_representation_RDEL}
    Fix \(\tau\in \bbR_{\geq 0}\) and assume that \( M^{(\tau)}\in \calM\) such that, for all \(z\in  \bbH\), \( M^{(\tau)}(z)\) solves the RDEL~\eqref{eq:RDEL}. Then, \(M^{(\tau)}_{1,1}(z) =  \int_{ \bbR}\frac{ \Omega^{(\tau)}_{1,1}(\d \lambda)}{\lambda-z}\) for all \(z\in \bbH\), where \(\Omega^{(\tau)}_{1,1}\) is a \(n\times n\) matrix-valued measure on Borel subsets of \(\bbR\) satisfying \(\int_{\bbR}\Omega^{(\tau)}_{1,1}(\d\lambda) =I_{n}\).
\end{lemma} 

\Cref{lemma:stieltjes_representation,lemma:stieltjes_representation_RDEL} are particularly significant because they allow us to treat the solution of the DEL as the limit of solutions to the RDEL as \(\tau\) approaches zero. The tightness of the family of measures induced by the Stieltjes representation of RDEL solutions plays a key role in this step.

\begin{corollary}\label{corollary:tight_measure}
    For every \(\tau\in \bbR_{> 0}\), let \( M^{(\tau)}\in \calM_{+}\) such that, for all \(z\in  \bbH\), \( M^{(\tau)}(z)\) solves the RDEL~\eqref{eq:RDEL} and denote by \(\Omega_{1,1}^{(\tau)}\) the positive semidefinite measure in the Stieltjes transform representation of \(M_{1,1}^{(\tau)}\). Then, for every \(v\in \bbC^{n}\) and \(\tau_{+}\in \bbR_{>0}\), the family of measures \(\{v^{\ast}\Omega_{1,1}^{(\tau)}v \st \tau\in [0,\tau_{+}]\}\) is tight.
\end{corollary}

\begin{proof}
    Let \(\tau\in \bbR_{>0}\). By \cref{lemma:large_z_norm_bound,lemma:large_z_img_bound}, there exists \(\kappa,c,\kappa_{+}\in \bbR_{>0}\) such that \(\|\Im[M^{(\tau)}_{1,1}(z)]\|\leq c (|z|-\kappa)^{-2}(\tau+\Im[z])\) for every \(z\in \bbH\) with \(|z|\geq \kappa_{+}\). Then,
    \[
        \|\Im[M^{(\tau)}_{1,1}(\lambda + i\epsilon)]\| \leq c(\sqrt{\lambda^{2}+\epsilon^{2}}-\kappa)^{-2}(\tau + \epsilon) \leq c(\lambda-\kappa)^{-2}(\tau + \epsilon)
    \]
    for every \(\lambda > \kappa_{+}\) and \(\epsilon \in [0,1]\). Here, \(c\) is
    some constant independent of \(\lambda\) and \(\tau\). Hence, for every \(\lambda_{+}>\kappa_{+}\), according to the Stieltjes inversion formula for \(\Omega^{(\tau)}\) as stated in~\cite[Theorem 5.4(v) and Theorem 5.4(vi)]{matrix_herglotz},
    \begin{align*}
        \Omega^{(\tau)}_{1,1}\left((\lambda_{+},\infty)\right) & \preceq  \pi^{-1}\lim_{\epsilon\downarrow 0}\int_{\lambda_{+}}^{\infty}\Im[M^{(\tau)}_{1,1}(\lambda+i\epsilon)]\d\lambda
        \\    & \preceq \pi^{-1}\lim_{\epsilon\downarrow 0}\int_{\lambda_{+}}^{\infty}\|\Im[M^{(\tau)}_{1,1}(\lambda+i\epsilon)]\|  \d\lambda
        \\ & \preceq c\pi^{-1}\tau\int_{\lambda_{+}}^{\infty} (\lambda-\kappa)^{-2}   \d\lambda.
    \end{align*}
    Therefore, if \(\tau\) is bounded, we may pick \(\lambda_{+}>\kappa_{+}\) arbitrarily large to ensure that \(\int_{\lambda_{+}}^{\infty} (\lambda-\kappa)^{-2}  \d\lambda\) is arbitrarily small.
\end{proof}

\subsection{Power Series Representation}\label{sec:power_series}

As the set of admissible solutions \(\calM\) comprises analytic matrix-valued functions, any solution to equation~\eqref{eq:DEL} can be locally expressed as a power series. Utilizing the Stieltjes transform representation provided in \Cref{lemma:stieltjes_representation}, we can derive a solvable recurrence relation that determines the coefficients in such an expansion. This recurrence relation facilitates the systematic computation of the coefficients in the power series representation of the solution.

\begin{lemma}\label{lemma:power_series}
    Let \(M\in \calM\) such that, for all \(z\in  \bbH\), \( M(z)\) solves the DEL~\eqref{eq:DEL} and let \(\Omega\) be the positive semidefinite measure in \Cref{lemma:stieltjes_representation}. Then, there exists \(\lambda_{+}> \sup\{|\lambda|\st \lambda \in \supp(\Omega)\}\) such that
    \[
        M(z)=\sum_{j=0}^{\infty}z^{-j}M_{j} = \left(\bbE L-z\Lambda\right)^{-1}\sum_{j=0}^{\infty}\left(\sum_{k=0}^{\infty} z^{-k}\supop(M_{k})(\bbE L-z\Lambda)^{-1}\right)^{j}
    \]
    for every \(z\in \bbH\) with \(|z|\geq \lambda_{+}\). Here, \(M_{0}=M_{\infty}\) and \(M_{j}=-\int_{\bbR}\lambda^{j-1}\Omega(\d \lambda)\) for every \(j\in \bbN\).
\end{lemma}
\begin{proof}
    Since \(\Omega\) is compactly supported by \Cref{lemma:stieltjes_representation}, \(\sup\{|\lambda|\st \lambda \in \supp(\Omega)\}\) is finite. Let \(z\in \bbH\) with \(|z|>\sup\{|\lambda| \st \lambda\in \supp(\Omega)\}\) and write
    \[
        M(z) =  M_{\infty} +  \int_{ \bbR}\frac{ \Omega(\d \lambda)}{\lambda-z} = M_{\infty} -  z^{-1}\int_{\bbR}\frac{ \Omega(\d \lambda)}{1-\lambda/z}.
    \]
    We recognize \((1-\lambda/z)^{-1}\) as a geometric series and write \((1-\lambda/z)^{-1}=\sum_{j=0}^{\infty}\frac{\lambda^{j}}{z^{j}}\). By Fubini's theorem,
    \[
        \int_{\bbR}\frac{ \Omega(\d \lambda)}{1-\lambda/z} = \sum_{j=0}^{\infty}z^{-j}\int_{\bbR}\lambda^{j}\Omega(\d \lambda)
    \]
    which implies that
    \[
        M(z) =  M_{\infty}-\sum_{j=0}^{\infty}z^{-j-1}\int_{\bbR}\lambda^{j}\Omega(\d \lambda).
    \]
    On the other hand, by definition, \(M(z)\) solves~\eqref{eq:DEL}, and we may write \(M(z) = \F(M(z)) = ( \bbE L - \supop( M(z)) - z\Lambda)^{-1}\). Using \Cref{lemma:block_inversion},
    \begin{equation}\label{eq:res_expect}
        (\bbE L - z\Lambda)^{-1} =
        \begin{bmatrix}
            R                  & -R \bbE[B^{\top}]\bbE[Q]^{-1}                     \\
            -\bbE[Q]^{-1} \bbE [B] R & \bbE[Q]^{-1} + \bbE[Q]^{-1}\bbE[B]R\bbE[B^{\top}]\bbE[Q]^{-1}
        \end{bmatrix}
    \end{equation}
    with \(R=(\bbE [B^{\top}](\bbE Q)^{-1}\bbE [B]-\bbE[A]-zI_n)^{-1}\). Since
    \[
        \|R\|\leq \mathrm{dist}(z,\sigma\left(\bbE[A]-\bbE [B^{\top}]Q^{-1}\bbE [B]\right))^{-1},
    \]
    we obtain \((\bbE L - z\Lambda)^{-1} \xrightarrow[]{} M_{\infty}\) as \(|z|\to\infty\). Because \(M_{\infty}\) is non-zero only in its lower-right \(d\times d\) block, it follows from \Cref{lemma:large_z_norm_bound} and the flatness of the superoperator that \( \|\supop( M(z))(\bbE L - z\Lambda)^{-1}\|\to 0\) as \(|z|\to\infty\). Let \(\lambda_{+}> \max\{|\lambda|\st \lambda \in \supp(\Omega)\}\) such that \(\|\supop( M(z))(\bbE L - z\Lambda)^{-1}\| <1\) for all \(z\in \bbH\) with
    \(|z|\geq \lambda_{+}\). Then, \(I_{\ell} - \supop( M(z))(\bbE L -
    z\Lambda)^{-1}\) is non-singular with Neumann series
    \[
        \left(I_{\ell}-\supop( M(z))(\bbE L - z\Lambda)^{-1}\right)^{-1} = \sum_{j=0}^{\infty} \left(\supop( M(z))(\bbE L - z\Lambda)^{-1}\right)^{j}.
    \]
    In particular,
    \[
        \left( \bbE L - \supop( M(z)) - z\Lambda\right)^{-1} = (\bbE L - z\Lambda)^{-1}\sum_{j=0}^{\infty} \left(\supop( M(z))(\bbE L - z\Lambda)^{-1}\right)^{j}.
    \]
    We obtain the result by plugging the series expansion for \(M(z)\) and using linearity of the superoperator.
\end{proof}

\subsection{Existence and Uniqueness}\label{sec:existence_uniqueness}

We now establish the existence and uniqueness of a solution to the DEL~\eqref{eq:DEL}. As discussed above, we begin by proving the existence of a solution to the RDEL~\eqref{eq:RDEL} for every \(\tau \in \mathbb{R}_{>0}\), followed by employing a continuity argument to consider the vanishing regularization.

\subsubsection{Solution to the Regularized Dyson Equation}\label{sec:existence_RDEL}

For every \(\tau \in \mathbb{R}_{>0}\), the existence of a unique \(M \in \mathcal{M}_{+}\) satisfying~\eqref{eq:RDEL} for every \(z \in \mathbb{H}\) follows directly from~\cite{helton2007operatorvalued}. At a high level, the proof of the existence of a solution to the RDEL~\eqref{eq:RDEL} proceeds by demonstrating that the map \(M \mapsto \F^{(\tau)}(M)\) is strictly holomorphic on bounded domains of \(\mathcal{M}_{+}\). By the Earle-Hamilton fixed-point theorem, strict holomorphicity implies that the mapping is contractive with respect to the CRF-pseudometric. For further details on the CRF-pseudometric, we refer the reader to \Cref{subsec:CRF}. 

Define
\begin{equation}\label{eq:admissible_func_b}
    \calM_{b}:=\Hol(\bbH \cap b\bbD,\calA_{b}),\quad \text{ and }\quad   \calA_{b}:=\calA_{+}\cap \{W\in \bbC^{\ell\times \ell}\st \|W\|< b\}
\end{equation}
for every \(b>0\). Indeed, for every \(b\in \bbR_{>0}\), \(\calM_{b}\) is a domain in the Banach space of matrix-valued bounded holomorphic functions on \(\bbH\) with the canonical supremum norm. Also, \(\calA_{b}\) is a domain in the Banach space of complex symmetric \(\ell \times \ell\) matrices with the operator norm. Using the result above, we can show that \(\Ftau\) is indeed a strict holomorphic function on \(\calA_{b}\) for every \(\tau\in \bbR_{>0}\). The following lemma is a direct adaptation of~\cite[Proposition 3.2]{helton2007operatorvalued}.

\begin{lemma}\label{lemma:strict_holomorphic}
    Let \(z\in \bbH\), \(\tau,b\in \bbR_{>0}\) and define \(m_{b}=\|\bbE L\|+sb+|z|+\tau\). Then, for every \(W\in \calA_{b}\), \(\|\Ftau(W)\|\leq \tau^{-1}\) and \(\Im[\Ftau(W)]\succeq \tau  m_{b}^{-2}I_{\ell} \succ 0\). In particular, if \(b>\tau^{-1}\), then \(\Ftau\) maps \(\calM_{b}\) strictly
    into itself.
\end{lemma}
\begin{proof}
    Let \(W\in \calA_{b}\). By \Cref{lemma:prior_bounds}, \(\|\Ftau(W)\|\leq \tau^{-1}\) and
    \[
        \Im[\Ftau(W)]\succeq \tau \Ftau( W) [\Ftau( W)]^{\ast}.
    \]
    Let \(v\in \bbC^{\ell}\) such that \(\|v\|=1\). By Cauchy-Schwarz inequality,
    \[
        1=v^{\ast}(\Ftau( W))^{-1}\Ftau( W)v \leq \|\Ftau( W)v\|\|(\Ftau( W))^{-\ast}v\|
    \]
    which implies that \(\|(\Ftau( W))^{-1}\|^{-2}\leq
    \|(\Ftau( W))^{-\ast}v\|^{-2}\leq \|\Ftau( W)v\|^{2}\).
    Additionally,
    \[
        \|(\Ftau( W))^{-1}\| = \| \bbE L - \supop(W)- z\Lambda-i\tau I_{\ell}\| \leq m_{b}.
    \]
    Thus, \(\Im[\Ftau(W)]\succeq \tau m_{b}^{-2}I_{\ell}\).
\end{proof}

The existence of a unique solution to~\eqref{eq:RDEL} then follows directly from an application of the Earle-Hamilton fixed-point theorem~\cite{earl_hamilton}. Indeed, for every \(b\in\bbR_{>0}\) large enough, \(\Ftau\) has exactly one fixed point on \(\calM_{b}\). Since \(\calM_{+}=\bigcup_{b\in \bbR_{>0}}\calM_{b}\), we obtain the following result.
\begin{proposition}[{\cite[Theorem 2.1]{helton2007operatorvalued}}]\label{proposition:existenceuniquness_RDEL}
    There exists a unique solution \(M\in \calM_{+}\) such that \(M^{(\tau)}(z)\) solves~\eqref{eq:RDEL} for every \(\tau\in \bbR_{>0}\) and \(z\in \bbH\). Furthermore, for every \(W_{0}\in \calM_{+}\), the iterates \(W_{k+1}=\Ftau(W_{k})\) converge in norm to \(M^{(\tau)}\).
\end{proposition}

In what follows, we will denote the unique solution of the RDEL with \(\tau \in \bbR_{>0}\) by \(M^{(\tau)}\). While not explicitly stated, the analyticity of \(M^{(\tau)}(z)\) in \(\tau\) can be inferred using an implicit function theorem, as demonstrated in~\cite[Theorem 2.14]{nemish_local_2020}. We state this result in the following lemma but omit the proof, directing the reader to the aforementioned reference for further details.

\begin{lemma}\label{lemma:analyticity_tau}
    For every \(z\in \bbH\), the map \(\tau \in \bbH \mapsto M^{(\tau)}(z)\) is analytic.
\end{lemma}

\subsubsection{Solution to the Dyson Equation for Linearizations}\label{sec:existence_DEL}

We now establish the existence and uniqueness of a solution to the DEL~\eqref{eq:DEL}.

\begin{lemma}[Existence and Uniqueness]\label{lemma:existenceuniquness}
    There exists a unique matrix-valued function \(M\in \calM\) such that \(M(z)\) solves the DEL~\eqref{eq:DEL} for every \(z\in \bbH\).
\end{lemma}

\begin{proof}
    For every \(k\in \bbN\), let \(M^{(k^{-1})}\) be the unique solution to the RDEL and write
    \[
        M_{1,1}^{(k^{-1})}(z)= \int_{ \bbR}\frac{ \Omega_{1,1}^{(k^{-1})}(\d \lambda)}{\lambda-z}
    \]
    the Stieltjes transform representation guaranteed by
    \Cref{lemma:stieltjes_representation_RDEL}. Additionally, let \(\{v_{j}\st j\in\bbN\}\subseteq \bbC^{n}\) be a countable dense subset of the ball of \(n\)-dimensional complex unit vectors.

    By \Cref{corollary:tight_measure}, the family of measures \(\{v_{1}^{\ast}\Omega_{1,1}^{(k^{-1})}v_{1}\st k\in \bbN\}\) is tight. Consequently, by Prokhorov's theorem, there exists a probability measure \(\omega_{1}\) and a subsequence \(\{\tau_{1,k}\st k\in \bbN\}\subseteq \{k^{-1}\st k\in \bbN\}\) such that \(v_{1}^{\ast}\Omega_{1,1}^{(\tau_{1,k})}v_{1}\) converges weakly to \(\omega_{1}\) as \(k\) approaches infinity.

    We now proceed inductively. Assume that there exists \(m\in \bbN\) and a collection of compactly supported measures \(\{\omega_{j}\st 1\leq j\leq m \}\) such that \(v_{j}^{\ast}\Omega_{1,1}^{(\tau_{m,k})}v_{j}\) converges weakly to \(\omega_{j}\) for all \(1\leq j\leq m\) as \(k\) approaches infinity. By \Cref{corollary:tight_measure} and Prokhorov's theorem, there exists a probability measure \(\omega_{m+1}\) and a subsequence \(\{\tau_{m+1,k}\st k\in\bbN\}\subseteq \{\tau_{m,k}\st k\in \bbN\}\) such that \(v_{m+1}^{\ast}\Omega_{1,1}^{(\tau_{m+1,k})}v_{m+1}\) converges weakly to \(\omega_{m+1}\) as \(k\) approaches infinity. Also, by construction of the subsequence, \(v_{j}^{\ast}\Omega_{1,1}^{(\tau_{m+1,k})}v_{j}\) converges weakly to \(\omega_{j}\) for all \(1\leq j\leq m+1\) as \(k\) approaches infinity.

    Let \(\tau_{k}=\tau_{k,k}\) for all \(k\in \bbN\). By construction, \(v_{j}^{\ast}\Omega^{(\tau_{k})}v_{j}\) converges weakly to a probability measure \(\omega_{j}\) for every \(j\in \bbN\) as \(k\to\infty\). Furthermore, by \Cref{lemma:prior_bounds}, \(\{M_{1,1}^{(\tau_{k})}\st k\in \bbN\}\) is a locally uniformly bounded sequence of analytic functions. Hence, Montel's theorem guarantees the existence of a subsequence of \(\{\tau_{k}\st k\in \bbN\}\), which we will assume WLOG to be \(\{\tau_{k}\st k\in \bbN\}\) for notational convenience, such that \(M_{1,1}^{(\tau_{k})}\) converges to an analytic function \(M_{1,1}\). By the proof of \Cref{corollary:tight_measure}, there exists \(\kappa_{+}\in \bbR_{>0}\) and a constant \(c\in \bbR_{>0}\) such that
    \[
        \int_{\lambda_{+}}^{\infty}\omega_{j}(\d\lambda)  = \lim_{k\to\infty}\int_{\lambda_{+}}^{\infty}v^{\ast}_{j}\Omega^{(\tau_{k})}_{1,1}(\d\lambda) v_{j} \leq c\lim_{k\to\infty} \tau_{k} \int_{\lambda_{+}}^{\infty} (\lambda-\kappa)^{-2} \d\lambda = 0
    \]
    for every \(\lambda_{+} \geq \kappa_{+}\) and \(j\in\bbN\). By \Cref{lemma:stieltjes_representation_RDEL},
    \[
        v_{j}^{\ast}\Im[M_{1,1}]v_{j} = \lim_{k\to\infty}v_{j}^{\ast}\Im[M^{(\tau_{k})}_{1,1}]v_{j}  = \Im[z]\int_{\bbR}\frac{\omega_{j}(\d\lambda)}{|\lambda-z|^{2}}.
    \]
    Since \(\omega_{j}\) is a probability measure,
    \[
        \int_{\bbR}\frac{\omega_{j}(\d \lambda)}{|\lambda-z|^{2}} = \int_{[-\kappa_{+},\kappa_{+}]}\frac{\omega_{j}(\d \lambda)}{|\lambda-z|^{2}}  \geq \left(\max_{\lambda\in [-\kappa_{+},\kappa_{+}]}|\lambda-z|\right)^{-2}
    \]
    which implies that \( v_{j}^{\ast}\Im[M_{1,1}]v_{j} \geq \Im[z]\left(\max_{\lambda\in [-\kappa_{+},\kappa_{+}]}|\lambda-z|\right)^{-2}\) for every \(j\in \bbN\). Fix \(z\in \bbH\), \(\epsilon = 3^{-1}(\Im[z])^{2}\left(\max_{\lambda\in
        [-\kappa_{+},\kappa_{+}]}|\lambda-z|\right)^{-2}\in \bbR_{>0}\). Let \(v\in
    \bbC^{n}\) be any unit vector and let \(j\in \bbN\) such that
    \(\|v-v_{j}\|\leq \epsilon\). Then,
    \[
        v^{\ast}\Im[M_{1,1}]v \geq v^{\ast}\Im[M_{1,1}]v - 2\|v_{j}-v\|\|M_{1,1}\|\|v\|  \geq \frac{\epsilon}{3\Im[z]} > 0.
    \]
    In particular, \(\Im[M_{1,1}(z)]\succ 0\) for all \(z\in \bbH\).

    Define \(M_{1,2}, M_{2,1}\) and \(M_{2,2}\) as functions of \(M_{1,1}\) using~\eqref{eq:M12_schur},~\eqref{eq:M21_schur} and~\eqref{eq:M22_schur_v2} respectively and let \(M\) be the block matrix with \(j,k\) block given by \(M_{j,k}\) for all \((j,k)\in \{1,2\}^{2}\). It follows from \Cref{lemma:invertible}, that \( \bbE Q-\supop_{2,2}(M)\) is non-singular and that \(M\) is well-defined. By construction, it is clear that \(M\in \calM\) and that \(M(z)\) solves~\eqref{eq:DEL} for all \(z\in \bbH\).

    Uniqueness of the solution follows from analycity and the power series representation in \Cref{lemma:power_series}. Let \(\lambda_{+}\in\bbR_{>0}\) such that
    \[
        M(z)=\sum_{j=0}^{\infty}z^{-j}M_{j} = \left(\bbE L-z\Lambda\right)^{-1}\sum_{j=0}^{\infty}\left(\sum_{k=0}^{\infty} z^{-k}\supop(M_{k})(\bbE L-z\Lambda)^{-1}\right)^{j}
    \]
    for every \(z\in \bbH\) with \(|z|\geq \lambda_{+}\).

    Since resolvent of Hermitian matrices are analytic when the spectral parameter is away from the support, is follows from the decomposition in~\eqref{eq:res_expect} that \((\bbE L - z\Lambda)^{-1}\) is analytic. Write \((\bbE L - z\Lambda)^{-1} = \sum_{j=0}^{\infty}z^{-j}C_{j}\) for some complex matrices \(\{C_{j}\st j\in \bbN_{\geq 0}\}\subseteq \bbC^{\ell\times\ell}\). Plugging this in the power series expansion of \(M\) and gathering coefficients of \(z^{-1}\), we get that \(M_{1} = C_{1}+ C_{0}\supop(M_{0})C_{1}+C_{0}\supop(M_{1})C_{0}\). We computed in \Cref{lemma:power_series} that \((\bbE L - z\Lambda)^{-1}\to M_{\infty}\) as \(|z|\to \infty\) and similarly for \(M(z)\). In other words, \(C_{0}=M_{0}=M_{\infty}\). Looking at the structure of the superoperator, \(\supop_{2,2}(M_{0})=0\), which gives us \(C_{0}\supop(M_{0})=0\). In particular, \(M_{1} = C_{1}+ C_{0}\supop(M_{0})C_{1}\) is expressible solely in terms of \(C_{0}\) and \(C_{1}\).

    This proves the base case. Let \(k\in \bbN\) and assume that \(\{M_{j}\st j\in \{0,1,\ldots, k\}\}\) are fully determined by \(\{C_{j}\st j\in \bbN_{\geq 0}\}\). Gathering the coefficients for \(z^{-(k+1)}\) in the power series expansion, we get that
    \[
        M_{k+1} = f(M_{0},M_{1},\ldots, M_{k})+C_{0}\supop(M_{k+1})C_{0}
    \]
    for some analytic function \(f\). By induction hypothesis, \(f(M_{0},M_{1},\ldots, M_{k})\) may be expressed as an analytic function of \(\{C_{j}\st j\in \bbN_{\geq 0}\}\). Furthermore, since \(C_{0}=M_{\infty}\) is \(0\) everywhere
    outside its lower \(d\times d\) block,
    \[
        C_{0}\supop(M_{k+1})C_{0} =
        \begin{bmatrix}
            0_{n\times n} & 0_{n\times d}                       \\
            0_{d\times n} & M_{\star} \supop_{2,2}(M_{k+1})M_{\star}.
        \end{bmatrix}
    \]
    Therefore, extracting the upper-left \(n\times n\) block, we obtain that the upper-left \(n\times n\) block along with both off-diagonal blocks of \(M_{k+1}\) are determined by the coefficient matrices \(\{C_{j}\st j\in\bbN_{\geq 0}\}\). Since \( \supop_{2,2}(M)\) does not depend on the lower-right block of \(M\), we may also determine the lower-right block of \(M_{k+1}\). Inducting, we get that any two solution to~\eqref{eq:DEL} must be equal for all \(z\in \bbH\) with \(|z|>\lambda_{+}\) for some \(\lambda_{+}\in \bbR_{>0}\). By analytic continuation, it follows that any two solution must be equal for all \(z\in \bbH\).
\end{proof}

The rationale behind our choices regarding the settings is now quite evident. By excluding \(\tilde{\supop}\) from the superoperator, we gain the advantage that each block in the block decomposition of the DEL can be determined by the upper-left \(n\times n\) block. This upper-left block exhibits favorable properties, including an a priori norm bound due to the position of the spectral parameter. By leveraging these properties, we establish the existence of a solution for the upper-left block of~\eqref{eq:DEL}. Subsequently, we utilize the existence of this sub-solution to construct the remaining part of the solution. However, our selection of superoperator also imposes implicit restrictions on the distribution of the entries of \(L\).

In contrast to the assurance provided by \Cref{proposition:existenceuniquness_RDEL}, it is crucial to acknowledge that \Cref{lemma:existenceuniquness} does not generally guarantee pointwise convergence for the fixed-point iteration \(M_{k+1} = \F(M_{k})\) with an initial condition \(M_{0}\in \calM\) to the solution of the Dyson equation.

\section{Proof of the Global Anisotropic Law}\label{sec:proof_global_anisotropic}

This section is dedicated to proving~\Cref{theorem:convergence}. The proof is divided into several parts, as outlined in~\eqref{eqs:comparisons} following the theorem's statement.

\subsection{Regularization}\label{sec:regularization}

It will be easier to work with the regularized expected pseudo-resolvent \((L-z\Lambda -i\tau I_{\ell})^{-1}\) instead of the pseudo-resolvent \((L-z\Lambda)^{-1}\). The next lemma provides a bound on the difference between the two.

\begin{lemma}\label{lemma:Ftau_vs_F}
    For every \(\tau\in \bbR_{\geq 0}\) and \(z\in \bbH\), \(\|(L-z\Lambda -i\tau I_{\ell})^{-1} - (L-z\Lambda)^{-1}\| \leq \tau \|(L-z\Lambda)^{-1}\|^{2}\).
\end{lemma}
\begin{proof}
    By \Cref{lemma:res_trick}, \(\|(L-z\Lambda -i\tau I_{\ell})^{-1} - (L-z\Lambda)^{-1}\| \leq \tau \|(L-z\Lambda -i\tau I_{\ell})^{-1}\|\|(L-z\Lambda)^{-1}\|\).
    Let \(v\in \bbC^{\ell}\) be arbitrary and decompose \(L - z\Lambda-i\tau I_{\ell} = X+iY-i\tau I_{\ell}\) with \(X=\Re[L - z\Lambda]\) and \(Y=\Im[L - z\Lambda]\). Then, using the fact that \((L - z\Lambda-i\tau I_{\ell})^{\ast}=X-iY+i\tau I_{\ell}\) and \(\Im[Y]\preceq 0\),
    \begin{align*}
          v^{\ast}\left(X+iY-i\tau I_{\ell}\right)^{\ast}\left(X+iY-i\tau I_{\ell}\right)v
         & \geq v^{\ast}(X+iY)^{\ast}(X+iY)v + \tau^{2}v^{\ast}v- 2\tau v^{\ast}Yv 
        \\ & \geq v^{\ast}(X+iY)^{\ast}(X+iY)v. 
    \end{align*}
    Because taking the inverse reverses the Loewner partial ordering, it follows that \(\|(L-z\Lambda -i\tau I_{\ell})^{-1}\|\leq \|(L-z\Lambda)^{-1}\|\).
\end{proof}

By \Cref{lemma:Ftau_vs_F} and Jensen's inequality, the expected pseudo-resolvent \(\bbE(L-z\Lambda)^{-1}\) is well-approximated by its regularized version for small \(\tau\) as long as \(\bbE\|(L-z\Lambda)^{-1}\|^{2}\) is bounded. In fact, if the norm squared of the expected pseudo-resolvent is bounded in expectation, then the regularized expected pseudo-resolvent converges to the expected pseudo-resolvent in operator norm as \(\tau\) approaches zero from above, uniformly in the dimension. As we expect the solution of~\eqref{eq:DEL} to be a deterministic equivalent for the (expected) pseudo-resolvent, we should anticipate a similar behavior from it and hence we suppose that \cref{condition:convergence_RDEL} holds. Together, \cref{condition:convergence_RDEL} and \Cref{lemma:Ftau_vs_F} control~\eqref{eq:comparison2} and~\eqref{eq:comparison4}.

\subsection{Stability}\label{sec:stability}

In this section, we establish the asymptotic stability of the RDEL~\eqref{eq:RDEL} with respect to small additive perturbations. In order to maintain a certain level of generality, we will fix \(z\in \bbH\) and consider matrices in \(\calA_{b}\) for some \(b\in \bbR_{>0}\), as defined in~\eqref{eq:admissible_func_b}. When considering the class of matrices \(\calA_{b}\), it is helpful to keep in mind the (expected) regularized pseudo-resolvent.

\begin{lemma}\label{lemma:Ftau_admissible}
    Fix \(z\in \bbH\) and let \(\tau,b\in \bbR_{>0}\) with \(\tau^{-1}<b\). Then, \((L-z\Lambda-i\tau I_{\ell})^{-1},\bbE(L-z\Lambda-i\tau I_{\ell})^{-1}\in \calA_{b}\).
\end{lemma}
\begin{proof}
    By \Cref{lemma:real/imag_singularity_bound} and Jensen's inequality, \(\|(L-z\Lambda -i\tau I_{\ell})^{-1}\|\vee \|\bbE (L-z\Lambda-i\tau I_{\ell})^{-1}\|\leq \tau^{-1}\).  Furthermore, by \Cref{lemma:real/imag_inverse}, we can follow a similar argument as in the proof of \Cref{lemma:strict_holomorphic} to obtain \(\Im[(L-z\Lambda -i\tau I_{\ell})^{-1}]\succeq \tau (L-z\Lambda -i\tau I_{\ell})^{-1}(L-z\Lambda -i\tau I_{\ell})^{-\ast}\succeq \tau \|L-z\Lambda -i\tau I_{\ell}\|^{-2}\succeq \tau (\|L\| + |z| +\tau)^{-2}\succ 0\). By monotonicity of the expectation, \(\Im[\bbE(L-z\Lambda -i\tau I_{\ell})^{-1}]\succeq  \tau \bbE (\|L\| + |z| +\tau)^{-2}\). Since the function \(x\mapsto x^{-2}\) is convex for \(x\in \bbR_{>0}\), we may apply Jensen's inequality to obtain \(\Im[\bbE(L-z\Lambda -i\tau I_{\ell})^{-1}]\succeq \tau(\bbE\|L\| + |z| +\tau)^{-2}\succ 0\).
\end{proof}

Fix \(z\in \bbH\), let \(b,\tau\in \bbR_{>0}\) and assume that \(F^{(\tau)}\in \calA_{b}\) satisfies
\begin{equation}\label{eq:F_RDEL}
    (\bbE L - \supop(F^{(\tau)})-z\Lambda -i\tau I_{\ell})F^{(\tau)}(z) = I_{\ell}+D^{(\tau)},
\end{equation}
where \(D^{(\tau)}\) is a perturbation term. In particular, \(F^{(\tau)}\) almost solves~\eqref{eq:RDEL} up to an additive perturbation term \(D^{(\tau)}\). For a fixed \(z\in \bbH\), let \(E_{\tau}=\Ftau(F^{(\tau)})D^{(\tau)}\) for every \(F\in \calA_{b}\) and \(\tau\in \bbR_{> 0}\), defining the \emph{error matrix}, and \(\epsilon_{\tau}=\|E_{\tau}\|\) representing the \emph{magnitude of the error} at \(\tau\). The objective for the rest of this section is to show that if \(\epsilon_{\tau}\) is small, then \(F^{(\tau)}\) is close to \(M^{(\tau)}\). We will establish this result using properties of the CRF-pseudometric introduced in \Cref{subsec:CRF}.

Before stating the first lemma, we recall from~\eqref{eq:admissible_func_b} and~\eqref{lemma:strict_holomorphic} that \(\calA_{b}:=\calA_{+}\cap \{W\in \bbC^{\ell\times \ell} \st \|W\|< b\}\) is a domain in the Banach space of \(\ell\times \ell\) complex matrices for every \(b\in \bbR_{>0}\). The CRF-pseudometric is a crucial tool because \(\Ftau\) is a strict contraction on \(\calA_{b}\) with respect to the CRF-pseudometric. This can be observed and quantified by combining \Cref{lemma:strict_holomorphic} with the proof of the Earle-Hamilton fixed-point theorem~\cite{earl_hamilton}.

\begin{lemma}\label{lemma:strict_contraction}
    Fix \(z\in \bbH\) and \(\tau\in \bbR_{>0}\). For every \(b\in \bbR_{>0}\), let \(m_{b}=\|\bbE L\|+s b+|z|+\tau+1\) and \(\delta=(m_{b}^{2}\tau^{-2}-1)^{-1}\). Suppose that \(\tau^{-1}(1+2\delta)<b\) and let \(\rho\) denotes the CRF-pseudometric on \(\calA_{b}\). Then, for every \(X,Y\in \calA_{b}\), \(\rho(\Ftau(X),\Ftau(Y))\leq (1+\delta)^{-1}\rho\left(X,Y\right)\).
\end{lemma}
\begin{proof}
    Let \(X\in \calA_{b}\) and define \(\G:Y\in \calA_{b}\mapsto \Ftau(Y) + \delta (\Ftau(Y)-\Ftau(X))\). By \Cref{lemma:strict_holomorphic}, \(\|\Ftau(Y)\|\leq \tau^{-1}\) and \(\Im[\Ftau(Y)]\succ \tau m_{b}^{-2}\) for every \(Y\in \calA_{b}\). Hence, \(\|\G(Y)\| \leq \tau^{-1}+2\delta\tau^{-1}\) and \(\Im[\G(Y)]\succeq (1+\delta)\Im[\Ftau(Y)]-\delta\|\Ftau(X)\| \succ  (1+\delta)\tau m_{b}^{-2} - \delta \tau^{-1}\). By our choice of \(\delta\) and \(b\), \(\|\G(Y)\|<b\) and \(\Im[G(Y)]\succ 0\) for every \(Y\in \calA_{b}\). In fact, \(\G\) is a strict holomorphic function on \(\calA_{b}\). By the proof of the Earle-Hamilton fixed-point theorem in~\cite[Theorem 3.1]{harris_fixed_2003} and~\cite[Theorem 4]{HARRIS1979345}, we obtain \(\rho(\Ftau(X),\Ftau(Y))\leq (1+\delta)^{-1}\rho(X,Y)\).
\end{proof}

While the DEL map exhibits favorable properties with respect to the CRF-pseudometric, we would like to know whether the CRF-pseudometric captures the convergence relevant to the problem at hand. Specifically, we aim to understand the behavior of the operator norm with respect to the CRF-pseudometric. The subsequent lemma demonstrates that the CRF-pseudometric can be employed to establish convergence in terms of generalized trace entries.

\begin{lemma}\label{lemma:CRF_norm_lower}
    Fix \(z\in \bbH\) and \(\tau\in \bbR_{>0}\). Let \(b\in \bbR_{>0}\) with \(b>\tau^{-1}\), \(F\in\calA_{b}\) and \(\rho\) be the CRF-pseudometric on \(\calA_{b}\). Then, \(\tr(U(M^{(\tau)}-F)) \leq (b+\tau^{-1})\tanh(\rho(M^{(\tau)},F))\) for every \(U\in \bbC^{\ell\times \ell}\) with \(\|U\|_{\ast}\leq 1\).
\end{lemma}
\begin{proof}
    Let \(m_{b}\) be defined as in the proof of \Cref{lemma:strict_contraction} and recall that \(\|M^{(\tau)}\|\leq \tau^{-1}\) as well as \(\Im[M^{(\tau)}]\succ \tau m_{b}^{-2}\). In particular, \(M^{(\tau)},F\in \calA_{b}\). Define the holomorphic function \(f:X\in \calA_{b}\mapsto \tr(U(X-M^{(\tau)}))(b+\tau^{-1})^{-1}\in \bbD\). By \Cref{lemma:trace_norm}, \(|\tr(U(X-M^{(\tau)}))|\leq \|X-M^{(\tau)}\| < b+\tau^{-1}\), which ensures that \(f\) is well-defined. By \Cref{prop:hol_inv} and \eqref{eq:poincare},
    \[
        \mathrm{arctanh}\left|\frac{\tr(U(X-M^{(\tau)}))}{(b+\tau^{-1})}\right|=\rho_{\Delta}\left(f(M^{(\tau)}),f(X)\right)\leq \rho\left(M^{(\tau)},X\right).
    \]
    Using the fact that the hyperbolic tangent is increasing, we obtain the result.
\end{proof}

Since the dual norm of the operator norm is the nuclear norm, \Cref{lemma:CRF_norm_lower} implies that the CRF-pseudometric captures the convergence of the operator norm.

\begin{corollary}\label{corollary:CRF_norm_lower}
    Under the settings of \Cref{lemma:CRF_norm_lower},
    \[
        \|M^{(\tau)}-F\| \leq (b+\tau^{-1})\tanh(\rho(M^{(\tau)},F)).
    \]
\end{corollary}

In \Cref{proposition:existenceuniquness_RDEL}, we established that the solution to~\eqref{eq:RDEL} can be obtained using a fixed-point iteration scheme. Using this idea, we will recursively define a sequence of matrices and use the contraction property in \Cref{lemma:strict_contraction} to control the distance between \(M^{(\tau)}(z)\) and \(F\in \calA_{b}\) in the CRF-pseudometric. Since the CRF-pseudometric dominates the operator norm, we will obtain convergence in norm. The only remaining ingredient is a loose control of the CRF-pseudometric. In fact, while the norm \(\|M^{(\tau)}(z)\|\) may be easily bounded uniformly in \(\ell\), transferring this bound to the CRF-pseudometric poses additional difficulties which we address in the following lemma.

\begin{lemma}\label{lemma:CRF_eps}
    Let \(z\in \bbH\) and \(\tau,b\in \bbR_{>0}\) such that \(b > \tau^{-1}+\tau m_{b}^{-2}\). Additionally, let \(F^{(\tau)}\in\calA_{b}\) satisfying~\eqref{eq:F_RDEL} with \(\epsilon_{\tau}<\tau m_{b}^{-2}\), \(\rho\) be the CRF-pseudometric on \(\calA_{b}\) and \(m_{b}=\|\bbE L\|+s b+|z|+\tau+1\) be defined as in \Cref{lemma:strict_contraction}. Then, \(\rho(\Ftau(F^{(\tau)}),F^{(\tau)}) \leq \mathrm{arctanh}(\epsilon_{\tau}m_{b}^{2}/\tau )\).
\end{lemma}
\begin{proof}
    We may assume WLOG that \(\Ftau(F^{(\tau)})\neq F^{(\tau)}\), as otherwise the claim is trivial. By \Cref{lemma:strict_holomorphic}, \(\|\Ftau(F^{(\tau)})\|\leq \tau^{-1}\) and \(\Im[\Ftau(F^{(\tau)})]\succ \tau m_{b}^{-2}\). Define the holomorphic function
    \[
        g:w\in \bbD \mapsto \Ftau(F^{(\tau)}) + \frac{w\tau m_{b}^{-2}}{\|\Ftau(F^{(\tau)})-F^{(\tau)}\|}(F^{(\tau)}-\Ftau(F^{(\tau)}))\in \calA_{b}.
    \]
    Then, it is straightforward to check that \(\|g(w)\|\leq \tau^{-1}+\tau m_{b}^{-2}<b\) and \(\Im[g(w)] \succ 0\) for every \(w\in \bbD\). 
    
    Let \(\rho_{\bbD}\) denote the CRF-pseudometric on \(\bbD\). Using~\eqref{eq:F_RDEL}, we may write \(\Ftau(F)-F^{(\tau)}=-E_{\tau}=-\Ftau(F^{(\tau)})D^{(\tau)}\) which implies that \(\|\Ftau(F)-F^{(\tau)}\|\leq \epsilon_{\tau}\). Hence, by \Cref{prop:hol_inv},
    \begin{align*}
        \rho(\Ftau(F^{(\tau)}),F^{(\tau)}) & = \rho\left(g(0),g\left(\frac{\|\Ftau(F)-F^{(\tau)}\|}{\tau m_{b}^{-2}}\right)\right) 
         \leq \rho_{\bbD}\left(0,\frac{\|\Ftau(F)-F^{(\tau)}\|}{\tau m_{b}^{-2}}\right).
    \end{align*}
    By~\eqref{eq:poincare}, \(\rho_{\bbD}(0,\|\Ftau(F)-F^{(\tau)}\|m_{b}^{2}/\tau)\leq \mathrm{arctanh}(\epsilon_{\tau}m_{b}^{2}/\tau )\).
\end{proof}

\Cref{lemma:CRF_eps} controls the discrepancy between the matrix \(F^{(\tau)} \in \calA_{b}\), which approximately solves the RDEL up to an additive perturbation term \(D^{(\tau)}\), before and after applying the RDEL map once, in terms of the magnitude of the error \(\epsilon_{\tau}\). 

Combining \cref{lemma:CRF_eps,lemma:CRF_norm_lower,lemma:strict_contraction}, we obtain the main stability result.

\begin{lemma}\label{lemma:stability}
    Fix \(z\in \bbH\), \(\tau\in \bbR_{>0}\) and let \(b=\tau^{-1}+2\tau\) and \(m_{b}=\|\bbE L\|+sb+|z|+\tau + 1\). Let \(F^{(\tau)}\in \calA_{b}\) satisfying~\eqref{eq:F_RDEL} with \(\epsilon_{\tau}<\tau m_{b}^{-2}\). Then, \(\|M^{(\tau)}-F^{(\tau)}\|\leq 2(\tau+\tau^{-1})\tanh((1+\tau^{2})(m_{b}^{2}\tau^{-2}-1)\mathrm{arctanh}(\epsilon_{\tau}m_{b}^{2}/\tau ))\).
\end{lemma}

\begin{proof}
    Recall that \(\delta=(m_{b}^{2}\tau^{-2}-1)^{-1}\). We begin by verifying that this choice of \(b\) satisfies the settings in \cref{lemma:CRF_eps,lemma:strict_contraction}. That is, we show that \(\tau^{-1}(1+2\delta)<b\) and \(b>\tau^{-1}+\tau m_{b}^{-2}\). To this end, notice that \(m_{b}^{2}\tau^{-2}>\tau^{-2}+1\). Therefore, \(\delta<\tau^{2}\) and \(\tau^{-1}(1+2\delta)<b\). Furthermore, \(m_{b}^{-2}<1\) so \(\tau^{-1}+\tau m_{b}^{-2}<b\).

    By \Cref{corollary:CRF_norm_lower}, \(\|M^{(\tau)}-F^{(\tau)}\| \leq (b+\tau^{-1})\tanh(\rho(M^{(\tau)},F^{(\tau)}))\). Recursively define a sequence \(\{M_{k}\st k\in \bbN_{0}\}\subseteq \calA_{b}\) such that \(M_{0}=F^{(\tau)}\) and \(M_{k}=\Ftau(M_{k-1})\) for every \(k\in \bbN\). By \cref{lemma:strict_contraction,lemma:CRF_eps},
    \begin{align*}
        \rho(M^{(\tau)},F^{(\tau)}) & \leq \rho(M^{(\tau)},M_{k}) + \sum_{j=1}^{k}\rho(M_{j},M_{j-1})
        \\ & \leq (1+\delta)^{-k}\rho(M^{(\tau)},F^{(\tau)}) +  \rho(\Ftau(F^{(\tau)}),F^{(\tau)})\sum_{j=1}^{k}(1+\delta)^{-j}
        \\ & \leq (1+\delta)^{-k}\rho(M^{(\tau)},F^{(\tau)}) +  \frac{\mathrm{arctanh}(\epsilon_{\tau}m_{b}^{2}/\tau )}{1-(1+\delta)^{-1}}.
    \end{align*}
    Since the above inequality hold for every \(k\in \bbN\), we may take the limit as \(k\to \infty\) to obtain \(\rho(M^{(\tau)},F^{(\tau)}) \leq (1+\delta)\mathrm{arctanh}(\epsilon_{\tau}m_{b}^{2}/\tau )/\delta\). Combining everything, using the fact that \(\delta<\tau^{2}\), we obtain the result.
\end{proof}

To summarize, if \(F^{(\tau)} \in \calA_{\tau^{-1}+2\tau}\) approximately solves~\eqref{eq:RDEL} up to an additive perturbation term \(D^{(\tau)}\) in the sense of~\eqref{eq:F_RDEL}, and \(\|\Ftau(F^{(\tau)})D^{(\tau)}\|\leq \tau^{-1}\|D^{(\tau)}\|\) vanishes as \(\ell\to\infty\), then \(\|M^{(\tau)}-F^{(\tau)}\|\) converges to \(0\) as \(\ell\to\infty\). Here, \(M^{(\tau)}\) refers to the unique solution to~\eqref{eq:RDEL}. Notably, if we take \(F^{(\tau)}=\bbE(L-z\Lambda - i\tau I_{\ell})^{-1}\), then we obtain a way to establish pointwise convergence in \(z\) and \(\tau\). Utilizing \cref{condition:flat,condition:convergence_RDEL} as well as \cref{lemma:Ftau_vs_F,lemma:Ftau_admissible}, along with a diagonalization argument, we can establish the following result.

\begin{corollary}\label{corollary:stability}
    Let \(z\in \bbH\), \(M\in \calM\) be the unique solution to~\eqref{eq:DEL} and assume \cref{condition:flat,condition:convergence_RDEL} hold. For every \(\tau\in \bbR_{>0}\), let \(D^{(\tau)}\) be defined by~\eqref{eq:F_RDEL} with \(F^{(\tau)}=\bbE(L-z\Lambda - i\tau I_{\ell})^{-1}\). If \(\|D^{(\tau)}\| \to 0\) as \(\ell\to\infty\) for every \(\tau\in \bbR_{>0}\), then \(\|\bbE (L-z\Lambda)^{-1}-M(z)\|\to 0\) as \(\ell\to\infty\).
\end{corollary}
\begin{proof}
    The result follows from a combination of \Cref{lemma:stability} along with \cref{condition:flat,condition:convergence_RDEL} as well as \cref{lemma:Ftau_vs_F,lemma:Ftau_admissible}. We write
    \begin{align*}
        \|\bbE (L-z\Lambda)^{-1}-M(z)\| & \leq \|\bbE (L-z\Lambda)^{-1}-\bbE (L-z\Lambda-i\tau I_{\ell})^{-1}\|
        \\ & + \|\bbE (L-z\Lambda-i\tau I_{\ell})^{-1} - M^{(\tau)}(z)\|
        \\ & + \|M^{(\tau)}(z)-M(z)\|.
    \end{align*}
    For the first term, we use \Cref{lemma:Ftau_vs_F} and \cref{condition:flat} to obtain \(\|\bbE(L-z\Lambda)^{-1}-\bbE (L-z\Lambda-i\tau I_{\ell})^{-1}\| \lesssim \tau\). For the second term, it follows from \Cref{lemma:stability} and \cref{condition:convergence_RDEL} that \(\|\bbE (L-z\Lambda-i\tau I_{\ell})^{-1} - M^{(\tau)}(z)\|\lesssim \tau^{-1}\tanh(\tau^{-4}\mathrm{arctanh}(\tau^{-4}\|D^{(\tau)}\|))\). For the third term, it follows directly from \cref{condition:convergence_RDEL} that there exists a sequence \(\{\tau_{k}\}_{k\in \bbN}\) and a function \(f:\bbR \mapsto \bbR_{\geq 0}\) such that \(\tau_{k}\to 0\) as \(k\to \infty\), \(f(\tau_{k})\to 0\) as \(k\to \infty\) and \(\|M^{(\tau)}(z)-M(z)\|\leq f(\tau_{k})+o_{\ell}(1)\). We obtain the result by letting \(\tau\to 0\) and \(\ell\to \infty\) simultaneously such that the rate of convergence of \(\tau\) is chosen in terms of the rate of convergence of \(D^{(\tau)}\).
\end{proof}

The proof of \Cref{corollary:stability} highlights a contrast between \(\tau\) and \(\ell\) regarding their impact on convergence behavior. On one hand, as \(\tau\) approaches \(0\), the solution to the regularized Dyson equation~\eqref{eq:RDEL} converges to the solution of the Dyson equation~\eqref{eq:DEL}. However, the solution to the Dyson equation possesses less desirable properties compared to the regularized solution, notably because the imaginary part is not guaranteed to be positive definite. On the other hand, as \(\ell\) approaches \(\infty\), we establish stability for fixed \(\tau\) by controlling the magnitude of the error \(\epsilon_{\tau}\). We leverage this stability to establish convergence between the expected regularized pseudo-resolvent and the solution to the DEL. Therefore, we need to allow \(\tau\) to approach \(0\) slowly enough as \(\ell\to\infty\) to preserve the stability property.

\subsection{Perturbation}\label{sec:perturbation}

In view of \Cref{lemma:stability} and \Cref{corollary:stability}, the focus shifts to proving that the perturbation matrix vanishes in norm as the problem dimension grows for every regularization parameter. For each \(\tau\in \bbR_{>0}\), we consider the expected regularized pseudo-resolvent \(F^{(\tau)}\equiv F^{(\tau)}(z) = \bbE (L-z\Lambda - i\tau I_{\ell})^{-1} \in \calA_{+}\) which satisfies \((\bbE L - \supop(F^{(\tau)}(z))-z\Lambda -i\tau I_{\ell})F^{(\tau)}(z) = I_{\ell}+D^{(\tau)}\) where \(D^{(\tau)}\) is a regularized perturbation term explicitly given by
\begin{equation}\label{eq:D}
    D^{(\tau)} = \bbE\left[\left(\bbE L - L -\supop(\bbE(L-z\Lambda -i\tau I_{\ell})^{-1})\right)(L-z\Lambda -i\tau I_{\ell})^{-1}\right].
\end{equation}

Under \cref{condition:gaussian_design}, we aim to decompose the perturbation matrix \(D^{(\tau)}\) into terms that are amenable to analysis. To achieve this, recall the definition of \(\Delta(L,\tau;z)\) from~\eqref{eq:dist_gaussian} and consider the decomposition
\begin{subequations}\label{eq:D_ineq}
    \begin{align}
        D^{(\tau)} & = \bbE\left[\supop((L-z\Lambda -i\tau I_{\ell})^{-1})(L-z\Lambda -i\tau I_{\ell})^{-1}\right] - \supop(F^{(\tau)})F^{(\tau)} \label{eq:D_ineq1}
        \\ & + \bbE\left[\tilde{\supop}((L-z\Lambda -i\tau \label{eq:D_ineq2}I_{\ell})^{-1})(L-z\Lambda -i\tau I_{\ell})^{-1}\right] 
        \\ & - \Delta(L,\tau). \label{eq:D_ineq3}
    \end{align}
\end{subequations}

In order to maintain an adequate level of abstraction, we will directly assume that the mapping \(g\mapsto \supop(L(g)-z\Lambda -i\tau I_{\ell})^{-1}\) is \(\lambda\)-Lipschitz with respect to the operator norm and employ an \(\epsilon\)-net argument to obtain bounds \(\bbE_{\tilde{L}}[\|(\tilde{L}-\bbE L)((L-z\Lambda -i\tau I_{\ell})^{-1}-\bbE (L-z\Lambda - i\tau I_{\ell})^{-1})(\tilde{L}-\bbE L)\|]\) for \(k\in \bbN\).

\begin{lemma}\label{lemma:supop_norm_gaussian}
    Fix \(z\in \bbH\) and \(\tau\in \bbR_{>0}\). Assume that the mapping \(g\in (\bbR^{\gamma},\|\cdot\|_{2})\mapsto \supop((L(g)-z\Lambda - i \tau I_{\ell})^{-1})\in (\bbC^{\ell\times \ell},\|\cdot\|_{2})\) is \(\lambda\)-Lipschitz. Then, for every \(k\in \bbN\), there exists an absolute constant \(c\in \bbR_{>0}\) such that
    \[
        \bbE\left[\|\supop\left((L-z\Lambda - i \tau I_{\ell})^{-1}-\bbE (L-z\Lambda - i\tau I_{\ell})^{-1}\right)\|^{k}\right] \leq c\ell^{k/2}\lambda^{k}.
    \]
\end{lemma}
\begin{proof}
    By the Gaussian concentration inequality for Lipschitz functions~\cite[Theorem 2.1.12]{tao_rmt} and~\cite[Proposition 1.10]{ledoux_concentration_2001}, there exists some absolute constant \(c_{1},c_{2}\in \bbR_{>0}\) such that
    \[
        \bbP\left(\lambda^{-1} \left|u^{\ast}\supop\left((L-z\Lambda - i \tau I_{\ell})^{-1}-\bbE (L-z\Lambda - i\tau I_{\ell})^{-1}\right)v\right|\geq x\right)\leq c_{1}e^{-c_{2} x^{2}}
    \]
    for all unit vectors \(u,v\in \bbC^{\ell}\). Suppose that \(\epsilon\in (0,2^{-3/2})\) and let \(\net\) be an \(\epsilon\)-net for the unit ball of \(\ell\)-dimensional real vectors. Then, given \(u\in \bbC^{\ell}\), we may find \(v_{1},v_{2}\in \net\) such that \(\|u-v_{1}-iv_{2}\|^{2} = \|\Re[u]-v_{1}\|^{2}+\|\Im[u]-v_{2}\|^{2}\leq 2\epsilon^{2}\). In particular, \(\net+ i\net:=\{v_{1}+iv_{2}\st v_{1},v_{2}\in \net\}\) forms a \(\sqrt{2}\epsilon\)-net for the unit sphere of \(\ell\)-dimensional complex unitary vectors. By~\cite[Corollary 4.2.13]{vershynin_2018}, \(|\net+ i\net| \leq (2\epsilon^{-1}+1)^{2\ell}\).

    Let \(u,v\in \bbC^{\ell}\) be unitary and let \(u_{0},v_{0}\in \net + i\net\) such that \(\|u-u_{0}\|\leq \sqrt{2}\epsilon\) and \(\|v-v_{0}\|\leq \sqrt{2}\epsilon\). Let \(X\in \bbC^{\ell\times \ell}\) be any matrix. Using the identity \(u^{\ast}Xv=u_{0}^{\ast}Xv_{0}+(u^{\ast}-u_{0}^{\ast})Xv +u_{0}^{\ast}X(v-v_{0})\), we obtain \(|u^{\ast}Xv| \leq \sup_{u_{0},v_{0}\in \net+i\net}|u_{0}^{\ast}Xv_{0}|+2^{3/2}\epsilon \|X\|\). Taking the supremum over unitary complex vectors \(u\) and \(v\), we get that \(\|X\|\leq (1-2^{3/2}\epsilon)^{-1}\sup_{u_{0},v_{0}\in \net+i\net}|u_{0}^{\ast}Xv_{0}|\). In particular, for \(X=\supop((L-z\Lambda - i \tau I_{\ell})^{-1}-\bbE (L-z\Lambda - i\tau I_{\ell})^{-1})\), we apply a union bound to obtain \(\bbP\left(\lambda^{-1} \|X\|\geq y\right) \leq c_{1} (2\epsilon^{-1}+1)^{4\ell}e^{-c_{3} y^{2}}\) for every \(y\in \bbR\), where \(c_{3}=c_{2}(1-2^{3/2}\epsilon)\). Let \(c_{4}\in \bbR_{>0}\) and \(y=c_{4}(2\sqrt{\ell}+x)\) for all \(x\in \bbR_{\geq 0}\) such that \(y^{2}\geq c_{4}^{2}(4\ell+x^{2})\). Choosing \(c_{4}\) large enough such that \(c_{3}^{2}y^{2}\geq \ln(2\epsilon^{-1}+1)4\ell+x^{2}\) for every \(x\in \bbR_{\geq 0}\), we have \( c_{1} (2\epsilon^{-1}+1)^{4\ell}e^{-c_{3} y^{2}}\leq c_{1} (2\epsilon^{-1}+1)^{4\ell}e^{-\ln(2\epsilon^{-1}+1)4\ell-x^{2}}=c_{1}e^{-x^{2}}\). Let \(k\in \bbN\) be arbitrary. Then,
    \begin{align*}
        \bbE [\lambda^{-k}\|X\|^{k}] & = k\int_{0}^{\infty} \bbP\left(\lambda^{-1}\|X\|\geq y\right) y^{k-1}\d y
        \\ & = k\int_{0}^{2c_{4}\sqrt{\ell}} \bbP\left(\lambda^{-1}\|X\|\geq y\right) y^{k-1}\d y + k\int_{2c_{4}\sqrt{\ell}}^{\infty} \bbP\left(\lambda^{-1}\|X\|\geq y\right) y^{k-1}\d y.
    \end{align*}
    On one hand, it is straightforward to bound \(k\int_{0}^{2c_{4}\sqrt{\ell}} \bbP\left(\lambda^{-1}\|X\|\geq y\right) y^{k-1}\d y\leq 2^{k}c_{4}^{k}\ell^{k/2}\). On the other hand, we have
    \begin{align*}
        k\int_{2c_{4}\sqrt{\ell}}^{\infty} \bbP\left(\lambda^{-1}\|X\|\geq y\right) y^{k-1}\d y & = c^{k}_{4}k\int_{2c_{4}\sqrt{\ell}}^{\infty} \bbP\left(\lambda^{-1}\|X\|\geq c_{4}(2\sqrt{\ell}+x)\right) (2\sqrt{\ell}+x)^{k-1}\d x
        \\ & \leq c_{1}c^{k}_{4}k\int_{2c_{4}\sqrt{\ell}}^{\infty} e^{-x^{2}} (2\sqrt{\ell}+x)^{k-1}\d x
        \\ & = c_{1}c^{k}_{4}k\int_{0}^{\infty} e^{-(x+2c_{4}\sqrt{\ell})^{2}} (2(1+c_{4})\sqrt{\ell}+x)^{k-1}\d x.
    \end{align*}
    In particular, writing
    \[
        \int_{0}^{\infty} e^{-(x+2c_{4}\sqrt{\ell})^{2}} (2(1+c_{4})\sqrt{\ell}+x)^{k-1}\d x\leq e^{-2c^{4}\ell}\int_{0}^{\infty} e^{-x^{2}} (2(1+c_{4})\sqrt{\ell}+x)^{k-1}\d x
    \]
    and noting that \(\int_{0}^{\infty} e^{-x^{2}} (2(1+c_{4})\sqrt{\ell}+x)^{k-1}\d x\) is polynomial in \(\ell\), we get that
    \[
        \int_{2c_{4}\sqrt{\ell}}^{\infty} \bbP\left(\lambda^{-1}\|X\|\geq y\right) y^{k-1}\d y=o_{\ell}(1).
    \]
    The result follows.
\end{proof}

The practicality of \Cref{lemma:supop_norm_gaussian} relies on the Lipschitz constant \(\lambda\) satisfying \(\lim_{\ell\to\infty} \lambda \sqrt{\ell}=0\).
Under this condition, we may show that the perturbation term \(D^{(\tau)}\) vanishes in norm as \(\ell\to\infty\) for every \(\tau\in \bbR_{>0}\).

\begin{lemma}\label{lemma:D_vanishes}
    Let \(\tau\in \bbR_{> 0}\), \(z\in \bbH\) and \(D^{(\tau)}\) be the perturbation matrix in~\eqref{eq:D}. Under \cref{condition:gaussian_design}, assume that the mapping \(g\in (\bbR^{\gamma},\|\cdot\|_{2})\mapsto \supop((L(g)-z\Lambda - i \tau I_{\ell})^{-1})\in (\bbC^{\ell\times \ell},\|\cdot\|_{2})\) is \(\lambda\)-Lipschitz. Then, there exists an absolute constant \(c\in \bbR_{>0}\) such that
    \[
        \|D^{(\tau)}\| \leq c\tau^{-1} \sqrt{\ell}\lambda + \tau^{-1}\bbE\|\tilde{\supop}((L-z\Lambda - i\tau I_{\ell})^{-1})\| + \|\Delta(L,\tau)\|.
    \]
\end{lemma}
\begin{proof}
    By~\eqref{eq:D_ineq}, we have
    \begin{align*}
        \|D^{(\tau)}\|  & \leq \|\bbE[ \supop((L-z\Lambda -i\tau I_{\ell})^{-1}-\bbE (L-z\Lambda - i\tau I_{\ell})^{-1})(L-z\Lambda -i\tau I_{\ell})^{-1}]\|
        \\ & + \| \bbE[\tilde{\supop}((L-z\Lambda -i\tau I_{\ell})^{-1})(L-z\Lambda -i\tau I_{\ell})^{-1}]\| + \|\Delta(L,\tau)\|.
    \end{align*}
    By Jensen's inequality, submultiplicativity of the operator norm, \Cref{lemma:prior_bounds} and \Cref{lemma:supop_norm_gaussian}, there exists an absolute constant \(c\in \bbR_{>0}\) such that \( \|\bbE[ \supop((L-z\Lambda -i\tau I_{\ell})^{-1}-\bbE (L-z\Lambda - i\tau I_{\ell})^{-1})(L-z\Lambda -i\tau I_{\ell})^{-1}]\|  \leq c \tau^{-1}\sqrt{\ell}\lambda\). Similarly, \(\|\bbE[\tilde{\supop}((L-z\Lambda -i\tau I_{\ell})^{-1})(L-z\Lambda -i\tau I_{\ell})^{-1}]\|  \leq \tau^{-1}\bbE\|\tilde{\supop}((L-z\Lambda - i\tau I_{\ell})^{-1})\|\). The result follows.
\end{proof}

As a direct outcome of \cref{lemma:D_vanishes,lemma:prior_bounds}, it follows that as the dimension \(\ell\) tends towards infinity, \(\|D^{(\tau)}\|\) diminishes, provided that \(\lim_{\ell\to\infty}\sqrt{\ell}\lambda = 0\), \(\lim_{\ell\to\infty}\|\supop\|=0\), and \(\lim_{\ell\to\infty}\|\Delta(L,\tau)\|=0\) for every \(\tau\in \bbR_{>0}\) sufficiently small. Before we proceed, we highlight two significant observations.

Despite the possibility to simplify the upper bound \(\tau^{-1}\bbE\|\tilde{\supop}((L-z\Lambda - i\tau I_{\ell})^{-1})\|\) in \Cref{lemma:D_vanishes} to \(\tau^{-2}\|\tilde{\supop}\|\) using \Cref{lemma:prior_bounds}, we choose to retain its current form since we consider it as more insightful and convenient in certain scenarios. For example, if \(B=0\) in the linearization, the regularized pseudo-resolvent becomes block-diagonal. Hence, the chosen bound allows us to focus solely on demonstrating the vanishing operator norm of \(\tilde{\supop}\) within the set of block diagonal matrices as \(\ell\) grows.

The convergence results from \Cref{corollary:stability} and \Cref{lemma:D_vanishes}, coupled with the outlined conditions on the linearization, establish that \(M(z)\) serves as a deterministic equivalent for the expected pseudo-resolvent \(\bbE (L-z\Lambda)^{-1}\) across the entire upper-half complex plane \(\bbH\).

\begin{corollary}\label{corollary:expect_deterministic_equivalent}
    Let \(z\in \bbH\) and \(\lambda\) be defined as in \Cref{lemma:D_vanishes}. Under \cref{condition:flat,condition:convergence_RDEL,condition:gaussian_design}, suppose that \(\lim_{\ell\to\infty}\sqrt{\ell}\lambda = \lim_{\ell\to\infty}\|\tilde{\supop}\|=\lim_{\ell\to\infty}\|\Delta(L,\tau)\|=0\) for every \(\tau\in \bbR_{>0}\) small enough. Then, \(\|\bbE(L-z\Lambda )^{-1}-M(z)\| \to 0\) as \(\ell\to\infty\).
\end{corollary}

\subsection{Concentration}\label{sec:concentration}

The only remaining task is to establish that the expected pseudo-resolvent is itself a deterministic equivalent for the pseudo-resolvent. We present one possible approach, which is based on \cref{condition:gaussian_design}.

\begin{lemma}\label{lemma:concentration}
    Under \cref{condition:gaussian_design}, let \(U\in \bbC^{\ell\times \ell}\) with \(\|U\|_{F}\leq 1\) and assume that the map \(g\in (\bbR^{\gamma},\|\cdot\|_{2})\mapsto  (L(g)-z\Lambda)^{-1}\in (\bbC^{\ell\times \ell},\|\cdot\|_{F})\) is \(\lambda\)-Lipschitz with \(\lambda \asymp \ell^{-r}\) for some \(r>0\). Then, \(\tr(U((L-z\Lambda)^{-1}-\bbE (L-z\Lambda)^{-1})) \to 0\) almost surely as \(\ell\to\infty\).
\end{lemma}

\begin{proof}
    Let \(g_{1},g_{2}\in \bbR^{\gamma}\). Then, by Cauchy-Schwarz's inequality \(|\tr(U((L(g_{1})-z\Lambda)^{-1}-(L(g_{2})-z\Lambda)^{-1}))|\leq \|U\|_{F}\|(L(g_{1})-z\Lambda)^{-1}-(L(g_{2})-z\Lambda)^{-1}\|_{F}\leq \lambda \|g_{1}-g_{2}\|\). By Gaussian concentration inequality for Lipschitz functions~\cite{tao_rmt,ledoux_concentration_2001}, there exists absolute constants \(c_{1},c_{2}\in \bbR_{>0}\) such that \(\bbP(|\tr(U((L-z\Lambda)^{-1}-\bbE (L-z\Lambda)^{-1}))|\geq x)\leq c_{1}e^{-c_{2}x^{2}/\lambda^{2}}\leq c_{1}e^{-c_{3}x^{2}\ell^{2r}}\) for every \(x\in \bbR_{>0}\) and \(\ell\in \bbN\). Here, \(c_{3}\in \bbR_{>0}\) is some constant satisfying \(0<c_{3}\leq c_{2}/(\lambda \ell^{r})^{2}\) for every \(\ell\in \bbN\) large enough. In particular, for any \(\epsilon \in \bbR_{>0}\), \(\sum_{\ell=1}^{\infty}\bbP(|\tr(U((L-z\Lambda)^{-1}-\bbE (L-z\Lambda)^{-1}))|\geq \epsilon)\leq c_{1}\sum_{\ell=1}^{\infty}e^{-c_{3}\epsilon^{2}\ell^{2r}}\). By the integral test, the series \(\sum_{\ell=1}^{\infty}e^{-c_{3}\epsilon^{2}\ell^{2r}}\) converges, and the result follows from the Borel-Cantelli lemma.
\end{proof}

\section{Proof of Asymptotic Equivalence for Random Features}\label{sec:proof_rf}

In this section, we utilize the framework developed in previous sections to prove \Cref{theorem:rf_error}. Before delving into the argument, note that we may expand the test error in~\eqref{eq:test_error} as
\[
    \Etest = \|\hA A^{\top}(AA^{\top} + \delta I_{\ntrain})^{-1}y\|^{2}
- 2 \hy^{\top}\hA A^{\top}(AA^{\top} + \delta I_{\ntrain})^{-1}y
+ \|\hy\|^{2}.
\]
Each term in the above equations is a bilinear form, aligning well with the framework of deterministic equivalence. In order to apply our framework, we have to find a linearization such that a matrix of interest is contained in one of the blocks of the inverse. To this end, let \(\ell=\ntrain+d+2\ntest\) and consider the linearization
\begin{equation}\label{eq:first_linearization}
    \renewcommand{\arraycolsep}{7pt}
    L =
    \begin{bmatrix}
        \delta I_{\ntrain}  &A    & 0_{\ntrain\times \ntest}  & 0_{\ntrain\times \ntest}     \\
        A^{\top}    &  -I_{d\times d}   & 0_{d\times \ntest}    & \hA^{\top} \\
        0_{\ntest\times \ntrain} &  0_{\ntest\times d}          & 0_{\ntest\times \ntest}              & -I_{\ntest} \\
        0_{\ntest\times \ntrain}    & \hA           & -I_{\ntest}  & 0_{\ntest \times \ntest}
    \end{bmatrix}
    \in \bbR^{\ell\times \ell}.
\end{equation}
Taking \(\Lambda:=\diag\{ I_{\ntrain+d},0_{2\ntest\times 2\ntest}\}\), we use \Cref{lemma:block_inversion} to express the pseudo-resolvent \((L-z\Lambda)^{-1}\) block-wise as
\[
\renewcommand{\arraycolsep}{7pt}
    (L-z\Lambda)^{-1} = 
    \begin{bmatrix}
             R                         & (1+z)^{-1}RA           & (1+z)^{-1}RA\hA^{\top} & 0 \\
        (1+z)^{-1} A^{\top}R              & \bar{R}                &     \bar{R}\hA^{\top}  & 0 \\
        (1+z)^{-1}\hA A^{\top}R  & \hA\bar{R}  &\hA\bar{R}\hA^{\top}  &  -I_{\ntest}  \\
           0   &  0    &   -I_{\ntest}    &  0
    \end{bmatrix}.
\]
Here \(R:=((1+z)^{-1}AA^{\top} + (\delta-z) I_{\ntrain})^{-1}\) represents a resolvent and \(\bar{R}:=-((1+z)I_{d}+(\delta-z)^{-1}A^{\top}A)^{-1}\) is a co-resolvent. Indeed, \(\lim_{z\to 0}(L-z\Lambda)^{-1}_{3,1}=\hA A^{\top}(AA^{\top} + \delta I_{\ntrain})^{-1}\) holds one of the relevant expression for which we want to find a deterministic equivalent. Therefore, it suffices to find a deterministic equivalent for the pseudo-resolvent \((L-z\Lambda)^{-1}\) and take the spectral parameter to zero in order to recover a deterministic equivalent for \(\hA A^{\top}(AA^{\top} + \delta I_{\ntrain})^{-1}\).

The linearization in~\eqref{eq:first_linearization} yields the superoperator
\[
    \supop:M \in\bbC^{\ell\times \ell} \mapsto 
    \begin{bmatrix}
        \tr(M_{2,2}) K_{AA^{\top}} & 0 & 0 & \tr(M_{2,2})K_{A\hA^{\top}} \\
        0 &  \rho(M) I_{d}  & 0 & 0 \\
        0 & 0 & 0 & 0 \\
        \tr(M_{2,2})K_{\hA A^{\top}} & 0 & 0 & \tr(M_{2,2}) K_{\hA\hA^{\top}}
    \end{bmatrix} \in \bbC^{\ell\times \ell}
\]
where \(\rho(M):= \tr(K_{AA^{\top}} M_{1,1}+K_{A\hA^{\top}}M_{4,1}+K_{\hA A^{\top}}M_{1,4} + K_{\hA\hA^{\top}} M_{4,4})\). Then, \(\supop(M)=\bbE[(L-\bbE L)M(L-\bbE L)]-\tilde{\supop}(M)\)
holds with
\[
    \tilde{\supop}(M):=
    \bbE 
    \begin{bmatrix}
        0 & \substack{K_{AA^{\top}} M^{\top}_{2,1} + K_{A\hA^{\top}}M^{\top}_{2,4}} & 0 & 0 \\
        \substack{M^{\top}_{1,2}K_{AA^{\top}}  + M^{\top}_{4,2}K_{\hA A^{\top}}} &  0 &0 & 
        \substack{M_{1,2}^{\top}K_{A\hA^{\top}}  + M_{4,2}^{\top}K_{\hA\hA^{\top}}} \\
        0 & 0 & 0 & 0 \\
        0 & \substack{K_{\hA A^{\top}} M^{\top}_{2,1} + K_{\hA\hA^{\top}} M^{\top}_{2,4}} & 0 & 0
    \end{bmatrix}.
\]
By \Cref{theorem:main_properties}, there exists a unique solution \(M\in \calM\) such that \(M(z)\) solves~\eqref{eq:DEL} for every \(z\in \bbH\). Plugging-in the expression for the superoperator above and using \Cref{lemma:block_inversion}, we find that
\begin{equation}\label{eq:M_RF}
    M(z) = 
    \begin{bmatrix}
        \substack{( (\delta-z) I_{\ntrain}-\tr(M_{2,2})K_{AA^{\top}})^{-1}} & 0 & -\substack{\tr(M_{2,2})M_{1,1}K_{A\hA^{\top}}} & 0 \\
        0          & \substack{d^{-1}\tr(M_{2,2}) I_{d}} & 0 & 0  \\
        \substack{-\tr(M_{2,2})K_{\hA A^{\top}}M_{1,1}} & 0 & \substack{(\tr(M_{2,2}))^{2}K_{\hA A^{\top}}M_{1,1}K_{A\hA^{\top}}
        \\ +\tr(M_{2,2})K_{\hA\hA^{\top}}} & \substack{-I_{\ntest}} \\
        0  & 0  & \substack{-I_{\ntest}} & 0
    \end{bmatrix}
\end{equation}
with \(M_{2,2}=-(1+z+\tr(K_{AA^{\top}} M_{1,1}))^{-1}I_{d}\).

A key observation that greatly simplifies both the theoretical analysis of the DEL and enables us to derive an iterative procedure for computing its solution is the fact that we can treat the upper-left \(\ntrain+d\) block of the DEL as a separate DEL. This insight allows us to effectively break down the problem and focus on a smaller sub-DEL. Let \(\Lsub\) denote the upper-left \(\ntrain+d\) block of \(L\), define a \((\ntrain+d)\times (\ntrain + d)\) sub-superoperator
\[
    \supop^{(\mathrm{sub})}: X \mapsto 
    \begin{bmatrix}
        \tr(X_{2,2})K_{AA^{\top}} & 0 \\
        0 & \tr(K_{AA^{\top}}X_{1,1})
    \end{bmatrix}
\]
and a new sub-DEL mapping
\[
    \Fsub: f \in \calM^{(\mathrm{sub})}_{+} \mapsto (\bbE \Lsub -  \supop^{(\mathrm{sub})}(f(\cdot)) - (\cdot)I_{\ntrain+d})^{-1} \in \calM^{(\mathrm{sub})}_{+}.
\]
Here, the set \(\calM^{(\mathrm{sub})}_{+}=\Hol(\bbH,\calA^{(\mathrm{sub})})\) and \(\calA^{(\mathrm{sub})}=\{N\in \bbC^{(\ntrain+d)\times (\ntrain + d)}\st \Im[N]\succ 0\}\). Given that the sub-DEL has a spectral parameter spanning its diagonal, the iteration scheme \(N_{k+1} = \Fsub(N_{k})\) converges to the unique solution of the sub-DEL \(\Msub = \Fsub(\Msub)\) for any \(N_{0} \in \calM^{(\mathrm{sub})}_{+}\), as per~\Cref{proposition:existenceuniquness_RDEL}.

We will extensively use the DEL in~\eqref{eq:M_RF} and the sub-DEL to establish \Cref{theorem:rf_error}. To apply our theoretical framework, it is necessary to demonstrate that \(\|\Delta(L,\tau;z)\|\), as defined in~\eqref{eq:dist_gaussian}, vanishes as \(n\to\infty\) for every regularization parameter \(\tau\in \bbR_{>0}\). To achieve this, we employ a leave-one-out method. While the ensuing argument involves detailed and intricate calculations, it is tedious and differs heavily from the rest of the argument. Therefore, we state the result establishing that \(\|\Delta(L,\tau;z)\|\) is vanishing and defer the proof to Appendix B.

\begin{lemma}\label{lemma:universality}
    Fix \(z\in \bbH\). Let \(L\) be the linearization defined in~\eqref{eq:first_linearization} and \(\Delta(L,\tau)\) be defined as in~\eqref{eq:dist_gaussian}. Under the settings of \Cref{theorem:rf_error}, \(\lim_{\ell\to\infty}\|\Delta(L,\tau)\|=0\) for every \(\tau\in \bbR_{>0}\).
\end{lemma}

An important lemma, which plays a key role in the proof of \Cref{lemma:universality} and is used in subsequent sections, establishes that the bound on the expected norm of the linearization can be extended to a bound on the expected value of powers of the norm of the linearization through a concentration argument. For completeness, we state and prove this lemma here.
\begin{lemma}\label{lemma:bound_moments}
    Under the settings of \Cref{theorem:rf_error}, \(\limsup_{n\to\infty}\bbE [\|L-\bbE L\|^{k}]<\infty\) for every \(z\in \bbH\) and \(k\in \bbN\).
\end{lemma}
\begin{proof}
    Let \(W_{1}=\varphi(Z_{1}),W_{2}=\varphi(Z_{1})\) for some \(Z_{1},Z_{2}\in \bbR^{n_{0}\times d}\). Then,
    \begin{align*}
        |\|n^{-\frac{1}{2}}\sigma(XW_{1})\| - \|n^{-\frac{1}{2}}\sigma(XW_{2})\|| &\leq |n^{-\frac{1}{2}}\sigma(XW_{1}) - n^{-\frac{1}{2}}\sigma(XW_{2})\| 
        \\ & \leq n^{-\frac{1}{2}}\|\sigma(XW_{1})-\sigma(XW_{2})\|_{F} 
        \\ & \leq \frac{\lambda_{\sigma}}{\sqrt{n}}\|X(W_{1}-W_{2})\|_{F}
        \\ & \leq \frac{\lambda_{\sigma}\lambda_{\varphi}\|X\|}{\sqrt{n}}\|Z_{1}-Z_{2}\|_{F}.
    \end{align*}
    By Gaussian concentration inequality for Lipschitz functions~\cite{tao_rmt,ledoux_concentration_2001}, there exists absolute constants \(c_{1},c_{2}\in \bbR_{>0}\) such that \(\bbP(|\|A\|-\bbE\|A\|| \geq t)\leq c_{1}\exp(-c_{2}nt^{2}/(\lambda_{\sigma}^{2}\lambda_{\varphi}^{2}\|X\|^{2}))\). In other words, the random variable \(\|A\|-\bbE\|A\|\) is \(\lambda_{\sigma}\lambda_{\varphi}\|X\|/\sqrt{n}\)-sub-Gaussian, and similarly for \(\|\hA\|-\bbE \|\hA\|\). The result follows by applying the moment bound for sub-Gaussian random variables and using the fact that \(\bbE\|A\|\) and \(\bbE\|\hA\|\) are bounded as \(n\to\infty\).
\end{proof}

In the following sections, we use the Dyson equation for linearization framework to derive a deterministic equivalent for the matrix \(\hA A^{\top}(AA^{\top} + \delta I_{\ntrain})^{-1}\) in \Cref{sec:first_deterministic_equivalent}, and subsequently derive a deterministic equivalent for its square in \Cref{sec:second_deterministic_equiv}.

\subsection{First Deterministic Equivalent}\label{sec:first_deterministic_equivalent}

We will show that the solution to the DEL given in~\eqref{eq:M_RF} is a deterministic equivalent for the pseudo-resolvent \((L-z\Lambda)^{-1}\) in a neighborhood of \(z=0\). The key lies in establishing control over \(M(z)\) in the proximity of \(z=0\). This control is secured through the insights provided by the following lemma.

\begin{lemma}\label{lemma:control_origin}
    Let \(z\in \bbH\) with \(|z| < 1\wedge \delta\) and \(M\in \calM\) be the unique solution to~\eqref{eq:M_RF}. Then, \(\Re[M_{1,1}(z)]\succ 0\) and \(\Re[M_{2,2}(z)]\prec 0\). Additionally, \(\|M_{1,1}(z)\|\leq (\delta-\Re[z])^{-1}\) and \(\|M_{2,2}(z)\|\leq (1+\Re[z])^{-1}\).
\end{lemma}
\begin{proof}
    Let \(N\) be any \((\ntrain+d)\times (\ntrain+d)\) matrix-valued analytic function on \(\bbH\) such that \(\Im[N(z)]\succ 0\), \(N_{1,2}(z)=N_{2,1}(z)= 0\) for every \(z\in \bbH\). Further assume that \(\Re[N_{1,1}(z)]\succeq 0\) and \(\Re[N_{2,2}(z)]\preceq 0\) for every \(z\in \bbH\) with \(|z|< 1\wedge \delta\). By \Cref{lemma:res_trick},
    \begin{equation}\label{eq:real_part_F}
        \Re[\Fsub(N)] = \Fsub(N)
            \begin{bmatrix}
            \substack{(\delta-\Re[z]) I_{\ntrain} 
            \\  - \tr(\Re[N_{2,2}])K_{AA^{\top}}} & 0 \\
                0 & \substack{-(1+\Re[z]+\tr(K_{AA^{\top}}\Re[N_{1,1}]))}
            \end{bmatrix} 
            (\Fsub(N))^{\ast}
    \end{equation}
    where we omit the dependence of \(N\) on \(z\). Thus, \(\Re[\Fsub_{1,1}(N)]\succ 0\) and \(\Re[\Fsub_{2,2}(N)]\prec 0\) for every \(z\in \bbH\) with \(|z|< 1\wedge \delta\). Additionally, \(\Fsub_{1,2}(N)=\Fsub_{2,1}(N)=0\). Since the iterates \(N_{k+1}=\F^{(\mathrm{sub})}(N_{k})\) converges to the unique solution to the sub-DEL, it must be the case that  \(\Re[\Msub_{1,1}(z)] \succ 0\) and \(\Re[\Msub_{2,2}(z)]\prec 0\) for every \(z\in \bbH\) with \(|z|< 1\wedge \delta\). In fact, by uniqueness of the solution to~\eqref{eq:DEL}, \(\Re[M_{1,1}(z)] \succ 0\) and \(\Re[M_{2,2}(z)]\prec 0\) for every \(z\in \bbH\) with \(|z|< 1\wedge \delta\). Using the fact that \(\Fsub(\Msub)=\Msub\) and~\eqref{eq:real_part_F}, we get
    \[
        \Re[M_{1,1}] \succeq (\delta - \Re[z])M_{1,1}(M_{1,1})^{\ast} \quad \text{and} \quad \Re[M_{2,2}]\preceq -(1+\Re[z])M_{2,2}(M_{2,2})^{\ast}
    \]
    for every \(z\in \bbH\) with \(|z|\leq 1\wedge \delta\). Since the spectral norm maintains the Loewner partial ordering, it follows from \Cref{lemma:real/imag_norm_bound} that \(\|M_{1,1}\| \geq (\delta - \Re[z])\|M_{1,1}(M_{1,1})^{\ast}\| =(\delta - \Re[z])\|M_{1,1}\|^{2}\) and \(\|M_{2,2}\| \geq (1+\Re[z])\|M_{2,2}\|^{2}\). Rearranging yields the desired result.
\end{proof}

It will be useful later to not only have a deterministic equivalent for \((L-z\Lambda)^{-1}\), but also for \((\Lsub-zI_{\ntrain})^{-1}\).

\begin{lemma}\label{lemma:convergence_subDEL}
     Let \(z\in \bbH\) with \(|z|< \delta \wedge 1\) and \(\Msub \in \calM\) be the unique solution to the sub-DEL. Under the settings of \Cref{theorem:rf_error}, \(\tr(U((\Lsub-z I_{\ntrain+d})^{-1}-\Msub(z)))\to 0\) almost surely as \(n\to\infty\) for every sequence \(U\in \bbC^{(\ntrain+d)\times (\ntrain+d)}\) with \(\|U\|_{\ast}\leq 1\).
\end{lemma}
\begin{proof}
    We first derive some useful norm bounds. Recall that \(R=((1+z)^{-1}AA^{\top}+(\delta - z)I_{\ntrain})^{-1}\). For \(|z|< 1\wedge \delta\), \(\Re[(1+z)^{-1}]\geq |1+z|^{-2}(1-|z|)\geq (1-|z|)/4>0\) and \(\Re[\delta-z]\geq \delta-|z|\). Hence, \(\Re[R]\geq (\delta-|z|)RR^{\ast}\) which implies that \(\|R\|\leq (\delta-|z|)^{-1}\). A similar argument applied to \(\bar{R}=-((1+z)I_{d}+(\delta-z)^{-1}A^{\top}A)^{-1}\) gives \(\|\bar{R}\| \leq (1-|z|)^{-1}\). Furthermore, we know that \(\|RA\|^{2} =\|RAA^{\top}R^{\ast}\| \leq \|RAA^{\top}\| \|R\|\). By definition, \(RAA^{\top}=(1+z)I_{\ntrain}-(1+z)(\delta-z)R\). Thus, \(\|RAA^{\top}\| \leq 2(2+\delta)(\delta-|z|)^{-1}\) and \(\|RA\| \leq \sqrt{2(2+\delta)}(\delta-|z|)^{-1}\).

    We now check that \(\Lsub\) satisfies  \cref{condition:gaussian_design,condition:flat,condition:convergence_RDEL}. Since \(R\) and \(\bar{R}\) as well as the moments of \(\|A\|,\|\hA\|\) are bounded as \(n\) increases, \(\limsup_{\ell\to\infty}\bbE \|(L-z\Lambda)^{-1}\|^{2}<\infty\). In particular, \cref{condition:flat} is satisfied. Since there is a spectral parameter spanning the entire diagonal in the sub-DEL, it follows from the stability properties of the DEL or~\cite[Corollary 3.8]{alt_Kronecker} that \cref{condition:convergence_RDEL} is satisfied for the sub-DEL. It follows from \Cref{lemma:bound_moments} that \(\limsup_{\ell\to\infty}(\|\supop\|\vee  \|\bbE L\|)<\infty\). Finally, \cref{condition:gaussian_design} is evidently satisfied. Hence, the conditions are satisfied.
    
    Let \(D^{(\mathrm{sub})} = \bbE[(\bbE \Lsub - \Lsub -\supop^{(\mathrm{sub})}(\bbE(\Lsub-z I_{\ntrain+d})^{-1}))(\Lsub-zI_{\ntrain+d})^{-1}]\) be the perturbation matrix, as defined in~\eqref{eq:D}, associated with the linearization \(\Lsub\). In particular,
    \((\bbE \Lsub - \supop(\bbE(\Lsub-z I_{\ntrain+d})^{-1})-zI_{\ntrain+d})\bbE(\Lsub-z I_{\ntrain+d})^{-1} = I_{\ntrain+d}+D^{(\mathrm{sub})}\). As a consequence of \Cref{lemma:universality} with \(\hA=0\), \(\|\Delta(\Lsub,\tau;z)\|\to 0\) as \(n\to\infty\) for every \(\tau\in \bbR_{>0}\). Hence, to show that the perturbation matrix is vanishing as the dimension increases, we only have to establish that the term involving the Lipschitz constant and the term involving the norm of \(\tilde{\supop}^{(\mathrm{sub})}\) in \Cref{lemma:D_vanishes} are asymptotically negligible.

    Based on \cref{condition:gaussian_design}, write \(\Lsub \equiv \Lsub(Z) =\cov(Z)+\bbE \Lsub\) for \(Z\in \bbR^{n_{0}\times d}\) a matrix of i.i.d. standard normal entries and let \(\lambda\) be the Lipschitz constant associated to the function \(Z\in (\bbR^{n_{0}\times d},\|\cdot \|_{F})\mapsto \supop^{(\mathrm{sub})}((\Lsub(Z)-z I_{\ntrain+d})^{-1})\in (\bbC^{(\ntrain+d)\times (\ntrain+d)},\|\cdot \|_{2})\). Then, \(\lambda \leq (\Im[z])^{-2}\|\supop^{(\mathrm{sub})}\|_{F\mapsto 2}\lambda_{\cov}\) where \(\lambda_{\cov}\) is the Lipschitz constant associated with map \(\cov:Z\in (\bbR^{n_{0}\times d},\|\cdot\|_{F})\mapsto \cov(Z)\in (\bbR^{(\ntrain+d)\times (\ntrain+d)},\|\cdot \|_{F})\). For every \(N\in \bbC^{(\ntrain+d)\times (\ntrain+d)}\), we can use Cauchy-Schwarz inequality to obtain
    \[
        \|\supop^{(\mathrm{sub})}(N)\| \leq \|K_{AA^{\top}}\| |\tr(N_{2,2})|+|\tr(K_{AA^{\top}}N_{1,1})| \leq (\sqrt{d}+\sqrt{\ntrain})\|K_{AA^{\top}}\| \|N\|_{F}. 
    \]
    By Jensen's inequality, \(\|K_{AA^{\top}}\|=\|d^{-1}\bbE[AA^{\top}]\| \leq d^{-1}\bbE \|A\|^{2}\). In fact, by a similar argument, \(\|K_{AA^{\top}}\|\vee \|K_{\hA A^{\top}}\|\vee \|K_{A\hA^{\top}}\|\vee \|K_{\hA\hA^{\top}}\|\lesssim n^{-1}\). Since \(\bbE \|A\|\) is bounded by \Cref{lemma:bound_moments}, and we are working in the proportional limit, \(\|\supop^{(\mathrm{sub})}\|_{F\mapsto 2}\lesssim n^{-1/2}\). Next, let \(Z_{1},Z_{2}\in \bbR^{n_{0}\times d}\) and notice that \(\|\cov(Z_{1})-\cov(Z_{2})\|_{F} \leq n^{-1/2}\lambda_{\sigma}\lambda_{\varphi}\|X\|\|Z_{1}-Z_{2}\|_{F}\). Since \(\lambda_{\sigma}\), \(\lambda_{\varphi}\) and \(\|X\|\) are all bounded by assumption, \(\lambda \lesssim (\Im[z])^{-2}n^{-1}\). Finally, for every \(N\in \bbC^{(\ntrain+d)\times (\ntrain+d)}\),
    \[
        \tilde{\supop}^{(\mathrm{sub})}(N) = 
        \bbE 
        \begin{bmatrix}
            0 & K_{AA^{\top}} N^{\top}_{2,1} + K_{A\hA^{\top}}N^{\top}_{2,4} \\
            N^{\top}_{1,2}K_{AA^{\top}}  + N^{\top}_{4,2}K_{\hA A^{\top}} &  0  
        \end{bmatrix}.
    \]
    Since \(\|K_{AA^{\top}}\|\vee \|K_{\hA A^{\top}}\|\vee \|K_{A\hA^{\top}}\|\vee \|K_{\hA\hA^{\top}}\|\lesssim n^{-1}\), \(\|\tilde{\supop}^{(\mathrm{sub})}\| \lesssim n^{-1}\). Combining everything along with concentration of linear functional of the resolvent around their mean, the result follows from \Cref{theorem:convergence}.
\end{proof}

The final prerequisite needed to establish that the solution to~\eqref{eq:M_RF} acts as a deterministic equivalent for \((L-z\Lambda)^{-1}\), where \(L\) is defined in~\eqref{eq:first_linearization}, for every \(z\in \bbH\) within a neighborhood of the origin, is to confirm that \cref{condition:convergence_RDEL} holds. We demonstrate this in the following lemma.

\begin{lemma}\label{lemma:RDEL_conv_DEL}
    Suppose that \(M\in \calM\) is the unique solution to~\eqref{eq:M_RF} and \(M^{(\tau)}\) is the unique solution to the regularized version of the same equation. Then, \(M\) and \(M^{(\tau)}\) satisfy \cref{condition:convergence_RDEL} for all \(z\in \bbH\) with \(|z| < \delta \wedge 1\).
\end{lemma}

\begin{proof}
    Fix \(z\in \bbH\) with \(|z|<1 \wedge \delta\) and let \(\tau\in \bbR_{>0}\). Expanding~\eqref{eq:RDEL}, we have \(((\delta - z -i\tau) I_{\ntrain}-\tr(M_{2,2}^{(\tau)})K_{AA^{\top}})M_{1,4}^{(\tau)}=\tr(M_{2,2}^{(\tau)})K_{A\hA^{\top}}M_{4,4}^{(\tau)}\), \(-i\tau M_{3,4}^{(\tau)}-M_{4,4}^{(\tau)}=0\) and \(-M_{3,4}^{(\tau)}-(\tr(M_{2,2}^{(\tau)})K_{\hA\hA^{\top}}+i\tau I_{\ntrain})M_{4,4}^{(\tau)}-\tr(M_{2,2}^{(\tau)})K_{\hA A^{\top}}M_{1,4}^{(\tau)}=I_{\ntest}\). Using those equations, we may solve for \(M_{4,4}^{(\tau)}\) and \(M_{1,4}^{(\tau)}\). In particular, we have
    \[
        M_{1,4}^{(\tau)} = \tr(M_{2,2}^{(\tau)})((\delta - z -i\tau) I_{\ntrain}-\tr(M_{2,2}^{(\tau)})K_{AA^{\top}})^{-1}K_{A\hA^{\top}}M_{4,4}^{(\tau)}
    \]
    and \((I_{\ntest}-i\tau \Xi)M_{4,4}^{(\tau)}=i\tau I_{\ntest}\) with \(\Xi = (\tr(M_{2,2}^{(\tau)})K_{\hA\hA^{\top}}+i\tau I_{\ntrain})-(\tr(M_{2,2}^{(\tau)}))^{2}K_{\hA A^{\top}}((\delta - z -i\tau) I_{\ntrain}-\tr(M_{2,2}^{(\tau)})K_{AA^{\top}})^{-1}K_{A\hA^{\top}}\). Since \(\|M_{2,2}^{(\tau)}\|\leq (\Im[z])^{-1}\) and \(\Im[M^{(\tau)}_{2,2}]\succeq 0\) for every \(\tau\in \bbR_{>0}\), it follows that \(\|\Xi\|\) is bounded as \(\tau \to 0\). Consequently, for every \(\tau\) small enough, \(\|i\tau \Xi\|<1\) and \(I_{\ntest}-i\tau \Xi\) is invertible. In particular, using the fact that \(\|K_{AA^{\top}}\|\vee \|K_{\hA A^{\top}}\|\vee \|K_{A\hA^{\top}}\|\vee \|K_{\hA\hA^{\top}}\|\lesssim n^{-1}\), \(M_{4,4}^{(\tau)}\) approaches \(0\) as \(\tau \to 0\) uniformly in \(n\). The same argument can be applied to \(M_{1,4}^{(\tau)}\).

    We turn our attention to~\eqref{eq:RDEL} again. Expanding, it is straightforward to see that \(M_{1,2}^{(\tau)}=M_{2,1}^{(\tau)}=M_{2,4}^{(\tau)}=M_{4,2}^{(\tau)}=0\). Furthermore, we can write
    \[
        (\bbE \Lsub -\supop^{(\mathrm{sub})}(\diag\{M_{1,1}^{(\tau)},M_{2,2}^{(\tau)}\})-(z + i\tau) I)  \diag\{M_{1,1}^{(\tau)},M_{2,2}^{(\tau)}\} = I + D^{(\mathrm{sub})}
    \]
    with
    \[
        D^{(\mathrm{sub})} = \diag\{\tr(M_{2,2}^{(\tau)})K_{A\hA^{\top}}M_{4,1}^{(\tau)},\tr(K_{A\hA^{\top}}M^{(\tau)}_{4,1}+K_{\hA A^{\top}}M^{(\tau)}_{1,4} + K_{\hA\hA^{\top}} M^{(\tau)}_{4,4}) M_{2,2}^{(\tau)}\}.
    \]
    In particular, \(\diag\{M_{1,1}^{(\tau)},M_{2,2}^{(\tau)}\}\) almost solves the sub-DEL up to an additive perturbation term \(D^{(\mathrm{sub})}\) which vanishes as \(\tau\to 0\) uniformly in \(n\). Using the stability properties of the sub-DEL, it follows that \(\|M_{1,1}^{(\tau)}-M_{1,1}\|\to 0\) and \(\|M_{2,2}^{(\tau)}-M_{2,2}\|\to 0\) as \(\tau\to 0\) uniformly in \(n\). Expanding~\eqref{eq:RDEL}, we can pass this convergence to the remaining blocks. Hence, \(M^{(\tau)}\) converges to \(M\) as \(\tau\to 0\) uniformly in \(n\). This completes the proof.
\end{proof}

\begin{lemma}\label{lemma:first_deterministic_equiv_res}
    Let \(z\in \bbH\) with \(|z|< \delta \wedge 1\) and \(M\in \calM\) be the unique solution to~\eqref{eq:M_RF}. Under the settings of \Cref{theorem:rf_error}, \(\tr(U((L-z)^{-1}-M(z)))  \to 0\) almost surely as \(n\to \infty\) for every \(U\in \bbC^{\ell\times \ell}\) with \(\|U\|_{\ast}\leq 1\).
\end{lemma}
\begin{proof}
   Utilizing a similar argument as in the proof of \Cref{lemma:convergence_subDEL}, we observe that \cref{condition:flat} is satisfied, \(\lim_{n\to\infty}\sqrt{n}\lambda =0\), and \(\limsup_{n\to\infty}\|\tilde{\supop}\|=0\), where \(\lambda\) is the Lipschitz constant defined in \Cref{lemma:D_vanishes}. In particular, since \(\|\Delta(L,\tau;z)\|\to 0\) as \(n\to\infty\) for every \(\tau\in \bbR_{>0}\) by \Cref{lemma:universality}, it follows from \Cref{lemma:D_vanishes} that \(\|D^{(\tau)}\|\to 0\) as \(n\to\infty\) for every \(\tau \in \bbR_{>0}\). By \Cref{lemma:RDEL_conv_DEL}, \cref{condition:convergence_RDEL} holds. The application of \Cref{theorem:convergence} yields the desired result.
\end{proof}

Having established that the solution to~\eqref{eq:M_RF} acts as a deterministic equivalent for the pseudo-resolvent \((L-z\Lambda)^{-1}\) linked to~\eqref{eq:first_linearization}, we aim to retrieve the expression in~\eqref{eq:test_error} by taking the spectral parameter to \(0\). To accomplish this, we need further control over the DEL near the origin.

\begin{lemma}\label{lemma:assumption_holds}
    Let \(z\in \bbH\) with \(|z|< \delta\wedge 1\) and \(M\in \calM\) the unique solution to~\eqref{eq:M_RF}. Then, for every \(\epsilon\in (0,2^{-1}]\) with \((2(\delta-\Re[z])^{-2}d^{2}\|K_{AA^{\top}}\|^{2}-1-\Re[z])\epsilon \leq 2^{-1}(1+\Re[z])\),
    \[
        1-d\|K_{AA^{\top}}^{\frac{1}{2}}M_{1,1}(z)K_{AA^{\top}}^{\frac{1}{2}}\|^{2}_{F}\|M_{2,2}(z)\|^{2} \geq \epsilon.
    \]
    In particular, under the settings of \Cref{theorem:rf_error}, there exists \(\epsilon\in (0,2^{-1}]\) depending on \(0<\eta<1\wedge \delta\) such that
    \[
        \limsup_{n\to\infty}d\|K_{AA^{\top}}^{\frac{1}{2}}M_{1,1}(z)K_{AA^{\top}}^{\frac{1}{2}}\|^{2}_{F}\|M_{2,2}(z)\|^{2}<1-\epsilon
    \]
    for all \(z\in \bbH\) with \(|z|\leq \eta\).
\end{lemma}
\begin{proof}
    Fix \(z\in \bbH\) with \(|z|< 1\wedge \delta\) and write \(M\equiv M(z)\). Let \(\epsilon\in (0,2^{-1}]\) such that
    \[
        (2(\delta-\Re[z])^{-2}d^{2}\|K_{AA^{\top}}\|^{2}-1-\Re[z])\epsilon \leq 2^{-1}(1+\Re[z])
    \]
    and assume, by contradiction, that \(d\|K_{AA^{\top}}^{\frac{1}{2}}M_{1,1}(z)K_{AA^{\top}}^{\frac{1}{2}}\|^{2}_{F}\|M_{2,2}(z)\|^{2}>1-\epsilon\). Using the definition of \(M_{1,1}\) and \(M_{2,2}\) and repeatedly applying~\eqref{eq:real_part_F}, 
    \begin{align*}
        \tr(K_{AA^{\top}}\Re[M_{1,1}]) & = (\delta-\Re[z])\tr(K_{AA^{\top}}M_{1,1}M_{1,1}^{\ast})
        \\ & - \tr(\Re[M_{2,2}])\tr(K_{AA^{\top}}M_{1,1}K_{AA^{\top}}M_{1,1}^{\ast}) 
        \\ & = (\delta-\Re[z])\tr(K_{AA^{\top}}M_{1,1}M_{1,1}^{\ast}) 
        \\ & + d\|M_{2,2}\|^{2}\|K_{AA^{\top}}^{\frac{1}{2}}M_{1,1}(z)K_{AA^{\top}}^{\frac{1}{2}}\|^{2}_{F} (1+\Re[z])
        \\ & + d\|M_{2,2}\|^{2}\|K_{AA^{\top}}^{\frac{1}{2}}M_{1,1}(z)K_{AA^{\top}}^{\frac{1}{2}}\|^{2}_{F} \tr(K_{AA^{\top}}\Re[M_{1,1}])
        \\ & \geq (\delta-\Re[z])\tr(K_{AA^{\top}}M_{1,1}M_{1,1}^{\ast}) + (1-\epsilon) (1+\Re[z])
        \\ & + (1-\epsilon)\tr(K_{AA^{\top}}\Re[M_{1,1}]).
    \end{align*}
    Solving for \(\tr(K_{AA^{\top}}\Re[M_{1,1}])\), we obtain that \(\tr(K_{AA^{\top}}\Re[M_{1,1}]) = t\epsilon^{-1}\) with
    \[
        t:= (\delta-\Re[z])\tr(K_{AA^{\top}}M_{1,1}M_{1,1}^{\ast}) + (1-\epsilon) (1+\Re[z]).
    \]
    In particular, taking the real part of \(M_{2,2}\), we have
    \[
        -\Re[M_{2,2}] = \|M_{2,2}\|^{2}(1+\Re[z]+\tr(K_{AA^{\top}}\Re[M_{1,1}]))I_{d} \succeq \|M_{2,2}\|^{2}(1+\Re[z]+t\epsilon^{-1})I_{d}.
    \]
    By taking the norm on both sides and leveraging the properties that the spectral norm preserves the Loewner partial ordering and that the spectral norm of the real and imaginary parts of a complex matrix are bounded above by the spectral norm of the matrix itself, we obtain
    \[
        \|M_{2,2}\|^{2}(1+\Re[z]+t\epsilon^{-1}) \leq \|\Re[M_{2,2}]\|\leq \|M_{2,2}\|.
    \]
    Rearranging, this implies that \(\|M_{2,2}\|\leq (1+\Re[z]+t\epsilon^{-1})^{-1}\). By \Cref{lemma:control_origin} and the definition of \(\epsilon\),
    \begin{align*}
        d\|K_{AA^{\top}}^{\frac{1}{2}}M_{1,1}(z)K_{AA^{\top}}^{\frac{1}{2}}\|^{2}_{F}\|M_{2,2}(z)\|^{2} & \leq \frac{(\delta-\Re[z])^{-2}d^{2}\|K_{AA^{\top}}\|^{2}}{1+\Re[z]+t\epsilon^{-1}}
         \leq 2^{-1}. 
    \end{align*}
    This is a contradiction. Thus, it must be the case that
    \[
        1-d\|K_{AA^{\top}}^{\frac{1}{2}}M_{1,1}(z)K_{AA^{\top}}^{\frac{1}{2}}\|^{2}_{F}\|M_{2,2}(z)\|^{2}\geq \epsilon
    \]
    for every \(\epsilon\in (0,2^{-1}]\) with \((2\delta^{-2}d^{2}\|X\|^{4}-1-\Re[z])\epsilon \leq 2^{-1}(1+\Re[z])\). The result follows.
\end{proof}

The statement of \Cref{lemma:assumption_holds} is intricate because the left-hand side of the inequality
\[
    (2\delta^{-2}d^{2}\|K_{AA^{\top}}\|^{2} - 1-\Re[z])\epsilon \leq 2^{-1}(1+\Re[z])
\]
may be negative. Nevertheless, the essence of \Cref{lemma:assumption_holds} lies in the fact that the quantity \(1-d\|K_{AA^{\top}}^{1/2}M_{1,1}(z)K_{AA^{\top}}^{1/2}\|^{2}_{F}\|M_{2,2}(z)\|^{2}\) can be consistently bounded away from \(0\) regardless of the dimension.

Now that we have some control on the solution of the DEL when the spectral parameter is close to the origin, we still need to continuously extend the function \(M\) to its boundary point \(0\). 
To do so, we analytically extend \(M\) by reflection to the lower complex plane \(\{z \in \bbH \st  \Im[z] < 0\}\) through an open interval containing the origin.

\begin{lemma}\label{lemma:convergence_origin}
    The unique solution \(M\) to~\eqref{eq:DEL} can be extended analytically to the lower-half complex plane through the open interval \((-(1\wedge \delta),1\wedge \delta)\).
\end{lemma}
\begin{proof}
    Using the definition of matrix imaginary part and the resolvent identity, we obtain the system of equations \(\Im[M_{1,1}] = M_{1,1}(\Im[z]+\tr(\Im[M_{2,2}])K_{AA^{\top}})(M_{1,1})^{\ast}\) and \(\tr(\Im[M_{2,2}]) = d \|M_{2,2}\|^{2}(\Im[z]+\tr(K_{AA^{\top}}\Im[M_{1,1}]))\). Combining the two equalities, we get \(d^{-1}\tr(\Im[M_{2,2}])(1-\|\sqrt{d}K_{AA^{\top}}^{1/2}M_{1,1}K_{AA^{\top}}^{1/2}\|^{2}_{F}\|M_{2,2}\|^{2}) =  \|M_{2,2}\|^{2} \Im[z](1+ \|\sqrt{d}K_{AA^{\top}}^{1/2} M_{1,1}\|_{F}^{2})\). By \Cref{lemma:assumption_holds}, \( 1-d\|K_{AA^{\top}}^{\frac{1}{2}}M_{1,1}K_{AA^{\top}}^{\frac{1}{2}}\|^{2}_{F}\|M_{2,2}\|^{2}>0\) uniformly on \(\{z\in \bbH \st |z|\leq \epsilon\}\) for every \(0<\epsilon <1\wedge \delta\). Using \Cref{lemma:control_origin},
    \[
        d^{-1}\tr(\Im[M_{2,2}])
        \leq \frac{\Im[z](1+ \|\sqrt{d} (\delta-\Re[z])^{-1}K_{AA^{\top}}^{\frac{1}{2}}\|_{F}^{2})}{1-\|\sqrt{d}K_{AA^{\top}}^{\frac{1}{2}}M_{1,1}K_{AA^{\top}}^{\frac{1}{2}}\|^{2}_{F}\|M_{2,2}\|^{2}}
        .
    \]
    Thus, we observe that \(\Im[M_{2,2}(z)]\downarrow 0\) uniformly as \(\Im[z]\to 0\) on \((-\epsilon,\epsilon)\) and similarly for \(\|\Im[M_{1,1}]\|\). Since \(M_{1,1}\) and \(M_{2,2}\) fully define the solution of the DEL, \(\|\Im[M(z)]\|\) vanishes uniformly for \(\Re[z]\in (-\epsilon,\epsilon)\) as \(\Im[z]\to 0\). By the Stieltjes inversion lemma~\cite[Theorem 5.4(v) and Theorem 5.4(vi)]{matrix_herglotz}, the positive semidefinite measure in \Cref{theorem:main_properties} has no support in \((-\epsilon ,\epsilon)\). The result follows from~\cite[Lemma 5.6]{matrix_herglotz}.
\end{proof}

We may now remove the spectral parameter in \Cref{lemma:first_deterministic_equiv_res}.

\begin{corollary}\label{corollary:first_deterministic_equiv}
    Let \(M\in \calM\) be the unique solution to~\eqref{eq:M_RF}. Under the settings of \Cref{theorem:rf_error}, \(\tr(U(L^{-1}-M(0)))\to 0\) almost surely as \(n\to\infty\) for every sequence \(U\in \bbC^{\ell\times \ell}\) with \(\|U\|_{\ast}\leq 1\).
\end{corollary}

The solution to~\eqref{eq:M_RF} is fully defined by the scalar \(\tr(M_{2,2})\). By fully defined, we mean that we may explicitly construct the full solution of the DEL using only knowledge of \(\tr(M_{2,2})\). Using the sub-DEL defined above, we get the following numerical result.

\begin{lemma}\label{lemma:numerical_linearization_first}
    Suppose that \(M(0)\) solves~\eqref{eq:M_RF} when \(z=0\). Let 
    \[
        T:x\in \bbR_{< 0}\mapsto -(1+\tr(K_{AA^{\top}} (\delta I_{\ntrain}-dx K_{AA^{\top}})^{-1}))^{-1}\in \bbR_{< 0}
    \] 
    and consider the iterates \(\{\alpha_{k}\}_{k\in \bbN_{0}}\) obtained via \(\alpha_{k+1}=T(\alpha_{k})\) for every \(k\in \bbN\) with arbitrary \(\alpha_{0}\in \bbR_{\leq 0}\).
    Then,
    \[
    M(0) = 
    \begin{bmatrix}
        ( \delta I_{\ntrain}-d\alpha K_{AA^{\top}})^{-1} & 0 & -d\alpha  M_{1,1}(0)K_{A\hA^{\top}} & 0 \\
        0          &  \alpha I_{d} & 0 & 0  \\
        -d\alpha K_{\hA A^{\top}}M_{1,1}(0) & 0 & (d\alpha )^{2}K_{\hA A^{\top}}M_{1,1}(0)K_{A\hA^{\top}}+d\alpha K_{\hA\hA^{\top}} & -I_{\ntest} \\
        0  & 0  & -I_{\ntest} & 0
    \end{bmatrix}
\]
where \(\alpha:=d^{-1}\tr(M_{2,2}(0))=\lim_{k\to\infty}\alpha_{k}\).
\end{lemma}
\begin{proof}
    In order to use the Earle-Hamilton fixed-point theorem~\cite{earl_hamilton}, we consider the set \(\calS\) of complex matrices \(N\in \bbC^{(\ntrain+d)\times (\ntrain+d)}\) with \(N_{1,2}=N_{2,1}=0\), \(\Re[N_{1,1}]\succ 0\) and \(\Re[N_{2,2}]\prec 0\) and denote
    \[
        \F^{(\mathrm{sub})}:N\in \calS \mapsto (\bbE \Lsub-\supop^{(\mathrm{sub})}(N))^{-1}\in \calS.
    \]
    Using an argument analogous to the one in \cref{lemma:control_origin,lemma:strict_holomorphic}, we see that
    \[
        \Re[\F^{(\mathrm{sub})}_{1,1}(N)]\succeq \delta \F^{(\mathrm{sub})}_{1,1}(N)(\F^{(\mathrm{sub})}_{1,1}(N))^{\ast}\succeq \delta (\delta + d\|K_{AA^{\top}}\|\|N\|)^{-2}
    \]
    and
    \[
        \Re[\F^{(\mathrm{sub})}_{2,2}(N)]\preceq -\F^{(\mathrm{sub})}_{2,2}(N)(\F^{(\mathrm{sub})}_{2,2}(N))^{\ast}\preceq -(1+d\|K_{AA^{\top}}\|\|N\|)^{-2}
    \]
    for every \(N\in \calS\). This implies that \(\F^{(\mathrm{sub})}\) is well-defined. Furthermore, by \Cref{lemma:real/imag_norm_bound}, \(\|\F^{(\mathrm{sub})}_{1,1}(N)\|\leq \delta^{-1}\) and \(\|\Re[\F^{(\mathrm{sub})}_{2,2}(N)]\|\leq 1\) for every \(N\in \calS\). Let \(\calS_{b}=\calS\cap \{N\in \bbC^{(\ntrain+d)\times (\ntrain+d)}\st \|N\|<b\}\). Then, for every \(b>\delta \vee 1\), \(\F^{(\mathrm{sub})}\) is strictly holomorphic on \(\calS_{b}\). By the Earle-Hamilton fixed-point theorem, there exists a unique \(N\in \calS_{b}\) such that \(\F^{(\mathrm{sub})}(N)=N\). Furthermore, the sequence \(\{N_{k}\}_{k\in \bbN_{0}}\) with \(N_{k+1}=\F^{(\mathrm{sub})}(N_{k})\) for every \(k\in \bbN\) converges to \(N\) for every \(N_{0}\in \calS_{b}\). Since \(\calS=\bigcup_{b>0}\calS_{b}\), it follows that there exists a unique \(N\in \calS\) such that \(\F^{(\mathrm{sub})}(N)=N\). Also, the sequence \(\{N_{k}\}_{k\in \bbN_{0}}\) with \(N_{k+1}=\F^{(\mathrm{sub})}(N_{k})\) for every \(k\in \bbN\) converges to \(N\) for every \(N_{0}\in \calS\). Choosing \(N_{0}=\diag\{I_{\ntrain},\alpha_{0}I_{d}\}\) gives the result.
\end{proof}

\subsection{Second Deterministic Equivalent}\label{sec:second_deterministic_equiv}

We now consider the squared matrix \((AA^{\top} + \delta I_{\ntrain})^{-1}A\hA^{\top}\hA A^{\top}(AA^{\top} + \delta I_{\ntrain})^{-1}\). Notice that \((L^{-2})_{1,1} = R^{2}+RA^{\top}AR+RA\hA^{\top}\hA AR\) and \((\Lsub)^{-2}_{1,1} = R^{2}+RA^{\top}AR\) for \(R:=(AA^{\top} + \delta I_{\ntrain})^{-1}\). Rearranging, we get that
\begin{equation}\label{eq:squared_lin}
    (AA^{\top} + \delta I_{\ntrain})^{-1}A\hA^{\top}\hA A(AA^{\top} + \delta I_{\ntrain})^{-1} = (L^{-2})_{1,1} - (\Lsub)^{-2}_{1,1}.
\end{equation}

Therefore, it suffices to find deterministic equivalents for \((\Lsub)^{-2}\) and \(L^{-2}\) to obtain a  deterministic equivalent for the random matrix \((AA^{\top} + \delta I_{\ntrain})^{-1}A\hA^{\top}\hA A(AA^{\top} + \delta I_{\ntrain})^{-1}\).

\begin{lemma}\label{lemma:theory_second_lin}
    Under the settings of \Cref{theorem:rf_error}, let \(\alpha=d^{-1}\tr(M_{2,2}(0))\) as in \Cref{lemma:numerical_linearization_first}, \(R=(AA^{\top} + \delta I_{\ntrain})^{-1}\) and define
    \[
        \beta = \frac{\alpha^{2}\tr\left(K_{\hA\hA^{\top}}+d\alpha K_{\hA A^{\top}}M_{1,1}(0)(I_{\ntrain}+\delta M_{1,1}(0))K_{A\hA^{\top}}\right)}{1-\|\sqrt{d}\alpha K_{AA^{\top}}^{\frac{1}{2}}M_{1,1}(0)K_{AA^{\top}}^{\frac{1}{2}}\|_{F}^{2}}\in \bbR_{\geq 0}.
    \]
    Then, \(\tr U(RA\hA^{\top}\hA A^{\top}R-d\beta M_{1,1}(0)K_{AA^{\top}}M_{1,1}(0)-M_{1,3}(0)M_{3,1}(0)) \to 0\) almost surely as \(n\to\infty\) for every sequence \(U\in \bbC^{\ntrain\times \ntrain}\) with \(\|U\|_{\ast}\leq 1\).
\end{lemma}
\begin{proof}
    First, as stated in \Cref{lemma:analyticity_tau}, we note that \(i\tau \mapsto (L - i\tau)^{-1}\) is an analytic function with \(\partial_{i\tau}(L - i\tau)^{-1} = (L - i\tau)^{-2}\). Overloading notation, let \(M^{(\zeta)}\in \calM_{+}\) be the unique solution to the DEL \((\bbE L -\supop(M^{(\zeta)}) - \zeta I_{\ell})M^{(\zeta)} = I_{\ell}\) where \(\zeta\in \bbH\). By the proof of~\cite[Theorem 2.14]{nemish_local_2020}, the function \(\zeta \mapsto M^{(\zeta)}\) is analytic on \(\bbH\). Adapting a general argument resembling to the one in~\cite[equation (174)]{schroder2023deterministic}, it follows from Cauchy's integral formula that
    \begin{align*}
        (L - i\tau)^{-2}-\partial_{i\tau}M^{(\tau)}(0) & = \partial_{i\tau}\left((L - i\tau)^{-1}-M^{(\tau)}(0)\right)
         \\ & = \frac{1}{2\pi}\oint_{\gamma}\frac{(L-\zeta)^{-1}-M^{(\zeta)}}{(\zeta-i\tau)^{2}}\d\zeta
    \end{align*}
    where \(\gamma\) forms a counterclockwise circle of radius \(\tau/2\) around \(i\tau\). We know that \(M^{(\zeta)}\) is a deterministic equivalent for \((L-\zeta)^{-1}\) for every fixed \(\zeta\in \bbH\). By the resolvent identity, \(\zeta \mapsto (L-\zeta I_{\ell})^{-1}\) is \(4/\tau^{2}\)-Lipschitz on \(\{z\in \bbH\st \Im[z]\geq \tau/2\}\). Similarly, by the proof of~\cite[Theorem 2.14]{nemish_local_2020}, the function \(\zeta\mapsto M^{(\zeta)}\) is \((2/\tau)^{12}\)-Lipschitz on \(\{z\in \bbH\st \Im[z]\geq \tau/2\}\). Therefore, we obtain \(\tr (U( (L - i\tau)^{-2}-\partial_{i\tau}M^{(\tau)}(0))) \to 0\) almost surely as \(n\to\infty\) for every \(\tau \in \bbR_{>0}\) and \(U\in \bbC^{\ell\times \ell}\) with \(\|U\|_{\ast}\leq 1\). Taking the derivative of~\eqref{eq:RDEL}, we obtain \(\partial_{i\tau}M^{(\tau)}(0) = M^{(\tau)}(0)(\supop(\partial_{i\tau}M^{(\tau)}(0))+I_{\ell})M^{(\tau)}(0)\) or, relating this equation to the stability operator, \(\stab^{(\tau)}(\partial_{i\tau}M^{(\tau)}(0)) = (M^{(\tau)}(0))^{2}\) with \(\stab^{(\tau)}:N\in \bbC^{\ell\times \ell}\mapsto N - M^{(\tau)}(0)\supop(N)M^{(\tau)}(0)\).

    In what follows, we omit the argument of \(M^{(\tau)}\) and write \(M^{(\tau)}\equiv M^{(\tau)}(0)\). Using simple but tedious computations, we decompose \(\partial_{i\tau}M^{(\tau)}_{j,k} = C_{j,k}+D_{j,k}\tr(\partial_{i\tau} M_{2,2}^{(\tau)})\) for every \((j,k)\in \{(1,1),(1,4),(4,4)\}\) with
    \begin{align*}
        & C_{1,1}:=M_{1,4}^{(\tau)}M_{4,1}^{(\tau)} +  (M_{1,1}^{(\tau)})^{2}+M_{1,3}^{(\tau)}M_{3,1}^{(\tau)},
        \\ & D_{1,1} := M_{1,1}^{(\tau)}K_{AA^{\top}}M_{1,1}^{(\tau)}+ M_{1,1}^{(\tau)} K_{A \hA^{\top}}M_{4,1}^{(\tau)}+ M_{1,4}^{(\tau)} K_{\hA A^{\top}}M_{1,1}^{(\tau)}+ M_{1,4}^{(\tau)}K_{\hA\hA^{\top}}M_{4,1}^{(\tau)},
        \\ & C_{4,4}:= M_{4,1}^{(\tau)}M_{1,4}^{(\tau)} + (M_{4,4}^{(\tau)})^{2}+M_{4,3}^{(\tau)}M_{3,4}^{(\tau)},
        \\ & D_{4,4}:= M_{4,1}^{(\tau)}K_{AA^{\top}}M_{1,4}^{(\tau)}+ M_{4,1}^{(\tau)}K_{A\hA^{\top}}M_{4,4}^{(\tau)}+ M_{4,4}^{(\tau)}K_{\hA A^{\top}}M_{1,4}^{(\tau)}+ M_{4,4}^{(\tau)}K_{\hA\hA^{\top}}M_{4,4}^{(\tau)},
        \\ & C_{1,4}:= M_{1,1}^{(\tau)}M_{1,4}^{(\tau)} + M_{1,4}^{(\tau)}M_{4,4}^{(\tau)}+M_{1,3}^{(\tau)}M_{3,4}^{(\tau)}, \text{ and}
        \\ & D_{1,4} := M_{1,1}^{(\tau)}K_{AA^{\top}}M_{1,4}^{(\tau)}+  M_{1,1}^{(\tau)} K_{A\hA^{\top}} M_{4,4}^{(\tau)}+  M_{1,4}^{(\tau)} K_{\hA A^{\top}} M_{1,4}^{(\tau)}+ M_{1,4}^{(\tau)}K_{AA^{\top}}M_{4,4}^{(\tau)}.
    \end{align*}
    Taking the trace of the \(2,2\) block of \(\partial_{i\tau} M^{(\tau)}(0)\), we get
    \begin{align*}
        & \tr(\partial_{i\tau} M_{2,2}^{(\tau)}) = \tr((M_{2,2}^{(\tau)})^{2}) (\rho(\partial_{i\tau}M^{(\tau)})+1)
        \\ & = \tr((M_{2,2}^{(\tau)})^{2}) (\tr(K_{AA^{\top}} C_{1,1}+K_{A\hA^{\top}}C_{1,4}^{\top}+K_{\hA A^{\top}}C_{1,4} + K_{\hA\hA^{\top}} C_{4,4})+1)
        \\ & + \tr(\partial_{i\tau} M_{2,2}^{(\tau)})\tr((M_{2,2}^{(\tau)})^{2}) \tr(K_{AA^{\top}} D_{1,1}+K_{A\hA^{\top}}D_{1,4}^{\top}+K_{\hA A^{\top}}D_{1,4} + K_{\hA\hA^{\top}} D_{4,4}).
    \end{align*}
    By the proof of \Cref{lemma:first_deterministic_equiv_res}, we observe that there exists a function \(f:\bbR_{>0}\mapsto \bbR_{\geq 0}\) with \(\lim_{\tau\downarrow 0}f(\tau)=0\) such that \(\|D_{4,4}\| \leq f(\tau)+o_{n}(1)\), \(\|D_{1,1} - M_{1,1}K_{AA^{\top}}M_{1,1}\| \leq f(\tau)+ o_{n}(1)\), \(\|D_{1,4}\|=\|D_{4,1}^{\top}\|\leq f(\tau) +o_{n}(1)\) and \(\|M_{2,2}^{(\tau)}- M_{2,2}\|\leq f(\tau)+o_{n}(1)\). By \Cref{lemma:assumption_holds},
    \[
        |1-\tr((M_{2,2}^{(\tau)})^{2})\tr(K_{AA^{\top}} D_{1,1}+K_{A\hA^{\top}}D_{4,1}+K_{\hA A^{\top}}D_{1,4} + K_{\hA\hA^{\top}} D_{4,4}) |
    \]
    is bounded away from \(0\) for every \(n\in \bbN\) large enough and \(\tau\in \bbR_{>0}\) small enough. In particular, the limit \(\lim_{\tau\to 0}\partial_{i\tau} M^{(\tau)}\) exists and satisfies
    \begin{align*}
        \lim_{\tau\to 0}d^{-1}\tr(\partial_{i\tau} M_{2,2}^{(\tau)}) & =
          \frac{\alpha^{2} (\tr(K_{AA^{\top}} M^{2}_{1,1}+K_{AA^{\top}}M_{1,3}M_{3,1} + K_{\hA\hA^{\top}})+1)}{1-\|\sqrt{d}\alpha K_{AA^{\top}}^{\frac{1}{2}}M_{1,1}K_{AA^{\top}}^{\frac{1}{2}}\|_{F}^{2}}
          \\ & -   \frac{\alpha^{2} \tr(K_{A\hA^{\top}}M_{3,1}+K_{\hA A^{\top}}M_{1,3})}{1-\|\sqrt{d}\alpha K_{AA^{\top}}^{\frac{1}{2}}M_{1,1}K_{AA^{\top}}^{\frac{1}{2}}\|_{F}^{2}}
          \\ & = \beta + \frac{\alpha^{2}(1+\tr(K_{AA^{\top}} M^{2}_{1,1}))}{1-\|\sqrt{d}\alpha K_{AA^{\top}}^{\frac{1}{2}}M_{1,1}K_{AA^{\top}}^{\frac{1}{2}}\|_{F}^{2}}
    \end{align*}
    where we recall that \(\alpha=d^{-1}\tr(M_{2,2})\) as defined in \Cref{lemma:numerical_linearization_first}. Plugging this into the expression for \(\partial_{i\tau} M^{(\tau)}_{1,1}\) and taking the limit as \(\tau\downarrow 0\), we get that
    \[
        d\beta M_{1,1}K_{AA^{\top}}M_{1,1}+  M_{1,3}M_{3,1} +M^{2}_{1,1} + \frac{d\alpha^{2}(1+\tr(K_{AA^{\top}} M^{2}_{1,1}))}{1-\|\sqrt{d}\alpha K_{AA^{\top}}^{\frac{1}{2}}M_{1,1}K_{AA^{\top}}^{\frac{1}{2}}\|_{F}^{2}}M_{1,1}K_{AA^{\top}}M_{1,1}
    \]
    is an asymptotic deterministic equivalent for \((L^{-2})_{1,1}\).

    Using a similar argument, we note that \(\partial_{z}M^{(\mathrm{sub})}(0)\) is a deterministic equivalent for \((L^{(\mathrm{sub})})^{-2}\) and
    \[
        \partial_{z}M^{(\mathrm{sub})}_{1,1} = M_{1,1}^{2}+\frac{d \alpha^{2}(1+\tr(K_{AA^{\top}}M^{2}_{1,1}))}{1-\|\sqrt{d}\alpha K_{AA^{\top}}^{\frac{1}{2}}M_{1,1}K_{AA^{\top}}^{\frac{1}{2}}\|_{F}^{2}}M_{1,1}K_{AA^{\top}}M_{1,1}.
    \]
    We obtain the result by~\eqref{eq:squared_lin}.
\end{proof}

\section{Concluding Remarks}\label{sec:conclusion}

In conclusion, our work extends the Dyson equation framework to handle linearizations with general correlation structures. We establish a global anisotropic law for pseudo-resolvents, showing that under suitable conditions, generalized trace entries of the pseudo-resolvent converge to those of the solution of the Dyson equation. In contrast with much of the existing literature, our approach focuses on the macroscale properties of structured linearizations, allowing us to relax the usual independence assumptions and to accommodate correlated block structures. The conditions we impose include standard boundedness assumptions, a technical stability condition on the linearization, and a Gaussian design assumption on its random components. Within this framework, we derive a deterministic equivalent and a Gaussian equivalence theorem for the empirical test error of random-features ridge regression, where the training and test data are deterministic and the random weight matrix is obtained by applying an entrywise Lipschitz function to a matrix with independent standard Gaussian entries.

Looking forward, several open questions remain. While the boundedness and stability conditions are somewhat fundamental to our analysis, the Gaussian design assumption could likely be relaxed. Doing so would require replacing the Gaussian concentration arguments used here with more general probabilistic tools to ensure that the error terms vanish asymptotically and that the generalized trace entries of the pseudo-resolvent concentrate around their expectations. The current stability condition, though essential for our proofs, is technical and somewhat difficult to handle. Simplifying it through direct conditions on the linearization or the underlying matrix-valued polynomial would be highly valuable.. Additionally, our empirical simulations (e.g., \Cref{fig:various_activation}) suggest that the result might extend beyond Lipschitz activation functions, warranting further exploration. Moreover, our assumption that the norm of the random features matrix is bounded in expectation excludes low rank spike scenario, in which the norm of the random features matrix is unbounded due to the presence of a low rank spike but the spectrum is still well-behaved. Addressing this by separating bulk and spike analysis could be an interesting direction for future research.

\begin{appendix}
    
\renewcommand{\thelemma}{A.\arabic{lemma}}
    \renewcommand{\theproposition}{A.\arabic{proposition}}
    
    \section{Preliminary Results}\label{sec:prelim}
    
    \subsection{Carathéodory-Riffen-Finsler Pseudometric}\label{subsec:CRF}
    
    Let \(\calD\) be a domain in a normed linear space over the complex numbers \(\calX\). We define the \emph{infinitesimal Carathéodory-Riffen-Finsler (CRF)-pseudometric} as
    \[
        \alpha : (x,v)\in \calD\times \calX\mapsto \sup\{\|\D f(x)v\| \st f\in \Hol(\calD,\bbD)\}\in \bbR,
    \]
    where \(\bbD\) denotes the open complex unit ball of unit radius and \(\Hol(\calD, \bbD)\) denotes the set of holomorphic functions from \(\calD\) to \(\bbD\)~\cite{harris_fixed_2003,HARRIS1979345}. The pseudometric \(\alpha\) is an \emph{infinitesimal Finsler pseudometric} on \(\calD\), meaning that it is non-negative, lower semicontinuous, locally bounded, and satisfies \(\alpha(x,tv)=|t|\alpha(x,v)\) for every \((x,v)\in \calD\times \calX\) and \(t\in \bbR\).
    
    Let \(\Gamma\) be the set of all curves in \(\calD\) with piecewise continuous derivatives, referred to as \emph{admissible} curves, and define
    \[
        \calL:\gamma \in \Gamma \mapsto \int_{0}^{1}\alpha(\gamma(y),\gamma^{\prime}(t))\mathrm{d} t \in \bbR.
    \]
    The pseudometric \(\alpha\) is a seminorm at each point in \(\calD\), and we interpret \(\calL(\gamma)\) as the length of the curve \(\gamma\) measured with respect to \(\alpha\)~\cite{harris_fixed_2003}. Then, the \emph{CRF-pseudometric} \(\rho\) of \(\calD\) is defined as
    \[
        \rho:(x,y)\in \calD^{2} \mapsto \inf\{\calL(\gamma)\st \gamma\in \Gamma,\, \gamma(0)=x,\, \gamma(1)=y\}\in \bbR_{\geq 0}.
    \]
    As the name suggests, \(\rho\) is a pseudometric.
    
    Our main tool for extrapolating results about norms is the Schwarz-Pick inequality, which we state here for completeness.
    \begin{proposition}[{\cite[Proposition 3]{HARRIS1979345}}]\label{prop:hol_inv}
        Let \(\calD_{1}\) and \(\calD_{2}\) be domains in complex normed vector spaces, and let \(\rho_{1}\) and \(\rho_{2}\) be the associated CRF-pseudometrics. If \(f:\calD_{1}\mapsto \calD_{2}\) is holomorphic, then \(\rho_{2}(f(x),f(y))\leq \rho_{1}(x,y)\) for all \(x,y\in \calD_{1}\).
    \end{proposition}
    
    In fact, the inequality in \Cref{prop:hol_inv} can be replaced by an equality when the function \(f\) is a biholomorphic mapping. This means that the CRF-pseudometric is biholomorphically invariant~\cite{helton2007operatorvalued}. In some sense, \Cref{prop:hol_inv} indicates that the CRF-pseudometric is non-expansive on the space of holomorphic functions mapping a domain onto itself.
    
    If we denote by \(\rho_{\bbD}\) the CRF-pseudometric on the complex open unit disk \(\bbD\), then \Cref{prop:hol_inv} becomes particularly useful because \(\rho_{\bbD}\), also known as the \emph{Poincaré metric}, admits the closed form expression
    \begin{equation}\label{eq:poincare}
        \rho_{\Delta}(z_{1},z_{2}) = \mathrm{arctanh}\left|\frac{z_{1}-z_{2}}{1-\bar{z_{1}}z_{2}}\right|.
    \end{equation}
    For a derivation of~\eqref{eq:poincare}, refer to~\cite[Example 2]{HARRIS1979345}.
    
    \subsection{General Matrix Identities}
    
    \begin{lemma}\label{lemma:res_trick}
        If \(M_{1},M_{2}\in \bbC^{n\times n}\) are non-singular, then \(M^{-1}_{1}-M_{2}^{-1}=M_{1}^{-1}(M_{2}-M_{1})M_{2}^{-1}\).
    \end{lemma}
    \begin{proof}
        Multiply on the left by \(M_{1}\) and on the right by \(M_{2}\).
    \end{proof}
    
    \begin{lemma}\label{lemma:neumann}
        Let \(z\in \bbC\), \(M\in \bbC^{n\times n}\) and assume that \(\| M\| \leq a < b\leq |z|\) from some \(a,b\in \bbR_{\geq 0}\). Then, \(M-z I_{n}\) is non-singular and \(\|(M-z I_{n})^{-1}\|\leq (b-a)^{-1}\).
    \end{lemma}
    \begin{proof}
        For every \(v\in \bbC^{n}\), we have \(\|(M-z I_{n})v\| \geq \|z v\| - \|Mv\| \geq (|z|-\|M\|)\|v\|\). This implies that \(M-z I_{n}\) is non-singular. Choosing \(v=(M-z I_{n})^{-1}u\) for some unit vector \(u\), we have \(\|(M-z I_{n})^{-1}u\| \leq (|z|-\|M\|)^{-1} \). Taking the supremum over unit vectors \(u\) and using the definition of spectral norm, we obtain the desired result.
    \end{proof}
    
    \begin{lemma}
        For every \(M_{1},M_{2}^{\top}\in \bbC^{n\times d}\) and \(z\in \bbC\) such that both \(M_{1}M_{2}-z I_{n}\) and \(M_{2}M_{1}^{\top}-z I_{d}\) are non-singular, we have \(M_{1}(M_{2}M_{1}-z I_{d})^{-1} = (M_{1}M_{2}-z I_{n})^{-1}M_{1}\).
    \end{lemma}
    \begin{proof}
        Left-multiply the equation on both sides by \(M_{1}M_{2}-zI_{n}\) and right-multiply by \(M_{2}M_{1}-zI_{d}\).
    \end{proof}

    \subsection{Real and Imaginary Parts of Matrices}
    Apart from general matrix identities, we will also need to consider the real and imaginary parts of matrices. Just like complex numbers, we can decompose a complex matrix \(M\in \bbC^{n\times n}\) as \(M=\Re[M]+i\Im[M]\) where \(2\Re[M]=M+M^{\ast}\) and \(2i\Im[M]=M-M^{\ast}\). The real and imaginary parts of \(M\) are Hermitian. The following lemma states that the norm of the real and imaginary parts of a matrix are bounded by the norm of the matrix itself.
    
    \begin{lemma}\label{lemma:real/imag_norm_bound}
        For every \(M\in \bbC^{n\times n}\), \(\|\Re[M]\| \vee \| \Im[M]\| \leq \|M\|\).
    \end{lemma}
    \begin{proof}
        Let \(v\in \bbC^{n}\) be a complex unitary vector. By Cauchy-Schwarz's inequality, \(\|Mv\|^{2} = \|Mv\|^{2}\|v\|^{2} \geq |v^{\ast}M v|^{2} = |v^{\ast}\Re[M]v+iv^{\ast}\Im[M]v|^{2}\). Since both \(\Re[M]\) and \(\Im[M]\) are Hermitian, the quadratic forms \(v^{\ast}\Re[M]v\) and \(v^{\ast}\Im[M]v\) are real. Hence, \(|v^{\ast}\Re[M]v+iv^{\ast}\Im[M]v|^{2} = (v^{\ast}\Re[M]v)^{2}+(v^{\ast}\Im[M]v)^{2}\). Taking the supremum over all unitary vectors \(v\), we obtain the desired result.
    \end{proof}
    
    The proof of \Cref{lemma:real/imag_norm_bound} is as crucial as the statement itself, if not more so. For instance, the following lemma follows directly from this argument.
    
    \begin{lemma}\label{lemma:real/imag_singularity_bound}
        Let \(M\in \bbC^{n\times n}\). If there exists \(a\in \bbR_{>0}\) such that \(\Re[M] \succeq a I_{n}\), \(\Re[M] \preceq -a I_{n}\), \(\Im[M] \succeq a I_{n}\) or \(\Im[M] \preceq -a I_{n}\), then \(M\) is non-singular and \(\|M^{-1}\|\leq a^{-1}\).
    \end{lemma}
    \begin{proof}
        Assume that \(\Re[M] \succeq a I_{n}\) or \(\Re[M] \preceq -a I_{n}\). By the proof of \Cref{lemma:real/imag_norm_bound}, we have \(\|Mv\| \geq |v^{\ast}\Re[M]v| \geq a\) for every unitary \(v\in \bbC^{n}\). Hence, \(M\) is non-singular. Taking \(v=\frac{M^{-1}u}{\|M^{-1}u\|}\) for some unitary \(u\in \bbC^{n}\setminus \{0\}\), we have \(\|M^{-1}u\| \leq a^{-1}\). Taking the supremum over all unitary \(u\) gives the result first half of the result. The second half follows similarly.
    \end{proof}
    
    Another important lemma relates the real and imaginary parts of an inverse matrix to those of the original matrix.
    
    \begin{lemma}\label{lemma:real/imag_inverse}
        Let \(M\in \bbC^{n\times n}\) be invertible. Then, \(\Re[M^{-1}]=M^{-1}\Re[M]M^{-\ast}\) and \(\Im[M^{-1}]=-M^{-1}\Im[M]M^{-\ast}\).
    \end{lemma}
    \begin{proof}
        Write \(M=\Re[M] + i\Im[M]\). Since the matrix \(\Re[M]\) is Hermitian and \(i\Im[M]\) is skew-Hermitian, we have \(M^{\ast} = (\Re[M]+i\Im[M])^{\ast} = \Re[M]-i\Im[M]\). By the definition of matrix real and imaginary parts as well as \Cref{lemma:res_trick}, \(2\Re[M^{-1}] = M^{-1} + M^{-\ast} = M^{-1}(M^{\ast}+M)M^{-\ast}=2M^{-1}\Re[M]M^{-\ast}\). The proof for the imaginary part is similar.
    \end{proof}
    
    \subsection{Matrix Norms}
    
    \begin{lemma}\label{lemma:frob_vs_spectral}
        Let \(M_{1}\in \bbC^{n\times d}\) and \(M_{2}\in \bbC^{d \times m}\) be arbitrary matrices. Then, \(\| M_{1}M_{2}\|_{F} \leq \|M_{1}\|\|M_{2}\|_{F}\) and \(\| M_{1}M_{2}\|_{F} \leq \|M_{1}\|_{F}\|M_{2}\|\).
    \end{lemma}
    \begin{proof}
        By definition, \(\| M_{1}M_{2}\|_{F}^{2} = \tr(M^{\ast}_{2}M_{1}^{\ast}M_{1}M_{2})\). Using the cyclic property of the trace, \(\tr(M^{\ast}_{2}M_{1}^{\ast}M_{1}M_{2})=\tr(M_{2}M^{\ast}_{2}M_{1}^{\ast}M_{1})\). Let \(M_{2}M^{\ast}_{2}=U\Lambda U^{\ast}\) for some unitary \(U\in \bbC^{d\times d}\) and real positive semidefinite diagonal \(\Lambda\). With \(W = U^{\ast}M_{1}^{\ast}M_{1} U\), we have \(\tr(M^{\ast}_{2}M_{1}^{\ast}M_{1}M_{2})=\tr(M_{2}M^{\ast}_{2}M_{1}^{\ast}M_{1}) = \sum_{j=1}^{d}\Lambda_{j,j}W_{j,j}\). Indeed, as \(\|M_{2}\|^{2}=\|M_{2}M^{\ast}_{2}\| = \max_{j}\Lambda_{j,j}^{2}\), we have \(\sum_{j=1}^{d}\Lambda_{j,j}W_{j,j} \leq \|M_{2}\|^{2}\tr(U^{\ast}M_{1}^{\ast}M_{1} U)= \|M_{2}\|^{2}\|M_{1}\|^{2}_{F}\). A similar argument can be made for the second inequality.
    \end{proof}
    
    The second inequality related the trace to the spectral norm and the nuclear norm.
    
    \begin{lemma}\label{lemma:trace_norm}
        Let \(M,U\in \bbC^{n\times n}\). Then, \(|\tr(UM)| \leq \|U\|_{\ast}\|M\|\).
    \end{lemma}
    \begin{proof}
        Let \(\{u_{j}\}_{j=1}^{n}\) and \(\{m_{j}\}_{j=1}^{n}\) be a non-increasing enumeration of the singular values of \(U\) and \(M\), respectively. By Von Neumann's trace inequality, \(|\tr(UM)|\leq \sum_{j=1}^{n}u_{j}m_{j} \leq m_{1}\sum_{j=1}^{n}u_{j}=\|U\|_{\ast}\|M\|\).
    \end{proof}
    
    \subsection{Block Matrices}

    The invertibility of block matrices is closely related to the invertibility of their so-called Schur complements. If \(M\) is a block matrix of the form
    \[
        M=\begin{bmatrix}
            A & B\\
            C & D
        \end{bmatrix}
    \]
    with \(A\in \bbC^{n\times n}\), \(B,C^{\top}\in \bbC^{n\times d}\) and \(D\in \bbC^{d\times d}\), then the Schur complement of \(A\) in \(M\) is defined as \(D-CA^{-1}B\) provided that \(A\) is non-singular. Similarly, the Schur complement of \(D\) in \(M\) is defined as \(A-BD^{-1}C\) provided that \(D\) is non-singular. The following lemma, known as the block matrix inversion lemma, relates the invertibility of \(M\) to that of its Schur complements.
    
    \begin{lemma}[Block matrix inversion lemma~{\cite[Theorem 2.1]{block_matrices}}]\label{lemma:block_inversion}
        Let
        \[
            M=\begin{bmatrix}
                A & B\\
                C & D
            \end{bmatrix}
        \]
        be a block matrix with \(A\in \bbC^{n\times n}\), \(B,C^{\top}\in \bbC^{n\times d}\) and \(D\in \bbC^{d\times d}\). If \(A\) is non-singular, then \(M\) is non-singular if and only if the Schur complement of \(A\), namely \(D-CA^{-1}B\), is non-singular. In this case, the inverse of \(M\) is given by
        \[
            M^{-1} = \begin{bmatrix}
                A^{-1} + A^{-1}B(D-CA^{-1}B)^{-1}CA^{-1} & -A^{-1}B(D-CA^{-1}B)^{-1}\\
                -(D-CA^{-1}B)^{-1}CA^{-1} & (D-CA^{-1}B)^{-1}.
            \end{bmatrix}.
        \]
        Alternatively, if \(D\) is non-singular, then \(M\) is non-singular if and only if the Schur complement of \(D\), namely \(A-BD^{-1}C\), is non-singular. In this case, the inverse of \(M\) is given by
        \[
            M^{-1} = 
            \begin{bmatrix}
                (A-BD^{-1}C)^{-1} & -(A-BD^{-1}C)^{-1}BD^{-1}\\
                -D^{-1}C(A-BD^{-1}C)^{-1} & D^{-1}+D^{-1}C(A-BD^{-1}C)^{-1}BD^{-1}.
            \end{bmatrix}
        \]
    \end{lemma}
    \begin{proof}
        The first part of the lemma follows from the decomposition
        \[
            M = \begin{bmatrix}
                I_{n} & 0_{n\times d} \\
                C A^{-1} & I_{d}
            \end{bmatrix}
            \begin{bmatrix}
                A & 0_{n\times d} \\
                0_{d\times n} & D-CA^{-1}B
            \end{bmatrix}
            \begin{bmatrix}
                I_{n} & A^{-1}B \\
                0_{d\times n} & I_{d}
            \end{bmatrix}
        \]
        and the fact that the triangular block matrices are non-singular if and only if their diagonal blocks are non-singular. The second part follows from a similar decomposition.
    \end{proof}

    \subsection{Positivity-Preserving Linear Maps on Matrices}
    We record some observations about positivity-preserving linear maps on matrices. Recall that a linear map \(\supop: \bbC^{n\times n}\to \bbC^{n\times n}\) is positivity-preserving if \(\supop(M)\succeq 0\) whenever \(M\succeq 0\).

    \begin{lemma}\label{lemma:lin_pos_preserve_self_adjoint}
        Suppose that \(\supop: \bbC^{n\times n}\to \bbC^{n\times n}\) is a positivity-preserving linear map. Then, \(\supop\) maps self-adjoint matrices to self-adjoint matrices.
    \end{lemma}
    \begin{proof}
        Let \(M\in \bbC^{n\times n}\) be self-adjoint. We may write \(M=M_{1}-M_{2}\) where \(M_{1},M_{2}\succeq 0\). Since \(\supop\) is linear and positivity-preserving, \(\supop(M)=\supop(M_{1})-\supop(M_{2})\) is the difference of two self-adjoint matrices, and hence self-adjoint.
    \end{proof}

    \begin{lemma}\label{lemma:lin_pos_symmetric}
        Suppose that \(\supop: \bbC^{n\times n}\to \bbC^{n\times n}\) is a positivity-preserving linear map. Then, for any \(M\in \bbC^{n\times n}\), \((\supop(M))^{\ast} = \supop(M^{\ast})\). In particular, it follows that \(\Im[\supop(M)] = \supop(\Im[M])\) and \(\Re[\supop(M)] = \supop(\Re[M])\).
    \end{lemma}
    \begin{proof}
        Decompose \(M=\Re[M] + i\Im[M]\). By \Cref{lemma:lin_pos_preserve_self_adjoint}, \(\supop(\Re[M])\) and \(\supop(\Im[M])\) are self-adjoint. Hence, \((\supop(M))^{\ast} = \supop(\Re[M]) - i\supop(\Im[M]) = \supop(M^{\ast})\).
    \end{proof}
    
    \section{Proof of Technical Lemmas}\label{sec:tecnical_lemmas}
    
    \renewcommand{\thelemma}{B.\arabic{lemma}}
    \renewcommand{\theproposition}{B.\arabic{proposition}}
    
    \subsection{{Proof of \Cref{lemma:dist_gaussian}}}
    \begin{proof}
        Let \(j,k\in \{1,2,\ldots, \ell\}\) be arbitrary. For notational convenience, let \(\cov\) be a \(\ell\times \ell\times \gamma\)-dimensional tensor and \(g\sim \calN(0,I_{\gamma})\) such that \(L=\cov(g)+\bbE L\) and \([\cov(g)]_{j,k} = \cov_{j,k,\alpha}g_{\alpha}\). Here, we use Einstein's notation which means that we sum over every subscript appearing at least two times in a given expression. By Stein's lemma,
        \begin{align*}
            \bbE\left[(L-\bbE L)(L-z\Lambda -i\tau I_{\ell})^{-1}\right]_{j,k} & = \bbE\left[\cov_{j,m,\alpha}g_{\alpha}(L-z\Lambda -i\tau I_{\ell})^{-1}_{m,k}\right]
            \\ & = \bbE\left[\cov_{j,m,\alpha}\frac{\partial (L-z\Lambda -i\tau I_{\ell})^{-1}_{m,k}}{\partial g_{\alpha}}\right]
        \end{align*}
        Let \(e_{\alpha}\in \bbR^{\gamma}\) be the \(\alpha\)-th canonical basis
        vector, \(\delta\in \bbR_{>0}\) and \(L_{\delta}=\cov (g+\delta
        e_{\alpha})+\bbE L\). Then,
        \begin{align*}
            (L_{\delta}-z\Lambda -i\tau I_{\ell})^{-1}_{m,k} &- (L-z\Lambda -i\tau I_{\ell})^{-1}_{m,k} 
            \\ & = \left[(L_{\delta}-z\Lambda -i\tau I_{\ell})^{-1}(L-L_{\delta})(L-z\Lambda -i\tau I_{\ell})^{-1} \right]_{m,k}
            \\ & = -\delta\left[(L_{\delta}-z\Lambda -i\tau I_{\ell})^{-1}\cov( e_{\alpha})(L-z\Lambda -i\tau I_{\ell})^{-1} \right]_{m,k}.
        \end{align*}
        Taking the limit of the quotient of this difference with \(\delta\) as it approaches \(0\), we get that
        \[
            \frac{\partial (L-z\Lambda -i\tau I_{\ell})^{-1}_{m,k}}{\partial g_{\alpha}} = - \left[(L-z\Lambda -i\tau
                I_{\ell})^{-1}\cov(e_{\alpha})(L-z\Lambda -i\tau I_{\ell})^{-1}\right]_{m,k}
        \]
        and, consequently, \(\bbE[(L-\bbE L)(L-z\Lambda -i\tau I_{\ell})^{-1}]_{j,k}=-\bbE[\cov_{j,m,\alpha}(L-z\Lambda -i\tau I_{\ell})^{-1}_{m,a}\cov_{a,b,\alpha}(L-z\Lambda -i\tau I_{\ell})^{-1}_{b,k}]\). Note that \(\bbE[(L-\bbE L)^{\top}W(L-\bbE L)]_{j,k}= \bbE[\cov_{j,a,\alpha}g_{\alpha}W_{a,b}\cov_{b,k,\beta}g_{\beta}]= \bbE[\cov_{j,a,\alpha}W_{a,b}\cov_{b,k,\alpha}]\) for every \(W\in \bbR^{\ell\times \ell}\) independent of \(L\).
    \end{proof}
    
    \subsection{{Proof of \Cref{lemma:universality}}}
    Fix \(z\in \bbH\), let \(\{a_{j}\}_{j=1}^{d}\), \(\{\hat{a}_{j}\}_{j=1}^{d}\) denote the columns of \(A\) and \(\hA\) respectively. Suppose that \(l^{\top}_{j}=(a_{j}^{\top},0,0,\ha^{\top}_{j})\) and \(L_{j} = l_{j}e_{\ntrain+j}^{\top} + e_{\ntrain+j}l_{j}^{\top}\) for every \(j\in\{1,2,\ldots, d\}\), where \(\{e_{j}\}_{j=1}^{\ell}\) is the canonical basis of \(\bbR^{\ell}\). In particular, we may write the linearization in~\eqref{eq:first_linearization} as \(L=\bbE L + \sum_{j=1}^{d}L_{j}\). For every \(j\in\{1,2,\ldots, d\}\), let \(P_{j}\in \bbR^{\ell\times \ell}\) be the orthogonal matrix permuting the first and \(\ntrain+j\)th entries exclusively and \(C_{j}\in \bbR^{(\ell-1)\times (\ell-1)}\) be the matrix cycling from position \(\ntrain+j-1\) to \(1\). For instance, if \(v=(v_{k})_{k=1}^{\ell-1}\), then
    \[
        v^{\top}C^{-1}_{j} = (v_{2},v_{3},\ldots, v_{j-1},v_{1},v_{j},v_{j+1},\ldots, v_{\ell-1}).
    \]
    We will rely heavily on a Schur complement decomposition of \((P_{j}LP_{j}-zI_{\ell})^{-1}\). For every \(j\in \{1,2,\ldots, d\}\), let \(l_{-j}\in \bbR^{\ell-1}\) be obtained by removing the \(\ntrain+j\)th entry of \(l_{j}\) and \(L_{-j}\in \bbR^{(\ell-1)\times (\ell-1)}\) be obtained by removing the \(\ntrain+j\)th columns and \(\ntrain+j\)th row from \(L\). Define the scalar \(\xi_{j}:=(1+z+l_{-j}^{\top}(L_{-j}-zI_{\ell-1})^{-1}l_{-j})^{-1}\)
    and the matrix
    \[
        \Xi_{j} := C_{j}(L_{-j}-zI_{\ell-1})^{-1} C_{j} -\xi_{j}C_{j}(L_{-j}-zI_{\ell-1})^{-1} l_{j}l_{j}^{\top} (L_{-j}-zI_{\ell-1})^{-1}C_{j}.
    \]
    We have the following block matrix inversion formula.
    \begin{lemma}\label{lemma:schur_PLP}
        For every \(j\in \{1,2,\ldots, d\}\) and \(z\in \bbH\), 
        \[
            (P_{j}LP_{j}-zI_{\ell})^{-1}
            =
            \begin{bmatrix}
                -\xi_{j} & \xi_{j} l_{-j}^{\top}(L_{-j}-z I_{\ell-1})^{-1}C_{j} \\
                \xi_{j} C_{j}(L_{-j}-zI_{\ell-1})^{-1}l_{-j} & \Xi_{j}
            \end{bmatrix}.
        \]
    \end{lemma}
    \begin{proof}
        The lemma follows directly from the observation
        \[
        P_{j}LP_{j}  = 
        \begin{bmatrix}
            -1 & l^{\top}_{-j}C_{j}^{-1} \\
            C_{j}^{-1}l_{-j} & C^{-1}_{j}L_{-j}C^{-1}_{j}
        \end{bmatrix}
        \]
        and an application of \Cref{lemma:block_inversion}.
    \end{proof}
    
    For every \(j\in\{1,2,\ldots, d\}\), let \(  q_{j}=l_{-j}^{\top}R_{-j} l_{-j}\) and \(R_{-j}:=(L_{-j}-zI_{\ell})^{-1}\).
    Concentration of bilinear forms is a central ingredient of many random matrix theory proof. 
    We obtain a concentration result for \(  q_{j}\) by adapting~\cite[Lemma 4]{louart_random-matrix-approach}.
    \begin{lemma}\label{lemma:quadratic_concentration}
        Under the settings of \Cref{theorem:rf_error}, \(\lim_{n\to\infty}\bbE[\max_{1\leq j\leq d}|  q_{j}-\bbE   q_{j}|^{2}]=0\) for every \(z\in \bbH\).
    \end{lemma}
    \begin{proof}
        Adapting~\cite[Lemma 4]{louart_random-matrix-approach}, there exists some absolute constants \(c_{1},c_{2}\in \bbR_{>0}\) such that
        \[
            \bbP\left(|  q_{j}-\bbE   q_{j}| >t\right) \leq c_{1}e^{-c_{2}n\min\{t,t^{2}\}}
        \]
        for every \(t\in \bbR_{\geq 0}\). Then, \(\bbE[\max_{1\leq j\leq d}|  q_{j}-\bbE   q_{j}|^{2}]  \leq n^{-\frac{1}{2}} + \int_{n^{-\frac{1}{2}}}^{1}\bbP(\max_{1\leq j\leq d}|  q_{j}-\bbE   q_{j}|^{2} > t)\d t + \int_{1}^{\infty}\bbP(\max_{1\leq j\leq d}|  q_{j}-\bbE   q_{j}|^{2} > t)\d t\). Using a union bound,
        \begin{align*}
            \int_{n^{-\frac{1}{2}}}^{1}\bbP\left(\max_{1\leq j\leq d}|  q_{j}-\bbE   q_{j}|^{2} > t\right)\d t & \leq c_{1}d \int_{n^{-\frac{1}{2}}}^{1}e^{-c_{2}nt}\d t
             = \frac{c_{1}d}{c_{2}n}\left(e^{-c_{2}\sqrt{n}}-e^{-c_{2}n}\right)
        \end{align*}
        Also,
        \begin{align*}
            \int_{1}^{\infty}\bbP\left(\max_{1\leq j\leq d}|  q_{j}-\bbE   q_{j}|^{2} > t\right)\d t  \leq c_{1}d\int_{1}^{\infty}e^{-c_{2}n\sqrt{t}}\d t
            & = 2c_{1}d\int_{1}^{\infty}t e^{-c_{2}nt}\d t 
            \\ & = \frac{2c_{1}d}{c_{2}n} e^{-c_{2}n}\left(1+\frac{1}{c_{2}n}\right)
        \end{align*}
        Taking \(n\to\infty\) and using the fact that \(d\propto n\) concludes the proof.
    \end{proof}
    
    We are ready to prove \Cref{lemma:universality}.
    
    \begin{proof}(\emph{Proof of \Cref{lemma:universality}})
        For simplicity, we will demonstrate that \(\lim_{n\to\infty}\| \bbE[(L-\bbE L)(L-z I_{\ell})^{-1}]+\bbE[(\tilde{L}-\bbE L)(L-z I_{\ell})^{-1}(\tilde{L}-\bbE L)(L-z I_{\ell})^{-1}]\|=0\) for every \(z\in \bbH\), where \(\tilde{L}\) is an i.i.d. copy of \(L\). This adjustment streamlines notation without altering any steps in the proof. For every \(j\in \{1,2,\ldots, d\}\),
        \[
            P_{j}L_{j}P_{j} =
            \begin{bmatrix}
                0 & l_{-j}^{\top}C^{-1}_{j} \\
                C^{-1}_{j} l_{-j} & 0
            \end{bmatrix}
        \]
        and, by \Cref{lemma:schur_PLP},
        \begin{align*}
            \bbE[(L-\bbE L)(L-z&I_{\ell})^{-1}]  = \sum_{j=1}^{d}P_{j}\bbE\left[P_{j}L_{j}P_{j}(P_{j}LP_{j}-zI_{\ell})^{-1}\right]P_{j}  
            \\ & = \sum_{j=1}^{d}P_{j}\bbE
                \begin{bmatrix}
                    \xi_{j} l_{-j}^{\top}R_{-j} l_{-j} & l_{-j}^{\top}R_{-j}C_{j} - \xi_{j}l_{-j}^{\top}R_{-j}l_{-j}l_{-j}^{\top}R_{-j}C_{j} \\
                    -\xi_{j} C_{j}^{-1}l_{-j} & \xi_{j}C^{-1}_{j}l_{-j}l_{-j}^{\top}R_{-j}C_{j}
                \end{bmatrix}
            P_{j} 
            \\ & = \sum_{j=1}^{d}P_{j}\bbE
                \begin{bmatrix}
                    \xi_{j} l_{-j}^{\top}R_{-j} l_{-j} &  - \xi_{j}l_{-j}^{\top}R_{-j}l_{-j}l_{-j}^{\top}R_{-j}C_{j} \\
                    -\xi_{j} C^{-1}_{j}l_{-j} & \xi_{j}C^{-1}_{j}l_{-j}l_{-j}^{\top}R_{-j}C_{j}
                \end{bmatrix}
            P_{j}   
        \end{align*}
    where we recall that \(R_{-j}=(L_{-j} -zI_{\ell-1})^{-1}\). On the other hand,
    \begin{align*}
        & \bbE\left[(\tilde{L}-\bbE L)(L-zI_{\ell})^{-1}(\tilde{L}-\bbE L)(L-zI_{\ell})^{-1}\right] 
        \\ & = \sum_{j=1}^{d}P_{j}\bbE\left[P_{j}\tilde{L}_{j}P_{j}(P_{j}LP_{j}-zI_{\ell})^{-1}P_{j}\tilde{L}_{j}P_{j}(P_{j}LP_{j}-zI_{\ell})^{-1}\right]P_{j}  
        \\ & = \sum_{j=1}^{d}P_{j}\bbE
            \begin{bmatrix}
                \xi_{j} \tilde{l}_{-j}^{\top}R_{-j} l_{-j} & \tilde{l}_{-j}^{\top}R_{-j}C_{j} - \xi_{j}\tilde{l}_{-j}^{\top}R_{-j}l_{-j}l_{-j}^{\top}R_{-j}C_{j} \\
                -\xi_{j} C^{-1}_{j}\tilde{l}_{-j} & \xi_{j}C^{-1}_{j}\tilde{l}_{-j}l_{-j}^{\top}R_{-j}C_{j}
            \end{bmatrix}^{2}
        P_{j}.  
    \end{align*}
    Thus,
    \begin{align*}
        \sum_{j=1}^{d}P_{j}\bbE[\xi_{j}\Psi_{j}]P_{j}  & = \bbE\left[(L-\bbE L)(L-zI_{\ell})^{-1}\right] 
        \\ &+  \bbE\left[(\tilde{L}-\bbE L)(L-zI_{\ell})^{-1}(\tilde{L}-\bbE L)(L-zI_{\ell})^{-1}\right]
    \end{align*}
    where \( q_{j}=l_{-j}^{\top}R_{-j} l_{-j}\), \(\tilde{q}_{j}=\tilde{l}_{-j}^{\top}R_{-j} \tilde{l}_{-j}\), \(r_{j}=\tilde{l}_{-j}^{\top}R_{-j} l_{-j}\) and
    \[
        \Psi_{j} = 
        \begin{bmatrix}
              \substack{q_{j}-\tilde{q}_{j}  + 2\xi_{j}r_{j}^{2}} &  \substack{r_{j} \tilde{l}_{-j}^{\top}R_{-j}C_{j} - 2 \xi_{j}r_{j}^{2}l_{-j}^{\top}R_{-j}C_{j} + (\tilde{q}_{j}- q_{j})l_{-j}^{\top}R_{-j}C_{j}} \\
            \substack{-2\xi_{j}r_{j}C^{-1}_{j}\tilde{l}_{-j}  - C^{-1}_{j}l_{-j}} & \substack{C^{-1}_{j}(l_{-j}l_{-j}^{\top}-\tilde{l}_{-j}\tilde{l}_{-j}^{\top})R_{-j}C_{j} + 2\xi_{j}r_{j}C^{-1}_{j}\tilde{l}_{-j}l_{-j}^{\top}R_{-j}C_{j}}
        \end{bmatrix}.
    \]
    We will consider the upper blocks of \(\sum_{j=1}^{d}P_{j}\bbE[\xi_{j}\Psi_{j}]P_{j}\) separately.
    
    First, for the upper-left corner, the sum along with the permutation matrices \(P_{j}\) are simply tiling the diagonal. Furthermore, by \Cref{lemma:real/imag_singularity_bound}, both \(|\xi_{j}|\) and \(\|R_{j}\|\) are bounded by \((\Im[z])^{-1}\) for every \(j\in \{1,2,\ldots, d\}\). Then,
    \begin{align*}
        \left\|\sum_{j=1}^{d}P_{j}\bbE \begin{bmatrix}
            \xi_{j}( q_{j}-\tilde{q}_{j})  + 2\xi^{2}_{j}r_{j}^{2} & 0 \\
          0 & 0
       \end{bmatrix}P_{j}\right\| & \leq \max_{1\leq j\leq d}|\bbE [\xi_{j}( q_{j}-\tilde{q}_{j})  + 2\xi^{2}_{j}r_{j}^{2}]|
       \\ & \leq \frac{1}{\Im[z]}\bbE [| q-\tilde{ q}|]  + 2|\bbE[\xi^{2}_{1}l_{-1}^{\top}R_{-1}K R_{-1} l_{-1}]|
       \\ & \leq \frac{2}{\Im[z]}\bbE [| q-\bbE q|] 
        + \frac{2\bbE[\|L-\bbE L\|^{2}]\|K\|}{(\Im[z])^{4}}.
    \end{align*}
    Here, we introduced the correlation matrix \(K=\bbE[l_{-1}l_{-1}^{\top}]\). Using Jensen's inequality and Cauchy Schwarz, 
    \begin{align}
        \|K\| &\leq  \|\bbE [a_{1}a_{1}^{\top}]\|+\|\bbE [\ha_{1}a_{1}^{\top}]\|+\|\bbE [a_{1}\ha_{1}^{\top}]\|+\|\bbE [\ha_{1}\ha_{1}^{\top}]\| \notag
        \\ & = d^{-1}(\|\bbE [AA^{\top}]\|+\|\bbE [\hA A^{\top}]\|+\|\bbE [A\hA^{\top}]\|+\|\bbE [\hA\hA^{\top}]\|)  \notag
        \\ & \leq d^{-1}(\bbE [\|A\|^{2}]+2\sqrt{\bbE [\|\hA\|^{2}]\bbE [\|A\|^{2}]}+\bbE [\|\hA\|^{2}]) \lesssim d^{-1}. \label{eq:norm_K}
    \end{align}
    Here, we used the fact that \(\bbE [\|A\|^{2}]\) and \(\bbE [\|\hA\|^{2}]\) are bounded by~\Cref{lemma:bound_moments}. Additionally, by \Cref{lemma:quadratic_concentration}, it is clear that \(\bbE [| q-\bbE q|]\to 0\) as \(n\to \infty\).
    
    We now turn our attention to the upper-right \(1\times (\ell-1)\) corner of \(\sum_{j=1}^{d}P_{j}\bbE[\xi_{j}\Psi_{j}]P_{j}\). For every unit vector \(x\in \bbC^{\ell}\),
    \begin{align*}
        & \left\|\sum_{j=1}^{d}P_{j}\bbE\begin{bmatrix}
          0 &  \xi_{j}r_{j} \tilde{l}_{-j}^{\top}R_{-j}C_{j} - 2 \xi^{2}_{j}r_{j}^{2}l_{-j}^{\top}R_{-j}C_{j} + \xi_{j}(\tilde{q}_{j}- q_{j})l_{-j}^{\top}R_{-j}C_{j} \\
           0 & 0
       \end{bmatrix}P_{j}x \right\|_{2}
       \\ & = \left\|\bbE 
       \begin{pmatrix}
       \begin{pmatrix}
          0 &  \xi_{1}r_{1} \tilde{l}_{-1}^{\top}R_{-1}C_{1} - 2 \xi^{2}_{1}r_{1}^{2}l_{-1}^{\top}R_{-1}C_{1} + \xi_{1}(\tilde{ q}_{1}- q_{1})l_{-1}^{\top}R_{-1}C_{1} \\
       \end{pmatrix}P_{1}x \\
       \vdots \\
       \begin{pmatrix}
          0 &  \xi_{d}r_{d} \tilde{l}_{-d}^{\top}R_{-d}C_{d} - 2 \xi^{2}_{d}r_{d}^{2}l_{-d}^{\top}R_{-d}C_{d} + \xi_{d}(\tilde{ q}_{d}- q_{d})l_{-d}^{\top}R_{-d}C_{d} \\
       \end{pmatrix}P_{d}x
       \end{pmatrix}
       \right\|_{2}
        \\ & \leq \sqrt{\ell}\max_{1\leq j\leq d}\|\bbE[\xi_{j}r_{j} \tilde{l}_{-j}^{\top}R_{-j} - 2 \xi^{2}_{j}r_{j}^{2}l_{-j}^{\top}R_{-j} + \xi_{j}(\tilde{q}_{j}- q_{j})l_{-j}^{\top}R_{-j}]\|_{2}.
    \end{align*}
    On one hand,
    \[
        \bbE[\xi_{j}r_{j} \tilde{l}_{-j}^{\top}R_{-j} - 2 \xi^{2}_{j}r_{j}^{2}l_{-j}^{\top}R_{-j}] = \bbE[\xi_{j}l_{-j}^{\top}R_{-j}KR_{-j} - 2 \xi^{2}_{j}l_{-j}^{\top}R_{-j}KR_{-j}l_{-j}l_{-j}^{\top}R_{-j}]
    \]
    and, since \(|\xi_{j}|\leq (\Im[z])^{-1}\),
    \[
        \max_{1\leq j\leq d}\|\bbE[\xi_{j}r_{j} \tilde{l}_{-j}^{\top}R_{-j} - 2 \xi^{2}_{j}r_{j}^{2}l_{-j}^{\top}R_{-j}]\| \leq \frac{\bbE[\|l_{-1}\|]\|K\|}{(\Im[z])^{3}} + \frac{2\bbE[\|l_{-1}\|^{3}]\|K\|}{(\Im[z])^{5}}.
    \]
    Furthermore, by Cauchy-Schwarz for complex random variables,
    \begin{align*}
        \|\bbE[\xi_{j}(\tilde{q}_{j}- q_{j})l_{-j}^{\top}R_{-j}]\|_{2} & = \sup_{\|y\|\leq 1} |\bbE [\xi_{j}(\tilde{q}_{j}- q_{j})l_{-j}^{\top}R_{-j}y] |
        \\ & \leq (\Im[z])^{-1}\sup_{\|y\|\leq 1} \sqrt{\bbE[| q -\tilde{ q}|^{2}]\bbE[|l_{-j}^{\top}R_{-j}y|^{2}]}  
        \\ & = (\Im[z])^{-1}\sup_{\|y\|\leq 1} \sqrt{\bbE[| q -\tilde{ q}|^{2}]\bbE[y^{\ast}R_{-j}^{\ast}K R_{-j}y]}
        \\ & \leq  \frac{\sqrt{\bbE[| q -\tilde{ q}|^{2}]\|K\|}}{(\Im[z])^{2}}.
    \end{align*}
    Combining everything, we obtain that the upper-right \(1\times (\ell-1)\) corner of \(\sum_{j=1}^{d}P_{j}\bbE[\xi_{j}\Psi_{j}]P_{j}\) is bounded, in norm, by
    \[
        \frac{\sqrt{\ell}\bbE[\|l_{-1}\|]\|K\|}{(\Im[z])^{3}} + \frac{2\sqrt{\ell}\bbE[\|l_{-1}\|^{3}]\|K\|}{(\Im[z])^{5}}+ \frac{\sqrt{\ell \bbE[| q -\tilde{ q}|^{2}]\|K\|}}{(\Im[z])^{2}}.
    \]
    We conclude that this bound vanishes as \(n\) increases using~\eqref{eq:norm_K}, \(\bbE[\|l_{-1}\|]\leq \bbE[\|L-\bbE L\|]\) as well as \cref{lemma:quadratic_concentration,lemma:bound_moments}.
    
    We consider the two lower blocks together. For notational convenience, let
    \[
        \ubar{\Psi}_{j} = 
        \begin{bmatrix}
            0 &  0 \\
            -2\xi_{j}r_{j}C^{-1}_{j}\tilde{l}_{-j}  - C^{-1}_{j}l_{-j} & C^{-1}_{j}(l_{-j}l_{-j}^{\top}-\tilde{l}_{-j}\tilde{l}_{-j}^{\top})R_{-j}C_{j} + 2\xi_{j}r_{j}C^{-1}_{j}\tilde{l}_{-j}l_{-j}^{\top}R_{-j}C_{j}
        \end{bmatrix}
    \]
    for every \(j\in \{1,2,\ldots, d\}\). Since we expect \( q_{j}\) to concentrate around its mean, we write \(\xi_{j} = (1+z+ q_{j})^{-1} = (1+z+\bbE q_{j})^{-1} + \frac{\bbE q_{j}- q_{j}}{(1+z+\bbE q_{j})}\xi_{j}\) and
    \begin{align*}
        \sum_{j=1}^{d}P_{j}\bbE [\xi_{j}\ubar{\Psi}_{j}]P_{j} & = (1+z+\bbE q)^{-1}\sum_{j=1}^{d}P_{j}\bbE [\ubar{\Psi}_{j}]P_{j}
        \\ & - (1+z+\bbE q)^{-1} \sum_{j=1}^{d}P_{j}\bbE [( q_{j}-\bbE q_{j})\xi_{j}\ubar{\Psi}_{j}]P_{j}.
    \end{align*}
    Using independence of \(R_{-j}\), \(l_{-j}\) and \(\tilde{l}_{-j}\),
    \[
        \bbE \ubar{\Psi}_{j} = 
        \begin{bmatrix}
          0  &  0 \\
            -2\xi_{j}C^{-1}_{j}KR_{-j}l_{-j} &  2\xi_{j}C^{-1}_{j}K R_{-j}l_{-j}l_{-j}^{\top}R_{-j}C_{j}
        \end{bmatrix}.
    \]
    Using a similar argument as above,
    \[
        \left\|\sum_{j=1}^{d}P_{j}\bbE\begin{bmatrix}
            0 &  0 \\
            -2\xi_{j}C^{-1}_{j}KR_{-j}l_{-j} & 0
        \end{bmatrix}P_{j}\right\| \leq \frac{2\sqrt{\ell}\bbE[\|l_{-1}\|]\|K\|}{(\Im[z])^{2}} \xrightarrow[]{n\to\infty} 0.
    \]
    Moreover, further decomposing the lower-right corner,
    \begin{align*}
        &\sum_{j=1}^{d}P_{j}\bbE\begin{bmatrix}
            0 &  0 \\
            0 &  2\xi_{j}C^{-1}_{j}K R_{-j}l_{-j}l_{-j}^{\top}R_{-j}C_{j}
        \end{bmatrix}P_{j}
          \\ &\quad\quad\quad\quad\quad =  (1+z+\bbE q)^{-1}\sum_{j=1}^{d}P_{j}\bbE\begin{bmatrix}
            0 &  0 \\
            0 &  2C^{-1}_{j}K R_{-j}KR_{-j}C_{j}
        \end{bmatrix}P_{j}
        \\  &\quad\quad\quad\quad\quad  - (1+z+\bbE q)^{-1}\sum_{j=1}^{d}P_{j}\bbE\begin{bmatrix}
            0 &  0 \\
            0 &  2( q_{j}-\bbE  q_{j})\xi_{j}C^{-1}_{j}K R_{-j}l_{-j}l_{-j}^{\top}R_{-j}C_{j}
        \end{bmatrix}P_{j}
    \end{align*}
    with \(|(1+z+\bbE q)^{-1}|\leq (\Im[z])^{-1}\),
    \begin{align*}
        \left\| \sum_{j=1}^{d}P_{j}\bbE\begin{bmatrix}
            0 &  0 \\
            0 &  2C^{-1}_{j}K R_{-j}KR_{-j}C_{j}
        \end{bmatrix}P_{j} \right\| \leq \frac{2d\|K\|^{2}}{(\Im[z])^{2}}
    \end{align*}
    and 
    \begin{align*}
        \left\|\sum_{j=1}^{d}P_{j}\bbE\begin{bmatrix}
            0 &  0 \\
            0 &  \substack{2( q_{j}-\bbE  q_{j})\xi_{j}C^{-1}_{j}K R_{-j}l_{-j}l_{-j}^{\top}R_{-j}C_{j}}
        \end{bmatrix}P_{j}\right\| 
        \leq \frac{2d\|K\|\sqrt{\bbE[| q-\bbE q|^{2}]\bbE[\|l_{-1}\|^{4}]}}{(\Im[z])^{3}}.
    \end{align*}
    In particular, \(\|(1+z+\bbE q)^{-1}\sum_{j=1}^{d}P_{j}\bbE [\ubar{\Psi}_{j}]P_{j}\| \xrightarrow[]{n\to\infty} 0\). It only remains to show that \(\|(1+z+\bbE q)^{-1} \sum_{j=1}^{d}P_{j}\bbE [( q_{j}-\bbE q_{j})\xi_{j}\ubar{\Psi}_{j}]P_{j}\|\) vanishes. To this end, we undo the decomposition and notice that
    \begin{align*}
        \sum_{j=1}^{d}P_{j}\bbE [( q_{j}-\bbE q_{j})\xi_{j}\ubar{\Psi}_{j}]P_{j} & = \bbE\left[(\ubar{L}-\bbE \ubar{L})\Omega (L-zI_{\ell})^{-1}\right] 
        \\ & + \bbE\left[(\tilde{\ubar{L}}-\bbE \ubar{L})\Omega(L-zI_{\ell})^{-1}(\tilde{L}-\bbE L)(L-zI_{\ell})^{-1}\right] 
    \end{align*}
    where
    \[
        \ubar{L} =  
        \begin{bmatrix}
            \delta I_{\ntrain}        & A                       & 0              & 0     \\
            0                &  -I_{d\times d}          & 0             & 0 \\
            0       &  0          & 0             &  -I_{\ntest} \\
            0        & \hA           & -I_{\ntest}                      & 0
        \end{bmatrix}
    \]
    and
    \[
        \Omega = \diag\{ 0_{n_{\mathrm{train}}\times n_{\mathrm{train}}},\diag\{ q_{j}-\bbE  q_{j}\}_{j=1}^{d},0_{2n_{\mathrm{test}}\times 2n_{\mathrm{test}}}\}.
    \]
    Using the bound \(\|(L-z I_{\ell})^{-1}\|\leq (\Im[z])^{-1}\), it follows from Jensen's and Cauchy-Schwarz inequalities that
    \[
        \left\| \bbE\left[(\ubar{L}-\bbE \ubar{L})\Omega (L-zI_{\ell})^{-1}\right] \right\| \leq \frac{\sqrt{\bbE[\|L-\bbE L\|^{2}] \bbE[\max_{1\leq j\leq d} | q_{j}-\bbE  q_{j}|^{2}]}}{\Im[z]}
    \]
    and
    \begin{multline*}
        \left\|\bbE\left[(\tilde{\ubar{L}}-\bbE \ubar{L})\Omega(L-zI_{\ell})^{-1}(\tilde{L}-\bbE L)(L-zI_{\ell})^{-1}\right]  \right\| 
        \\\leq \frac{\sqrt{\bbE[\|L-\bbE L\|^{4}]\bbE[\max_{1\leq j\leq d} | q_{j}-\bbE  q_{j}|^{2}]}}{(\Im[z])^{2}}.
    \end{multline*}
    
    This term gives us the bottleneck conditions on the norm of the matrix \(L-\bbE L\) and the concentration of \( q\) around its mean. By \cref{lemma:quadratic_concentration,lemma:bound_moments}, both of the RHS bounds vanish as \(n\) diverges to infinity.
\end{proof}
    
\end{appendix}

\begin{acks}[Acknowledgments]
    The authors would like to thank David Renfrew for helpful conversations relating to the Dyson equation. The authors would also like to thank Konstantinos Christopher Tsiolis for his help at multiple stages of the project.
\end{acks}
\begin{funding}
    The first author was supported in part by the Canada CIFAR AI Chair Program (held at Mila by Courtney Paquette). The second author was supported by NSERC Discovery Grant RGPIN-2020-04974.
\end{funding}



\bibliographystyle{imsart-number} 
\bibliography{bibliography}       


\end{document}